\documentclass[reqno]{amsart}

\usepackage[T1]{fontenc}

\usepackage{hyperref}
\usepackage{tabularx} 

\usepackage{varwidth}
\usepackage{enumitem}
\usepackage{subcaption}  

\usepackage{esint}
\usepackage{amssymb,centernot}
\usepackage{mathtools} 
\usepackage{tikz}
\usetikzlibrary{arrows.meta}

\usepackage{pgfplots}
\usepackage{amsmath,amsthm}

\pgfplotsset{compat=1.18}

\usepackage{graphicx}
\usepackage{amsrefs,enumitem,pgfkeys}
\usepackage{soul}
\usepackage{dirtytalk}
\usepackage{mathrsfs}
\usepackage{MnSymbol,wasysym}
\usepackage{fancyhdr}
\usepackage{time}

\usepackage[margin=1in]{geometry}
 
\newcommand\restr[2]{{
  \left.\kern-\nulldelimiterspace 
  #1 
  \vphantom{\big|} 
  \right|_{#2} 
  }}

\DeclareMathOperator*{\essinf}{ess\,inf}
\DeclareMathOperator*{\esssup}{ess\,sup}

\newcommand{\dimrlow}{\underline{\dim}_\textup{r} }
\newcommand{\dimrup}{\overline{\dim}_\textup{r} }
\newcommand{\dimr}{{\dim_\textup{r}} }
\newcommand{\dimlow}{\underline{\dim}_{\textup{loc} }}
\newcommand{\dimup}{\overline{\dim}_{\textup{loc}} }

\usepackage{comment}
\newtheorem{theorem}{Theorem}
\newtheorem{lemma}[theorem]{Lemma}
\newtheorem{corollary}[theorem]{Corollary}

\theoremstyle{definition}

\newtheorem{remark}[theorem]{Remark}

\newtheorem{example}[theorem]{Example}
\newtheorem{definition}[theorem]{Definition}

\newtheorem{problem}[theorem]{Problem}

\newtheorem{mylettertheorem}{Theorem}[section]

\renewcommand{\Subset}{\subset\joinrel\subset}

\newcommand{\spt}[1]{{\textup{Spt} \ #1}}

\newcommand{\cref}[1]{Corollary \ref{c.#1}}

\numberwithin{theorem}{section}
\numberwithin{equation}{section}

\newcommand{\trace}[1]{{\textup{Tr}\lb #1 \rb}}

\newcommand{\diam}{\textup{diam} }

\newcommand{\Na}{\mathcal{N}}
\newcommand{\N}{\mathbb{N}}
\newcommand{\Ept}{\mathbb{E}}
\newcommand{\Z}{\mathbb{Z}}
\newcommand{\R}{\mathbb{R}}

\newcommand{\T}{\mathbb{T}}

\usepackage{mathtools}

\newcommand{\pt}{\partial}

\newcommand{\bca}{\begin{cases}}
\newcommand{\eca}{\end{cases}}

\newcommand{\lb}{\left(}

\newcommand{\rb}{\right)}
\newcommand{\lmb}{\left[}
\newcommand{\rmb}{\right]}
\newcommand{\lma}{\left\{}
\newcommand{\rma}{\right\}}
\newcommand{\Ha}{{\mathcal{H}}}

\newcommand{\Da}{{\mathscr{D}}}
\newcommand{\Fa}{{\mathcal{F}}}
\newcommand{\Ra}{{\mathcal{R}}}
\newcommand{\Ea}{{\mathcal{E}}}

\newcommand{\Ta}{{\mathcal{T}}}

\newcommand{\csubset}{\subset\joinrel\subset}

\newcommand{\Lme}{\mathcal{L}}

\newcommand{\gd}{\nabla}
\newcommand{\rta}{\rightarrow}
\newcommand{\lw}{\left|}
\newcommand{\rw}{\right|}
\newcommand{\be}{\begin{equation}}
\newcommand{\ee}{\end{equation}}

\newcommand{\bt}{\begin{thm}}
\newcommand{\et}{\end{thm}}
\newcommand{\bc}{\begin{cor}}
\newcommand{\ec}{\end{cor}}
\newcommand{\bl}{\begin{lem}}
\newcommand{\el}{\end{lem}}
\newcommand{\norm}[1]{\left\lVert#1\right\rVert}
\newcommand{\avg}[1]{\langle#1\rangle}
\newcommand{\normm}[1]{{\left\vert\kern-0.25ex\left\vert\kern-0.25ex\left\vert #1 
    \right\vert\kern-0.25ex\right\vert\kern-0.25ex\right\vert}}

\newcommand{\dist}{\operatorname{dist}}

\newcommand{\weakcv}{\rightharpoonup}

\def\Xint#1{\mathchoice
{\XXint\displaystyle\textstyle{#1}}%
{\XXint\textstyle\scriptstyle{#1}}%
{\XXint\scriptstyle\scriptscriptstyle{#1}}%
{\XXint\scriptscriptstyle\scriptscriptstyle{#1}}%
\!\int}
\def\XXint#1#2#3{{\setbox0=\hbox{$#1{#2#3}{\int}$ }
\vcenter{\hbox{$#2#3$ }}\kern-.6\wd0}}

\def\dashint{\Xint-}

\usetikzlibrary{decorations.pathmorphing, arrows.meta, patterns}

\usepackage{marvosym}

\def\Xint#1{\mathchoice
{\XXint\displaystyle\textstyle{#1}}%
{\XXint\textstyle\scriptstyle{#1}}%
{\XXint\scriptstyle\scriptscriptstyle{#1}}%
{\XXint\scriptscriptstyle\scriptscriptstyle{#1}}%
\!\int}
\def\XXint#1#2#3{{\setbox0=\hbox{$#1{#2#3}{\int}$ }
\vcenter{\hbox{$#2#3$ }}\kern-.6\wd0}}

\def\dashint{\Xint-}

\newcommand{\ep}{\varepsilon}

\begin{document}

\title[Singular Wiener Bound]{On the attainability of the singular Wiener bound}

\author{Zhonggan Huang}
\address{Department of Mathematics, University of Utah, Salt Lake City, UT 84112}
\email{zhonggan@math.utah.edu}

\date{\today}
\keywords{Optimal design, conductance maximization, leaf venation pattern, reticulation, geometric measure theory, stationary varifolds, homogenization, homotopy}
\subjclass{49Q10, 49Q20, 35B27, 92B99}

\begin{abstract}
    We characterize the lower and upper attainability of the Wiener bound (also known as the conductive analogue of the Voigt-Reuss-Hill bound in {elasticity theory}) for singularly distributed conductive material mixtures. For the lower attainability we consider mixtures in which high-conductance materials support on sets having finite one-dimensional Hausdorff measures. We show that, under a mild coercivity condition, the kernel of the effective tensor of the mixture is equal to the orthogonal complement of the homotopy classes of closed paths in the supporting set. This shows that a periodic planar network has positive definite effective tensor, i.e., it is resilient to fluctuations, if and only if the network is reticulate. We prove a geometric characterization of the upper attainability by applying a transformation from varifolds to matrix-valued measures. We show that this transformation leads to an equivalence between two distinct notions from material science and geometric measure theory respectively: conductance maximality and area criticality. Based on this relation we show a pointwise dimension bound for mixtures that attain the upper Wiener bound by applying a fractional version of the monotonicity formula for stationary varifolds. This dimension bound illustrates how the maximality condition constrains the local anisotropy and the local distribution of conductance magnitudes. Both the lower and upper attainability results have potential novel applications in modeling leaf venation patterns.
\end{abstract}

\maketitle
\tableofcontents

\section{Introduction}

 What are the appropriate geometric objects to describe leaf vein patterns? While they are common planar networks, it is difficult to determine the \say{energies} that effectively characterize the geometry. In this paper, we approach this issue through a variational problem on the hydraulic conductivity properties of leaves. More specifically, we study the lower and upper attainability of a singular version of the Wiener bound \cite{Wiener1912} for mixtures of conductive materials. {This bound is also widely known as the conductive analogue Voigt-Reuss-Hill bound \cites{voigt,reuss,Hill_1952} in elasticity literature.} In this new singular setting, high-conductance materials concentrate on lower dimensional sets. Here by ``attainability'' we mean the characterization of the mixtures that attain or do not attain the lower or the upper bound.
 
In particular, our work on the lower attainability provides a potential mathematical support to an existing biological theory on the reticulation phenomenon in higher-order leaf vein patterns, which states that random hydraulic fluctuations imply reticulation \cites{Corson2010,Katifori2010}.

Let us recall the original version of the Wiener bound. By a standard homogenization argument \cites{Jikov1994,Tartar2010,Milton_2002} (also see Appendix \ref{appendix.periodichomog}), a material mixture is represented by a positive definite matrix field $A(x)$ on the torus $\T^n$ (or equivalently $A(x)$ is a periodic matrix field on $\R^n$). Such a matrix field often satisfies the following {ellipticity} condition
\be
\label{eq.ellipticityconditionforA}
\lambda^{-1}\le A(x) \le \lambda, \textup{ for all }x\in\T^n,
\ee
for some positive constant $\lambda$. The \emph{effective conductance tensor} of the mixture $A(x)$ is defined as the positive definite matrix $Q(A)$ that satisfies the following variational problem
\be
\label{eq.classicaleffectivetensor}
p\cdot Q(A) p := \inf_{\varphi\in C^\infty(\T^n)} \int_{\T^n} (\gd \varphi(x)+p)\cdot A(x) (\gd\varphi(x)+p) d\Lme^n,
\ee
where $\Lme^n$ is the Lebesgue measure on $\T^n$ and $p\in \R^n$ is an arbitrary vector that represents the effective pressure gradient from the exterior that applies to the mixture $A(x)$. 

The \emph{classical Wiener bound} provides an explicit range of the effective tensor $Q(A)$. 
 \be
\label{eq.classicalWIenrbound}
\lb\int_{\T^n} A^{-1}(x) d\Lme^n \rb^{-1} \le Q(A) \le \int_{\T^n} A(x) d\Lme^n.
\ee
The upper bound is attained if and only if $\textup{div}\, A =0$, and the lower bound is attained if and only if $\textup{curl}\, A^{-1} =0$. See \cite{Jikov1994}*{Section 1.6} for more details and literature.

Although the upper and lower attainability for the classical Wiener bound \eqref{eq.classicalWIenrbound} can be characterized explicitly by the equations $\textup{div}\, A =0$ and $\textup{curl}\, A^{-1} =0$ respectively, the problem becomes much harder as one looks at $A$'s that concentrate on lower dimensional sets.

Indeed, to model leaf vein patterns, it is natural to consider the matrix fields $A$ on $\T^2$ that have high magnitude near a 1-dimensional network $\Gamma\subset\T^2$. To give a heuristic discussion we consider the $2\times 2$ matrix fields $A_\delta$ of the following form
\[
A_\delta(x):=\bca
\frac{1}{\delta} \, I_{2\times 2} & \textup{ for }\dist(x,\Gamma)<\delta/2\\
\delta\, I_{2\times 2}  & \textup{ elsewhere,}
\eca
\]
where $\delta>0$ is a small parameter. Note that $A_\delta$ do not satisfy the ellipticity condition \eqref{eq.ellipticityconditionforA} uniformly. To sharply characterize the conductivity property of $\Gamma$ we send the small parameter $\delta\rta0^+$. The limit of $A_\delta$ turns out to be a matrix-valued measure $A_0$ of the form
\be\label{eq.exampleinintro1}
dA_0(x)= I_{2\times 2}\, d\restr{\Ha^1}{\Gamma}(x),
\ee
where $\restr{\Ha^1}{\Gamma}$ is the one-dimensional Hausdorff measure restricted to the 1-D set $\Gamma$. Note that at least heuristically the classical Wiener bound \eqref{eq.classicalWIenrbound} degenerates to
\be\label{eq.singularforveinswiener1}
0 \le Q(A_0) \le \int_{\T^2} dA_0(x),
\ee
where the lower bound becomes zero because $A_0^{-1}$ is infinite for Lebesgue almost every point.

Now the difficulty immediately occurs as one looks back on the attainability equations for the upper and lower bound in \eqref{eq.singularforveinswiener1}
\[
\textup{div}\, A_0 =0\; \textup{ and }\;\textup{curl}\, A_0^{-1} =0.
\]
Notice that $A_0$ is a singular measure supported on a lower dimensional set $\Gamma$. The first equation $\textup{div}\, A_0 =0$ can still be interpreted in the sense of distributional derivatives. Unfortunately the second equation $\textup{curl}\, A_0^{-1} =0$ is ill-posed even in the distributional sense. This is because $A_0$ is a singular measure and $A_0^{-1}$ is supposed to be infinity for Lebesgue almost every point. 

In this paper, we develop a theory to deal with both the lower and upper attainability problem in the singular case \eqref{eq.singularforveinswiener1}, by using techniques from geometric measure theory, calculus of variations, homotopy groups and etc. The main results are the two theorems: 
\[
\textup{Theorem \ref{t.forma1} on the lower attainability and Theorem \ref{t.forma2} on the upper attainability.}
\]

To make a more precise presentation, we will model singular matrix fields like $A_0$ in \eqref{eq.exampleinintro1} by positive semi-definite matrix-valued Radon measures, which we call \emph{medium} for convenience (See Definition \ref{def.singulartensorcomposite}). More precisely we call $\theta$ a medium if it is a matrix-valued measure on $\T^n$ that takes the form
\be\label{eq.mediumintro}
d\theta=\sigma\, d\norm{\theta}
\ee
for some Radon measure $\norm{\theta}$ and some positive semi-definite matrix field $\sigma$ such that the trace $\trace{\sigma(x)}=1$ for $\norm{\theta}$-almost all $x\in \T^n$. Here $\sigma$ represents the local anisotropy and $\norm{\theta}$ represents the local magnitude of conductance. Note that the effective tensor $Q(\theta)$ can be defined similarly as in \eqref{eq.classicaleffectivetensor} by replacing $A(x)$ by $\sigma(x)$ and $d\Lme^n$ by $d\norm{\theta}$. In this general setting, the singular Wiener bound \eqref{eq.singularforveinswiener1} becomes
\be\label{eq.singularforveinswiener}
0 \le Q(\theta) \le \theta(\T^n).
\ee

Our first main result Theorem \ref{t.forma1}, focuses on the lower attainability of isotropic 1-D media $\theta$ on $\T^n$. By 1-D we mean that the medium $\theta$ takes the form 
\be\label{eq.networkformintro}
d\theta = I_{n\times n}\, dw = \frac{1}{n} I_{n\times n} \, d (n w)
\ee
with the Radon measure $w$ satisfying
\be\label{eq.1dregularintro}
0<\limsup_{r\rta0^+}\frac{w(B_r(x))}{r}<\infty
\ee
for $w$-almost every $x\in \T^n$. We also require the following \emph{coercivity} condition
\be
\label{eq.coerciveintro}
\limsup_{r\rta0^+}\frac{w(B_r(x))}{r}>c>0
\ee
for some positive constant $c$ and $\Ha^1$-almost every $x$ in the support $ \spt{w}=\spt{\theta}$. An immediate consequence of the coercivity condition is that $\Ha^1(\spt{\theta})<\infty$ (See Section \ref{subsubsection.density}). Also note that whenever $E$ is a closed subset of $\T^n$ such that $\Ha^1(E)<\infty$, the restriction $w=\restr{\Ha^1}{E}$ satisfies \eqref{eq.1dregularintro} and \eqref{eq.coerciveintro}.

For such 1-D media $\theta$, we characterize the relation between the kernel of the effective tensor $Q(\theta)$ and the collection of \emph{homotopy classes} of closed paths in the support $\spt{\theta}\subset\T^n$ of $\theta$. As we shall see later, the kernel of the effective tensor represents the directions in which there is zero conductance in large scales. The interest of this characterization lies in the interaction between the large scale behavior in conductance and the microscale topology of the support of the medium.

The notion of homotopy classes is a standard tool to classify closed paths in a topological space (See Section \ref{subsection.torusandtopology} for the preliminaries and references). {However, to describe the corresponding topological properties of the support $\spt{\theta}$ for different $\ker Q(\theta)$, we need to introduce several new notions that apply the homotopy classes in a slightly unusual way. First of all, we need to define a notion that characterizes the abundance of loops in a given subset $E$ of $\T^n$ in terms of the closed paths in $E$. Specifically for each subset $E\subset\T^n$, we define $H_E\subset \Z^n$ the collection of all the homotopy classes of closed paths in $\T^n$ that have image contained in $E$, see Definition \ref{def.definehomotopygroupofsets} for a precise definition.}   {Next, we call $E$ \emph{loopy} if $H_E$ spans the whole $\R^n$ (not $\Z^n$; see Remark \ref{r.onthedefinitonofloopy}), and call $E$ \emph{reticulate} if there is a loopy connected component in $E$, see Definition \ref{def.topologynetwork} for more details.} 

We remark here that loopiness is equivalent to reticulation in dimension $n=2$, but they are not necessarily the same in higher dimensions. Please see Remark \ref{r.equiloopyandreticulate} and Lemma \ref{l.characterizeloopysetn=2} for more details.

\begin{mylettertheorem}[See Theorem \ref{t.Qtohomotopy}]
    \label{t.forma1}\emph{
Suppose a medium $\theta$ takes the form \eqref{eq.networkformintro} with $w$ satisfying \eqref{eq.1dregularintro} and the coercivity condition \eqref{eq.coerciveintro}. Then the following identity holds
\[
\ker Q(\theta) = H_{\Gamma}^{\perp},
\]
where $H_{\Gamma}^{\perp}$ is the orthogonal space of the homotopy classes of closed paths in the support $\Gamma=\spt{w}=\spt{\theta}$. In particular, $Q(\theta)$ is positive definite if and only if $\spt{\theta}$ is loopy. }
\end{mylettertheorem}

As loopiness is equivalent to reticulation in dimension $n=2$, the theorem immediately implies the following corollary. 

\begin{corollary}[See Theorem \ref{t.characterizationoflowerboundn=2}]\label{cor.positivedefiniteequalreticulate}
Under the same assumptions in Theorem \ref{t.forma1}, and in dimension $n=2$, the effective tensor $Q(\theta)$ is positive definite if and only if $\spt{\theta}$ is reticulate.
\end{corollary}
We refer to Section \ref{subsubsection.outlineoftheoremB} for an outline of the proof of both Theorem \ref{t.forma1} and the above corollary. 

Both the positive definiteness of $Q(\theta)$ and the term ``reticulate'' are closely related to the modeling of leaf vein patterns. To give an intuition on the term ``reticulate'', we refer to Figure \ref{realloopy} for a graphical comparison among a real leaf pattern, a reticulate network and two non-reticulate networks.

\begin{figure}[hbpt]
\centering
\begin{subfigure}[t]{.25\textwidth}
  \centering
  \includegraphics[width=.7\linewidth]{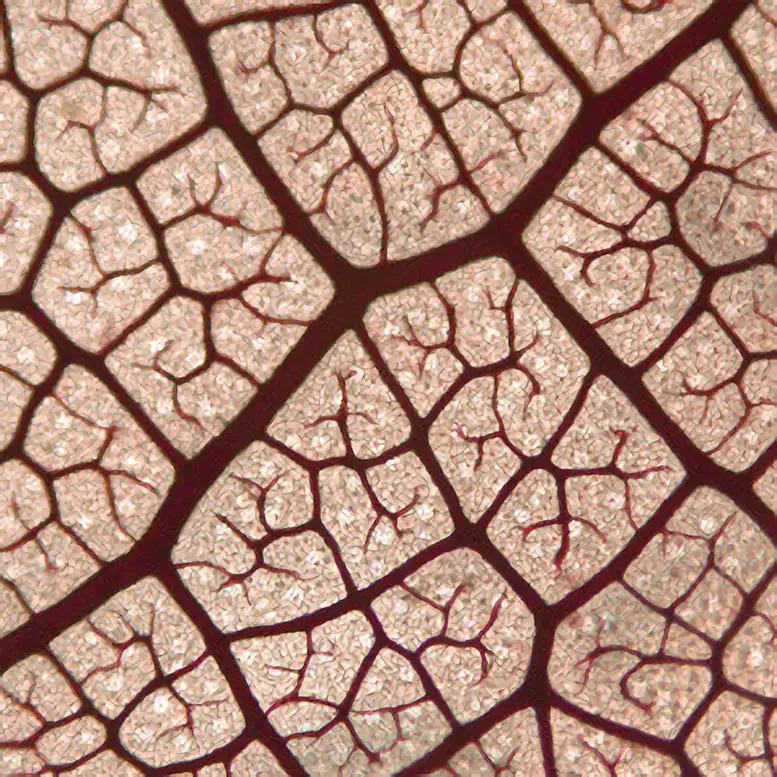}
  \label{truleaf}
\end{subfigure}%
\begin{subfigure}[t]{.25\textwidth}
  \centering
  \includegraphics[width=.7\linewidth]{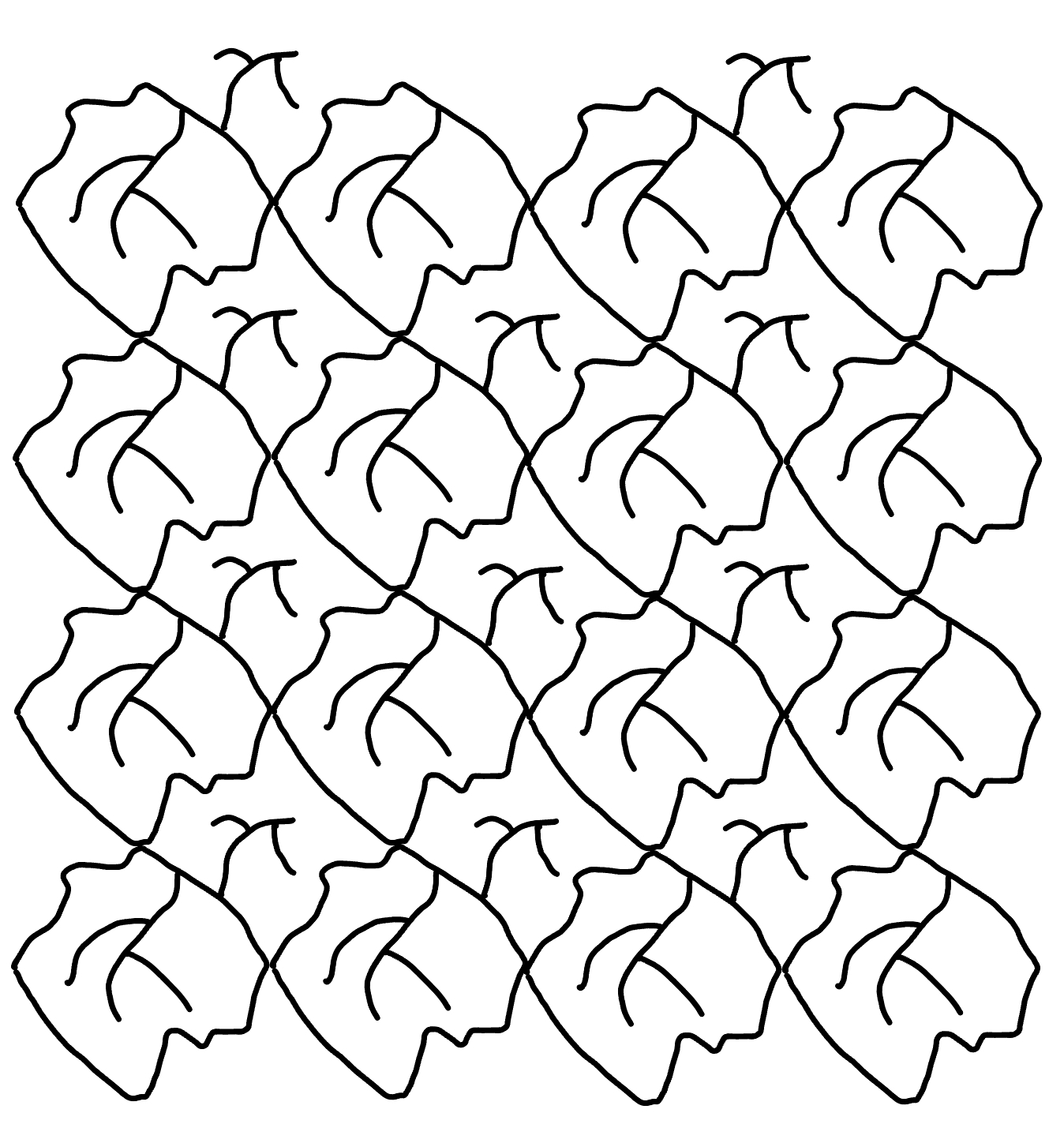}
\end{subfigure}%
\begin{subfigure}[t]{.25\textwidth}
  \centering
  \includegraphics[width=.7\linewidth]{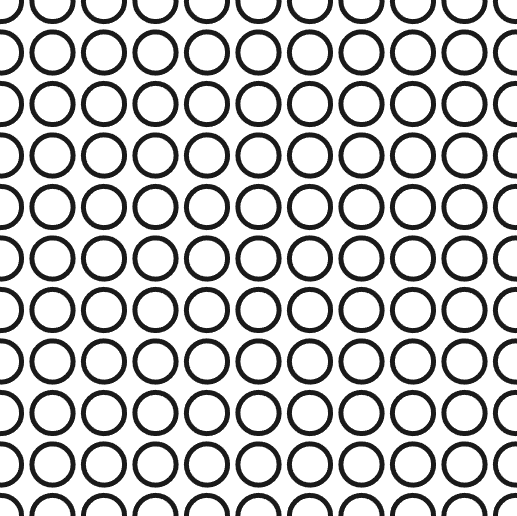}
\end{subfigure}%
\begin{subfigure}[t]{.25\textwidth}
  \centering
  \includegraphics[width=.7\linewidth]{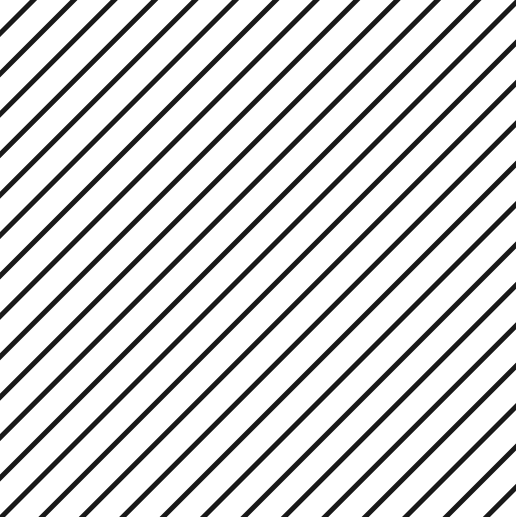}
\end{subfigure}%
\caption{Left: a picture of higher-order veins in a tropical forest tree, \emph{Ampelocera ruizii}. The picture is reproduced from \cite{sack2013leaf}, with permission from Wiley; Middle left: a reticulate $\Z^2$-periodic network that has full rank effective tensor; Middle right: a non-reticulate $\Z^2$-periodic network that has zero effective tensor; Right: a non-reticulate $\Z^2$-periodic network that has rank 1 effective tensor.}
\label{realloopy}
\end{figure}

The positive-definiteness of $Q(\theta)$ is closely related to the resilience of a conductive network to random fluctuations. According to the biological literature \cites{Corson2010,Katifori2010} (also see Section \ref{subsubsection.randomfluctuationsandreticulation} below), a leaf will experience random fluctuations of hydraulic pressure due to various factors like sunlight, humidity, insect bites and etc. In terms of mathematics, the vector $p$ in the definition of effective tensor \eqref{eq.classicaleffectivetensor} is a random vector with no preference in directions. It is then natural to define a ``resilient network'' to be those having positive effective conductance in all directions. Equivalently, this is saying that the effective tensor of this network is positive definite. See Definition \ref{def.differentmedia} for a precise definition of resilience in terms of media. 

We remark here that there is a gap between the original setting in \cites{Corson2010,Katifori2010} and the explanation above about the term \say{resilience}. We will make a clearer connection between these ideas in Appendix \ref{appendix.formalderivation} by presenting a formal derivation of the positivity of $Q(\theta)$ from the minimization of the expected total dissipation of a piece of leaf under random fluctuations.

Our second main result concerns about the upper attainability of the singular Wiener bound \eqref{eq.singularforveinswiener} for general media. We use a transformation, first introduced in \cite{ambrosio1997}*{Remark 3.2} and also see \cite{mease4}, to show that all media can be equivalently represented by varifolds. Varifolds are measure-theoretic extensions of surfaces with nice compactness properties, which has been widely used in the field of minimal surfaces (See Section \ref{subsubsection.varifolds} for more details and references). 

By using the notion of varifolds, we have the following formal identity
\be
\label{eq.formalequivalenceintro}
\textup{\emph{Conductance Maximality} } = \textup{ \emph{Area Criticality}}.
\ee
On the left of the above identity, we mean media that attain the singular upper Wiener bound \eqref{eq.singularforveinswiener}. On the right we mean stationary varifolds, which is a weak notion of the critical points to the area functional of surfaces \cite{Allard1972}. More precisely, we have the second main theorem.

\begin{mylettertheorem}[See Theorem \ref{t.equivalence}]\label{t.forma2}
\emph{There is a continuous surjective map $\Ta$ from the space of all varifolds on $\T^n$ to the space of media. The following statements are equivalent for a fixed medium $\theta$:
\begin{enumerate}[label=(\alph*)]
    \item \label{condition.(a)intro} The medium $\theta$ attains the upper Wiener bound \eqref{eq.singularforveinswiener}.
    \item \label{condition.(a2)intro} The medium $\theta$ satisfies $\gd\cdot\theta=0$ in the distributional sense (See Lemma \ref{l.equationformaximalmedia}).
    \item \label{condition.(b)intro} All varifold realizations $\mu\in \Ta^{-1}(\theta)$ are stationary.
    \item \label{condition.(c)intro} There exists a stationary varifold $\mu\in \Ta^{-1}(\theta)$.
\end{enumerate}}
\end{mylettertheorem}

To give a quick explanation of Theorem \ref{t.forma2} we apply the celebrated Allard-Almgren characterization of 1-D stationary varifolds that have positive densities (See \cites{AllardAlmgren1976}; also see Theorem \ref{t.allardalmgren} for a presentation of this theorem), and obtain the following corollary. We refer to Figure \ref{balancedpattern} for a graphical illustration of 1-D stationary varifolds.

\begin{corollary}\label{cor.applyallardalmgren}
Let $\theta$ be a medium that attains the upper Wiener bound \eqref{eq.singularforveinswiener} and satisfies the 1-dimensional lower density bound
\be\label{eq.1dlowerdensityboyundfortensorintro}
\liminf_{r\rta0^+} \frac{\norm{\theta}(B_r(x))}{2r} \ge \delta >0
\ee
for $\norm{\theta}$-almost every $x\in \T^n$. Then $\Ta^{-1}(\theta)=\{\mu\}$ contains a unique 1-rectifiable stationary varifold $\mu$ that satisfies the following properties (we denote $\norm{\mu}=\norm{\theta}$ the area/mass distribution)
\begin{enumerate}
\item the support $\spt{\norm{\mu}}=\spt{\theta}$ is, up to an $\Ha^1$-null closed set $S$, a countable union of straight line segments, which are open relative to $\spt{\norm{\mu}}$;
\item on each line segment there is a constant $c\ge\delta$ such that the density $$\xi(x):=\lim_{r\rta0^+}\frac{\norm{\mu}(B_r(x))}{2r} \equiv c, $$
for $x$ on the line segment;
\item at every $x\in S$, there exists a unique stationary tangent cone consisting of finitely many half lines with constant densities. See Figure \ref{fig:localcone} for the illustration of a stationary tangent cone. In such a cone, the tangent vectors $T_j$ start from the origin and the densities $w_j$ satisfy
\be\label{eq.mcone}
\sum_{j=1}^m w_j T_j =0;
\ee

\item if the density $\xi$ is discretely valued, then the number of line segments is finite.
\end{enumerate}
\end{corollary}

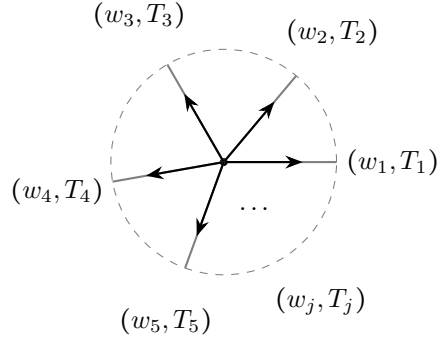
\begin{figure}[htbp]
    \centering
\begin{tikzpicture}[>=Stealth, scale = 1.5]
  
    \draw[gray, dashed] (0,0) circle (1);

    \fill (0,0) circle (1pt);
    
    \foreach \angle [count=\j from 1] in {0,50,120,190,250} {
        \draw[gray, thick] (0,0) -- (\angle:1);
        
        \draw[->, black, thick] (0,0) -- (\angle:0.7);
        
        \node at (\angle:1.5) {$(w_{\j},T_{\j})$};
    }

     \foreach \angle in {305} {
        
        \node at (\angle:0.5) {$\dots$};

        \node at (\angle:1.5) {$(w_j,T_j)$};
    }
    
\end{tikzpicture}
    \caption{For some integer $m\ge 2$, there are $m$ half lines with the unit outward tangent vectors denoted by $T_j$ and densities $w_j>0$ for $1\le j \le m$.  In a stationary tangent cone the vectors $T_j$ and the weights $w_j$ satisfy \eqref{eq.mcone}.
}
    \label{fig:localcone}
\end{figure}

In 1977, Schulgasser \cites{Schulgasser_1977} proved that the upper Wiener bound can be attained by sequentially laminated isotropic materials, in which the notion of attainability is slight weaker than what we mean here. Similar attainability result of maximizing media to Theorem \ref{t.forma2} was first addressed in the case of networks in 2003 by Zhikov and Pastukhova \cites{ZhikovYosifian2014,ZhikovPastukhova2003}.  In 2020 the author \cite{Huang2021} also addressed the network case. It was also discussed in \cites{BOURDIN20081043} that the maximality is attained by the superpositions of thin laminates. We refer to Section \ref{section.maximalmedium} for a more precise discussion on Theorem \ref{t.forma2}.

The fact that the conductance maximality is equivalent to the divergence free equation $\gd\cdot\theta=0$ is well-known \cites{matsci1,meas1,meas2,meas5} in the composite-theory community (also see Section \ref{subsection.elequation} for a rigorous derivation). On the other hand, it was discovered in \cites{ambrosio1997,mease4} that area criticality implies the divergence free equation $\gd \cdot\theta=0$. This implicitly proves that area criticality implies conductance maximality, although the transformation they proposed from varifolds to matrix-valued measures was initially designed to study surface energies.

Based on the identity in Theorem \ref{t.forma2}, we establish a pointwise dimension bound similar to \cite{meas3}*{Corollary 1.4} and \cite{mease6} by introducing a \emph{fractional} version of the monotonicity formula of stationary varifolds. See Section \ref{subsubsection.fractionalmono} for a more rigorous discussion. Also see Theorem \ref{t.realizabledimensionandlocaldimension} for more details.

\begin{figure}[hbpt]
\centering
\begin{subfigure}[t]{.45\textwidth}
  \centering
  \includegraphics[width=.6\linewidth]{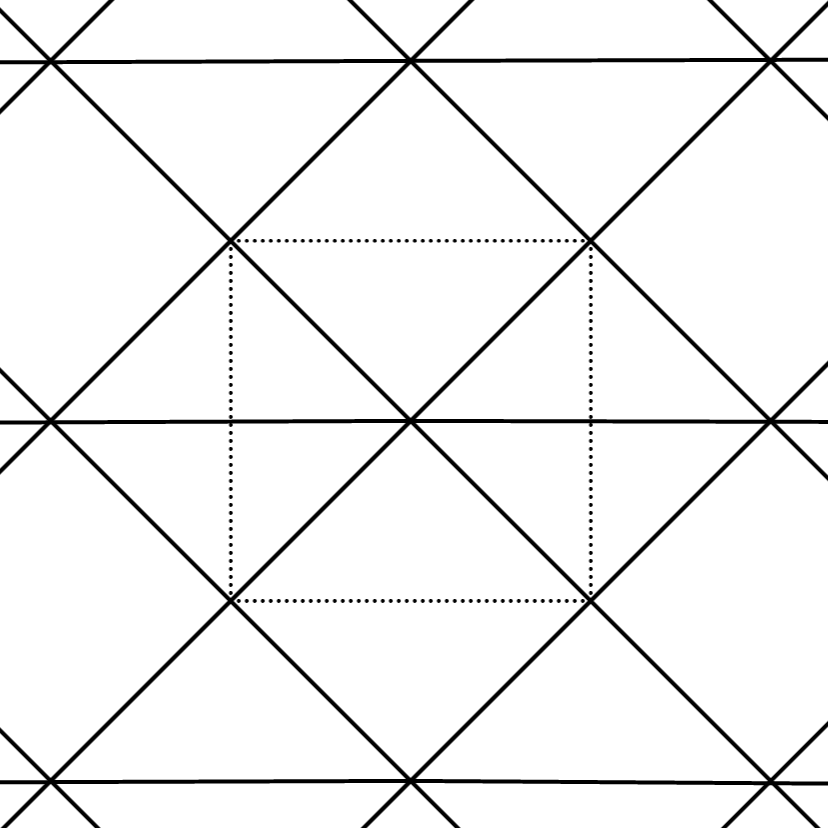}
\end{subfigure}%
\hspace{0.2cm}
\begin{subfigure}[t]{.45\textwidth}
  \centering
  \includegraphics[width=.6\linewidth]{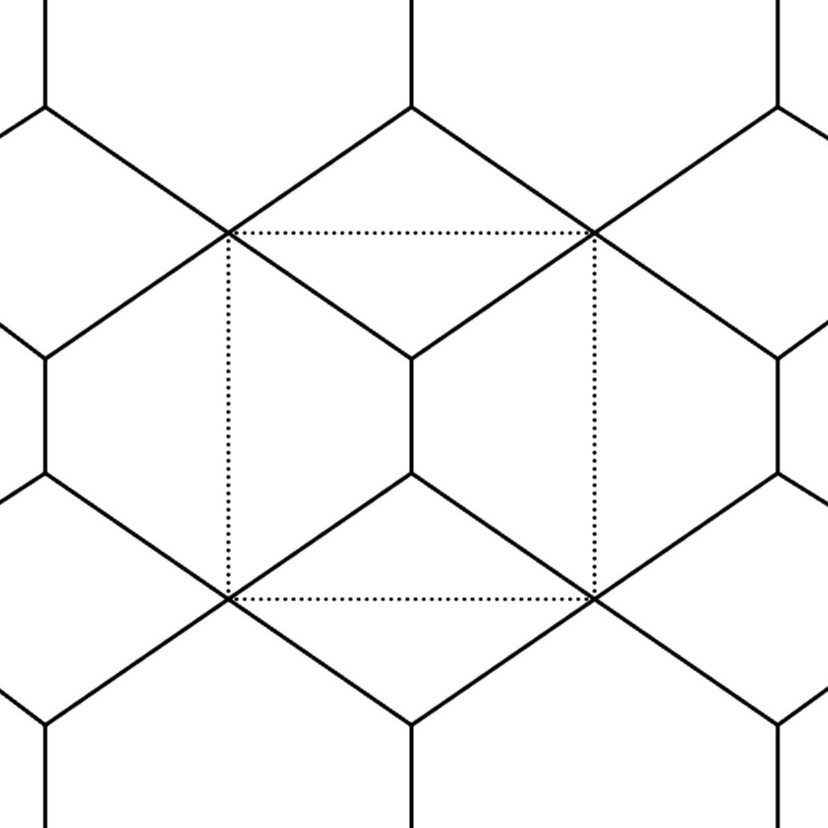}

\end{subfigure}%
\caption{Both networks are $\Z^2$-periodic stationary 1-varifolds, with one period indicated in the box enclosed by dotted lines. Theorem \ref{t.forma2} indicates that these are all appropriate models for leaf vein patterns, as leaves tend to maximize its hydraulic conductance to maintain its functions.}
\label{balancedpattern}
\end{figure}

Interestingly, Theorem \ref{t.forma2} also allows us to access certain questions on leaf vein patterns in rigorous mathematical terms. Specifically we can safely transfer the questions on the geometry of leaf vein patterns to the questions on 1-D stationary varifolds on $\T^2$, which has a more precise definition in mathematics and potentially allows a more careful analysis. For example we can ask the question of maximal valency of leaf vein patterns by asking the same question on 1-D stationary varifolds. We define a stationary network to be a stationary 1-varifold on $\T^2$ that has a.e. density 1. The maximal valency of such a network is defined as the maximal number of edges joining at some node (See Problem \ref{prob.maxval} for more details).

\begin{problem}[See Problem \ref{prob.maxval}]
    \label{prob.intromaxvalency}
    Is the maximal valency of irreducible (See Definition \ref{def.irreducible}) stationary networks on $\T^2$ bounded? Is it bounded by 4, 5 or 6?
\end{problem}

The interest of this question is that the maximal valency bound $4$ can be intuitively discovered from real leaves (See for example Figure \ref{realloopy}), but the math is not so clear. We refer to Section \ref{subsection.openquestion} for the rigorous formulation of this question and some basic discussions. We also explain the notion irreducibility there.

\subsection{Rigorous discussion on the mathematical results}\label{subsection.rigorousdiscussionintro}
In the following we present some new techniques and consequences in Theorem \ref{t.forma1} and Theorem \ref{t.forma2}.

\subsubsection{An outline for the lower attainability characterization}
\label{subsubsection.outlineoftheoremB}

Let us briefly outline the new notions and techniques in the proof of Theorem \ref{t.forma1}. There are three main ingredients in proving Theorem \ref{t.forma1}: countable decompositions, Wa\.{z}ewski parametrizations (See Section \ref{subsubsection.wazuskiprel}) and some covering space arguments (See Section \ref{subsection.torusandtopology}). We refer to Section \ref{section.leafvenation} for a complete presentation.

First, for a medium with its support having finite $\Ha^1$ measure, we have the following countable decomposition. 

\begin{theorem}[See Theorem \ref{t.countabledecompositionof1dmedium}]\label{t.countabledecompositionof1dmediumintro}
    Let $\theta$ be a medium such that $Q(\theta)\ne0$ and $\Ha^1(\spt{\theta})<\infty$.  Then there exist countably many 1-rectifiable mutually disjoint connected components $E_i\subset\spt{\theta}$ such that the submedia $\theta_i:=\restr{\theta}{E_i}$ satisfy
   \be
   \label{eq.countabledecompintro}
  Q(\theta)= \sum_{i=1}^\infty Q(\theta_i) .
   \ee
\end{theorem}

The condition $\Ha^1(\spt{\theta})<\infty$ is necessary as there is a counterexample in Example \ref{ex.uncountablymanycomponentsmaximalmedia}. 

Theorem \ref{t.countabledecompositionof1dmediumintro} is proved by an induction argument to select ``nice'' connected components, which hinges on the following characterization of nontrivial medium.

\begin{theorem}[See Theorem \ref{t.dimensiononeofnontrivialmedium}]
    \label{t.dimensiononeofnontrivialmediumintro}
    Let $\theta$ be a medium with $Q(\theta)\ne0$, then the support $\spt{\theta}$ is not totally disconnected, that is, there is a nonsingleton component in $\spt{\theta}$. This implies that the 1-D Hausdorff measure
    \[
    \Ha^1(\spt{\theta}) >0.
    \]
    In particular, the Hausdorff dimension $\dim_{\Ha}(\spt{\theta})\ge 1$.
\end{theorem}

The decomposition in Theorem \ref{t.countabledecompositionof1dmediumintro} reduces the problem in Theorem \ref{t.forma1} to the case where the support $\Gamma=\spt{\theta}$ is further connected and hence 1-rectifiable. 

The second main step is the Wa\.{z}ewski parametrization, which states that there is a surjective constant speed reparametrization $\gamma:[0,1]\rta \Gamma$, with constant multiplicity along with many other fine properties (See Theorem \ref{t.existenceofanicelipschitzparam} for more details and references).  Using this we prove the inclusion $H_{\Gamma}^\perp \subset\ker Q(\theta)$ (See Section \ref{subsection.oneside}). The key is to construct in a neighborhood of $\Gamma$ the linear function $p\cdot x$ for $p\in H_\Gamma^\perp$. This is accomplished by analyzing the path lifting in the covering space $\R^n/\Z H_\Gamma$ of a surjective Wa\.{z}ewski parametrization $\gamma$ as described above.

In our last step (Section \ref{subsection.wazewskiparam} and \ref{subsection.proofoftheoremqhomotopy}), we finish the proof by showing the following inequality
 \[
 p\cdot Q(\theta) p \ge C{|p|^2},
 \]
for all $p\in \Z H_\Gamma$ (which is equal to $H_\Gamma$ because $\Gamma$ is connected in this step; see Lemma \ref{l.homotopygroup}), where $C>0$ depends only on $n$, the geometry of $\Gamma$ and the coercivity constant in \eqref{eq.coerciveintro}. We establish this lower bound by a modification procedure that improves the estimation of lengths and multiplicities of Lipschitz closed paths in $\Gamma$ without changing the homotopy classes via the techniques in \cite{ALBERTI201735}.

The Corollary \ref{cor.positivedefiniteequalreticulate} follows from Theorem \ref{t.forma1} and Lemma \ref{l.characterizeloopysetn=2}, which characterizes the equivalence relation between loopiness and reticulation in dimension $n=2$.

\subsubsection{Fractional monotonicity formula and pointwise dimension bound}\label{subsubsection.fractionalmono}
We will be using notations from the theory of varifolds. We refer to Section \ref{subsubsection.varifolds} for the preliminaries. It is known that a stationary $k$-varifold $\mu(x,\tau)$ satisfies that \(\displaystyle\frac{1}{r^k}\norm{\mu}(B_r(x))\) is monotone nondecreasing in $r>0$, where $\norm{\mu}$ is the area distribution. See \cites{lsimon,Allard1972} for references, and also see Section \ref{subsubsection.varifolds}. In Lemma \ref{l.fracmonotonicityformula} we extend this monotonicity result to a fractional version for media. To explain the term ``fractional'' we introduce the following new version of ``rank''  (See Definition \ref{def.realizabledimension}) for a positive semi-definite matrix $A\in \R^{n\times n}$:
\[
1\le\dimr(A):=\frac{\trace{A}}{\lambda_{\textup{max}}(A)}\le \textup{rank}(A),
\]
where $\lambda_{\textup{max}}(A)$ is the maximal eigenvalue of $A$. This new ``rank'' is useful as it provides a sharp characterization for media that can be realized as $k$-varifolds through the transformation in \cites{mease4,ambrosio1997}. Indeed, in Lemma \ref{l.realizablematrix}, we show that a positive semi-definite matrix $A\in \R^{n\times n}$ can be written as 
\[
A=\int_{G(k,n)} P_\tau d\rho(\tau)
\]
for some probability measure $\rho$ on the Grassmannian manifold $G(k,n)$ (See Section \ref{subsubsection.varifolds}) if and only if $\dimr(A)\ge k$ and $\trace{A}=k$, where $P_\tau$ is the orthogonal projection matrix from $\R^n$ to the $k$-dimensional subspace $\tau$. Based on this observation we obtain that all media can be regarded at least as 1-varifolds. This is one of the key observations in proving Theorem \ref{t.forma2}. We refer to Theorem \ref{t.equivalence} for more details. 

For a medium $d\theta:=\sigma \, d\norm{\theta}$ as introduced in \eqref{eq.mediumintro} we consider the lower semi-continuous envelop of the \say{rank} $\dimr(\sigma(x))$ of the local anisotropy $\sigma$
\[
\dimrlow(\theta)(x):=\sup_{\delta>0}\essinf\{\dimr(\sigma(y))\ ; \ |y-x|\le \delta,\, y\in \spt{\theta}\},
\]
where $\spt{\theta}$ is the support of $\norm{\theta}$ and ``ess\,inf'' is taken with respect to the Radon measure $\norm{\theta}$. Based on these new notions we show a fractional monotonicity formula in Lemma \ref{l.fracmonotonicityformula} for media that attain the upper Wiener bound \eqref{eq.singularforveinswiener} (or equivalently divergence free media $\gd\cdot\theta=0$ by Lemma \ref{l.equationformaximalmedia}). Specifically we show that
\[
\textup{for } \alpha\in  [1,\dimrlow(\theta)(x)) ,\quad\textup{\(\displaystyle \frac{1}{r^\alpha} \norm{\theta}(B_r(x))\) is monotone nondecreasing }
\]
in $0<r<r_{\alpha,x}$ for some $r_{\alpha,x}>0$. Moreover, if $\dimrlow(\theta)(y)\ge m$ for $y\in U\cap \spt{\theta}$ in an open neighborhood $U$ of $x$, then $\alpha$ can be chosen as the constant $m$. This result slightly improves the standard monotonicity formula for $k$-varifolds, as the corresponding medium $\theta_\mu$ of a $k$-varifold $\mu$ always satisfies
\[
\dimrlow(\theta_\mu)(x) \ge k \textup{ for }x\in\spt{\theta}
\]
with equality holds if and only if for $\norm{\mu}$-(or equivalently $\norm{\theta_\mu}$-)almost every $x$ the matrix field $\sigma_\mu(x):=\frac{d\theta_\mu}{d\norm{\theta_\mu}}$ is an orthogonal projection to a $k$-dimensional subspace of $\R^n$.

Based on this fractional monotonicity result, we can show a pointwise dimension bound similar to \cite{meas3}*{Corollary 1.4}. We also refer to \cite{mease6} for more references on the problems of dimension bound.
\begin{theorem}[See Theorem \ref{t.realizabledimensionandlocaldimension}]
    \label{t.dimensionboundintro}
   Suppose a medium $\theta$ attains the upper Wiener bound \eqref{eq.singularforveinswiener}. Then for every point $x$ in its support, one has the inequality
   \be
\label{eq.dimensionboundintro}
\dimrlow(\theta)(x) \le \dimlow(\theta)(x),
   \ee
   where $\dimlow(\theta)(x)$ is the standard lower local dimension of $\norm{\theta}$ (See Definition \ref{def.localdimension}). In particular, the lower local dimension always satisfies $\dimlow(\theta)(x) \ge 1$.
\end{theorem}

The basic spirit of this result is that, when the maximality is attained, the ``rank'' of the matrix field does not exceed the lowest dimension of the local mass distribution. To be more specific, we consider a medium in $\T^3$ of the form
\[
d\theta:= \sigma \, d\restr{\Ha^2}{\T^2\times\{0\}},
\]
where $\sigma$ is a constant positive semi-definite matrix. Notice that $\theta$ attains the upper Wiener bound \eqref{eq.singularforveinswiener} if and only if $\sigma p =0$ for all $p\in \{0\}^2\times\R$. An immediate consequence of this observation is that 
\[
\textup{rank}(\sigma) \le \dim_{\Ha}\lb\T^2\times\{0\} \rb =2.
\]
This shows that, at least heuristically, the conductance-maximizing media tend to avoid the waste of material to put anisotropy off the plane $\T^2\times\{0\}$. The bound \eqref{eq.dimensionboundintro} provides a extension of this observation in general maximizing media. We refer to Example \ref{ex.dimensionboundnotanidentity} and Example \ref{ex.upperrealdimstrictlygreaterthanlocaldim} for more examples around this idea.

\subsection{Literature: on leaves}

\subsubsection{The cohesion-tension theory}
Let us begin with a short presentation of the hydraulic behavior within a tree. The water transportation within trees can be explained by the \emph{Cohesion-Tension} theory (CTT): the water molecules are cohesive and hence the water forms a continuous flow within the trees; because of the air/water surface tension caused by the effects of transpiration in the open pores of the stomata, the pressure of the water within the transport system is negative; the negative pressure then drives the water within the xylems upward against gravity to the leaves. After being driven through the xylem to the petiole of a leaf \cites{McElrone2013}, the water will be transported through different hierarchies of the leaf veins to the opening sites of the stomata and evaporate there \cites{McKown2010,RothNebelsick2001,sack2013leaf}.  CTT was first proposed by Boehm \cites{boehm1893capillaritat} in 1893 and was soon summarized by Dixon and Joly \cites{dixon1895ascent} in 1894. The theory has been thoroughly studied by experiments and remains the most accepted explanation of water transport within plants \cites{zimmermann1994xylem,canny1995new,zimmermann2004water,Kim2014}. 

\subsubsection{Darcy's law}

The Darcy's law is a widely used model for porous media, including the hydraulic behavior in trees \cites{darcylaw,McDowell2015}. It states that the flow velocity field $j$ is related to the pressure gradient $\gd \phi$ in the following form
\[
j=-A \gd \phi,
\]
where $A$ is a positive definite matrix field that represents the local conductance of the tree/leaf. We will give a more precise description of our model based on Darcy's law on leaf vein patterns in Appendix \ref{appendix.formalderivation}.

\subsubsection{Hierarchy, fractal structure, and power law}
The hydraulic properties of leaf venations have received much attention for decades \cites{sack2013leaf,RothNebelsick2001,Cochard2004,Roth1995,Sack2004}. It is commonly accepted that lower-order veins are typically composed of a midrib and two orders of lateral veins and provide for bulk and far-reaching water transport, while higher-order veins often show smaller and reticulate patterns that provide ``diffusive'' local water distribution \cites{RothNebelsick2001,sack2013leaf}. This hierarchical order can also be observed in the models proposed in \cites{Bohn2007,Corson2010,Katifori2010}, where Bohn et al., Corson, and Katifori et al. studied the optimization of hydraulic dissipation of networks, based on some power law constraint \cite{banavar1999size}, which naturally leads to a fractal geometric viewpoints on the venation patterns. A well-known power law, \emph{Murray}'s law \cites{Sherman1981}, originally proposed for modeling blood vessels, is also used in the study of leaf venations. The law was discovered by finding the balance between the efficiency of blood flow and the maintenance cost of vessel construction. {Murray}'s law is shown in \cites{mcculloh2003water} to hold also for leaf venations under some assumptions, but it was also pointed out in \cites{RothNebelsick2001} that this law may not hold for plants because plants are sun-powered and the plant vessels are often leaky. 

In this paper, instead of looking at the local efficiency of veins, which leads to Murray type laws, we look for appropriate geometric shapes of networks from a larger scale by using homogenization method \cites{Jikov1994,Tartar2010}. We model leaf vein patterns, especially the higher-order ones, with these networks. This method no longer leads to a fractal geometry but it still satisfy the Darcy's law for porous medium, which is appropriate for modeling the ``diffusive'' water transportation in the higher order vein structures.

\subsubsection{The random fluctuations and reticulation}\label{subsubsection.randomfluctuationsandreticulation}
The random fluctuations of the pressure gradient \(p\) in \eqref{eq.classicalWIenrbound} can be justified by the experimental results in \cites{beyschlag1992stomatal,Beyschlag1990,Raschke1970}, where it is observed that the closing and opening of the stomata pores (sites of transpiration) form a random patchy distribution (see Figure \ref{fig:infitrationdistribution}). The reticulation of veins is regarded as a redundancy to deal with such random fluctuations and possible exterior damages \cites{Corson2010,Katifori2010}. In these works, the random opening and closing of sinks are shown to be sufficient for producing loops in the optimal structure. Without random fluctuations, the computation indicates that the optimal patterns are tree-like \cites{Bohn2007,banavar2000topology,durand2007structure,bernot2009optimal}. 

\begin{figure}[htbp]
    \centering
    \includegraphics[width=0.4\textwidth]{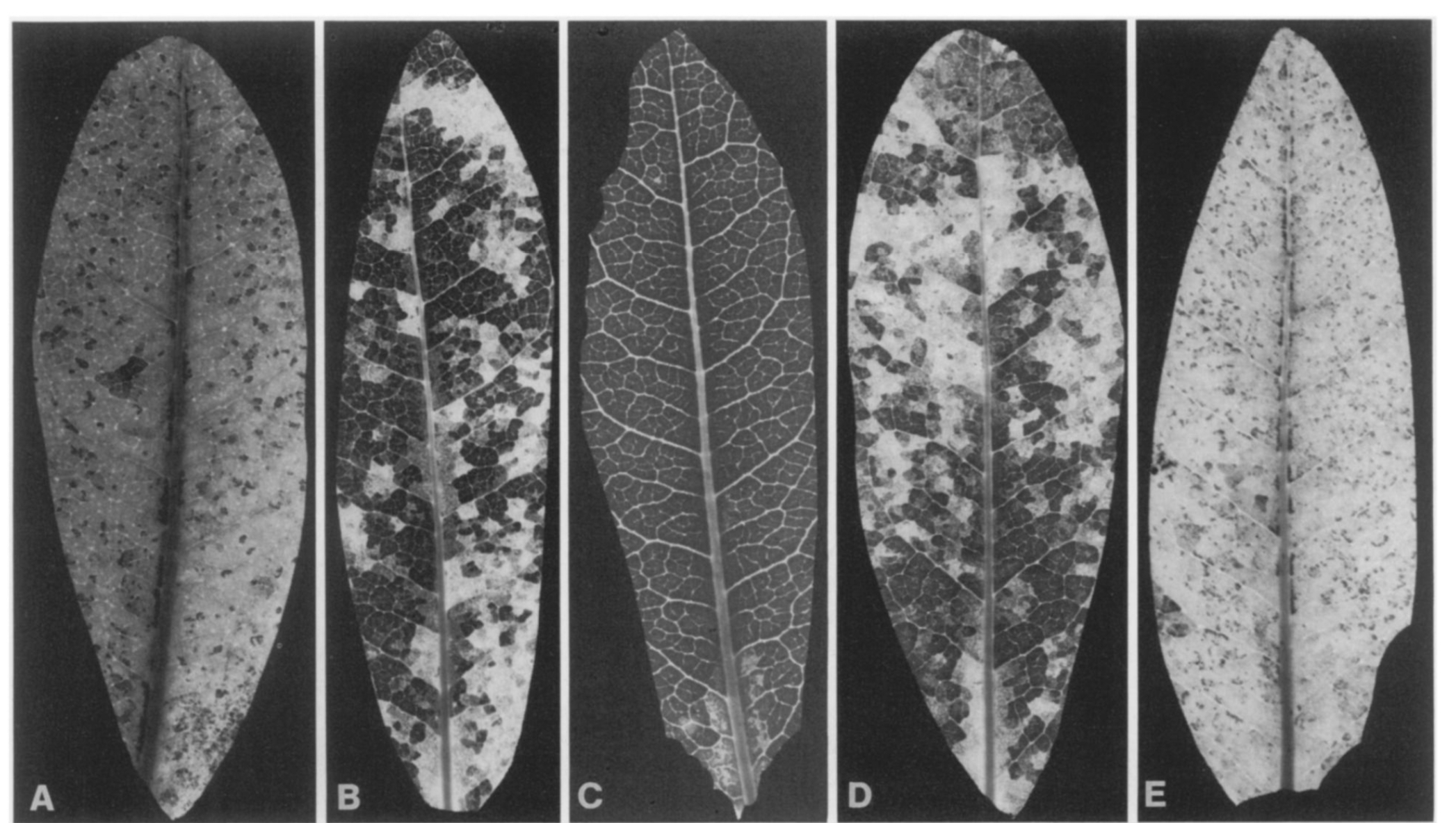}
    \caption{Backlit lower leaf surface of \emph{Arbutus unedo} of different times of a day. (Light area: infiltrated area; dark area: non-infiltrated area. A: 9 a.m.; B: 10 a.m.; C: 12p.m.; D: 4 p.m.; E: 5 p.m.) This picture is reproduced from \cites{Beyschlag1990}, with permission from Springer Nature.}
    \label{fig:infitrationdistribution}
\end{figure}

\subsubsection{Other models on the venation patterns}

Models on many other aspects of veins have also been proposed. In \cites{Dimitrov2006}, Dimitrov and Zucker studied auxin concentration using reaction-diffusion equations and used this to model leaf venation formation in the early stage of a leaf. In \cites{ronellenfitsch2016global,Katifori2018,Ronellenfitsch2019}, Ronellenfitsch and Katifori also proposed a model on auxin concentration and obtained the reticulation result if one assumes fluctuations in auxin production. In \cites{Ronellenfitsch2021}, Ronellenfitsch justified the optimality of reticulation in the elasticity aspect of leaf vein structures.

\subsection{Literature: mathematics}

\subsubsection{Optimal design and Hashin-Shtrikman bound}The upper and lower attainability problem for the singular version of the Wiener bound \eqref{eq.classicalWIenrbound} lies within the more general realm of optimal design \cites{optdesign1,optdesign2,optdesign3,optdesign4}. We also refer to \cites{Milton_2002,cherkaev2000variational} for a detailed introduction in this field. In general, an optimal design problem concerns about minimizing a cost function under some mass constraints. In the theory of conductive composites, the most famous one, other than the Wiener bound, could be the Hashin-Shtrikman bound \cites{HASHIN1963127,hashinliup}. In this problem, one can derive a bound sharper than the classical Wiener bound \eqref{eq.classicalWIenrbound} for two-phase isotropic mixtures \cites{HASHIN1963127,lurie1984exact2,lurie1986exactn,tartar1985estimation,hashinliup}. Specifically in the case that the matrix field $A=\lmb\alpha 1_{U_\alpha}+\beta \lb1-1_{U_\alpha}\rb\rmb I_{n\times n}$ for two positive numbers $\alpha>\beta>0$ and $|U_\alpha|=h\in(0,1)$, the effective tensor $Q(A)$ satisfies 
\be\label{eq.hashinshtrikmanbound}
\beta\frac{\alpha+(n-1)\beta+(n-1)h(\alpha-\beta)}{\alpha+(n-1)\beta-(\alpha-\beta)h}\le\frac{\textup{Tr}(Q(A))}{n} \le \frac{n\alpha\beta+(n-1)\alpha^2 h - (n-1)\alpha\beta h}{n\alpha-\alpha h + h\beta}.
\ee
Interestingly, in the case that high conductive material concentrate on lower dimensional sets, i.e. $\alpha\rta\infty$, $\alpha h\rta m\in (0,\infty)$ and $\beta\rta 0$, the upper H-S bound becomes $\frac{n-1}{n}m$. This is exactly the upper singular Wiener bound  (in the form of average eigenvalues) for $(n-1)$-dimensional medium $\theta$ of the form \(d\theta(x) = P_{\tau_x} \, dw(x),\) where $w$ is an $(n-1)$-rectifiable Radon measure and $P_{\tau_x}$ is the orthogonal projection matrix from $\R^n$ to the $(n-1)$-dimensional tangent space $\tau_x$ of $w$ at $x$. Note that there is a formal dimension reduction from $I_{n\times n}$ to $P_{\tau_x}$. One should notice that the classical sequentially laminated optimizers of the H-S bound no longer capture the attainability property of this singular H-S bound as they simply homogenize in the limit. The author studied the attainability of this singular bound in dimension 2 in \cite{Huang2021}, and this became one of the motivation for studying the singular Wiener bound in a more general setting.

\subsubsection{Singular structures, metric Sobolev spaces and \texorpdfstring{$\mathcal{A}$}{s}-free measures}The modeling of singular structures is one of the most challenging problems in calculus of variations. This topic has attracted attention for decades \cites{Bouchitte1997,ZhikovYosifian2014,bouchitte2001,sing1,sing2}. It also motivated the huge development of Sobolev spaces on metric measure spaces in recent years \cites{metricsobolev1,Lucic2021,Ambrosio2013}.  Interestingly the study of minimal surfaces also provides some good insight in some problems of singular structures. As we have pointed out, the transformation from stationary varifolds to divergence free matrix-valued measures made in \cites{mease4,ambrosio1997} is an important ingredient in proving Theorem \ref{t.forma2}. In fact, such observations have led to an extensive study of linear distributional PDEs, specifically $\mathcal{A}$-free measures \cites{meas3,mease4,mease6,meas7}. 

\subsubsection{Optimal transport as a model of leaf veins} There are also some works on modeling leaf vein patterns by using the theory of optimal transportation \cites{Villani2009}. We refer to \cites{optitrans1,bernot2009optimal} for more details. In some sense the optimal transportation approach follows the same route of Murray type power laws. As we have alluded in Section \ref{subsubsection.randomfluctuationsandreticulation}, these models lead to tree-like patterns. To the best of the author's knowledge, there is no analytic results on the emergence of reticulation under the assumption of random fluctuations.

\subsection{Structure of the paper}

In Section \ref{section.preliminaries} we present some preliminary results from various fields, including measure theory, geometric measure theory, calculus of variations and homotopy groups on the torus. In Section \ref{section.conductivemdium}, we introduce the notion of medium and its effective conductance tensor. We also introduce some novel properties of the effective tensors in the singular case, which includes super-additivity, upper semi-continuity in weak* topology, discontinuity and faithful convergence, and efficient submedium. We also provide a rigorous derivation of the Euler-Lagrange equation $\gd \cdot\theta=0$ for maximal media. In Section \ref{section.maximalmedium} we discuss the transformation of varifolds and establish Theorem \ref{t.forma2} rigorously. We provide a proof for the pointwise dimension bound in Theorem \ref{t.dimensionboundintro}. We also provide a rigorous discussion on Problem \ref{prob.intromaxvalency}.  In Section \ref{section.nontrivialmedia} we provide the proofs of several useful theorems that characterize the supports of media $\theta$ that satisfy $Q(\theta)\ne0$. In particular, we prove the countable decomposition Theorem \ref{t.countabledecompositionof1dmediumintro}. We also provide some examples to show the sharpness of these theorems. In Section \ref{section.leafvenation}, we prove Theorem \ref{t.forma1} and Corollary \ref{cor.positivedefiniteequalreticulate}.

\subsection*{Acknowledgments} I am deeply grateful to my advisor, William M. Feldman, for his unwavering encouragement and invaluable guidance throughout the development of this paper. I also gratefully acknowledge the support of the National Science Foundation under Award No. DMS-2407235. My sincere thanks go to Xuefeng Wang for his support and insightful advice during the early stages of this project. I also extend my appreciation to Frederick R. Adler, Andrej Cherkaev, Graeme Milton, Sean D. Lawley, and Raghavendra Venkatraman for their stimulating discussions and constructive suggestions, which have greatly enriched this work.

\section{Preliminaries}\label{section.preliminaries}

\subsection{Notations}
\begin{itemize}
    \item Let $n\ge 1$ be an integer that represents the dimension. We denote $\R^n$, $\Z^n$ and $\T^n=\R^n/\Z^n$ the Euclidean space, integer lattice and the torus with dimension $n$.
    \item If not particularly mentioned, $\theta$ represent a medium and $\mu$ represent a varifold.
    \item Denote $\Lme^n$ as the Lebesgue measure on $\R^n$ or $\T^n$.
    \item Denote $\trace{A}$ to be the trace of a square matrix $A$, that is, the sum of the diagonal entries of $A$.
    \item For any linear subspace $\tau\subset \R^n$, we denote $ P_\tau$ to be the \emph{orthogonal projection} from $\R^n$ to $\tau$.
    \item Denote $G(k,n)$ the Grassmannian manifold consisting of $k$-dimensional subspaces of $\R^n$.
    \item For a measure $w$ on a metric space $X$, we denote $$\spt{w}=\textup{{support} of }w,$$
    which is the complement of the largest open subset of $X$ that has zero $w$-measure. For a matrix-valued measure the support is defined as that of its total variation.
    \item For a real number $t\in \R$ we denote 
    \[
    \lfloor t\rfloor = \textup{ largest integer below }t,
    \]
    and $\lceil t \rceil=- \lfloor -t\rfloor $.
    \item Given a measure $w$ on a measure space $X$ and a measurable map $\Psi:X\rta Y$, we denote the pushforward as
\be
\label{eq.pushforward}
\Psi_\# w (A):=w(\Psi^{-1}(A)) \textup{ for all measurable }A\subset Y.
\ee
Given a measurable subset $E\subset X$, define the restriction of $w$ as
\be
\label{eq.restriction}
w_E(A):=w(E\cap A) \textup{ for all measurable }A\subset X.
\ee

\item  Let $f: X\rta Y$ be any mapping, then define $\textup{Im} f$ as the image of $f$. If $f:X \rta Y$ is further a linear mapping on linear spaces, then define $\ker f$ as the kernel of $f$.

\item  Let $S$ be any finite set, define $\#S$ as the number of elements in $S$.
\end{itemize}

\subsection{Measure theory}

Let $X$ be a metric space. In this subsection we define and present some basic properties of Radon measures and the matrix-valued versions.

\subsubsection{Radon measure and Riesz representation}

A Radon measure $w$ on a metric space $X$ is a measure on the $\sigma$-algebra of Borel sets on $X$ that satisfies
\begin{itemize}
    \item $w(A)$ is finite for any compact $A\Subset X$
    \item $w$ is outer regular on Borel sets and inner regular on open sets.
\end{itemize}

Denote by $\Ra(X)$ the space of all Radon measures on $X$. There is a functional analytic characterization of $\Ra(X)$.

\begin{theorem}[Riesz representation]
    Let $C_0(X)$ be the space of compactly supported continuous functions, then there is a 1-1 correspondence between the positive linear functionals on $C_0(X)$ and $\Ra(X)$. Specifically a nonnegative linear functional $\psi$ is related to a Radon measure $w$ in the following form
    \[
    \psi(f)=\int_X f dw \textup{  for all }f\in C_0(X).
    \]
    In particular, the functional norm
    \[
    \norm{\psi}_{C_0^*(X)} = w(X)
    \]
    if both sides are finite.
\end{theorem}

A convenient consequence of this is the following compactness result.

\begin{lemma}\label{l.compactnessofweakstar}
    Suppose $w_j$ is a sequence of Radon measures with bounded total mass
    \[
    \sup_{j} w_j(X) <\infty,
    \]
    then up to a subsequence there is another Radon measure $w_\infty$ such that $w_\infty(X)<\infty$ and
    \be\label{eq.weakconvergenceofradonmeasure}
    \int_X f dw_j \rta \int_X f dw_\infty
    \ee
    for all $f \in C_0(X)$ as $j\rta\infty$.
\end{lemma}

\begin{proof}
    This is an immediate consequence of the Banach-Alaoglu theorem.
\end{proof}

\begin{definition}
  Call the notion of convergence in \eqref{eq.weakconvergenceofradonmeasure} the weak* convergence and denote $w_j \overset{*}{\weakcv} w_\infty$. 
\end{definition}

\subsubsection{Matrix-valued measures and polar decomposition}
\label{subsubsection.matrixvaluedmeasures}
A matrix-valued measure on a metric space $X$ is an $\R^{n\times n}$-valued set function $\theta$ such that for any countable sequence of pairwise disjoint Borel sets $U_j\subset X$ we have
\[
\theta\lb\bigcup_{i=1}^\infty U_j\rb=\sum_{i=1}^\infty \theta(U_j).
\]
In this paper, we focus on positive semi-definite matrix-valued Radon measures. More precisely we require 
\begin{itemize}
    \item $\theta(E) $ is positive semi-definite for all Borel $E\subset X$.
    \item The trace 
    \be
\label{eq.totalvariationwithrespecttotrace}
\norm{\theta}(E):=\trace{\theta(E)}
    \ee
    is a Radon measure.
\end{itemize}

\begin{lemma}
A matrix-valued measure $\theta$ is a positive semi-definite matrix-valued Radon measure if and only if $\theta$ decomposes as
\[
d\theta=\sigma\, d\norm{\theta}
\]
where $\norm{\theta}$ the trace is a Radon measure and $\sigma\in L^1(\norm{\theta};\R^{n\times n})$ is the Radon-Nikodym derivative 
\be
\label{eq.definematrixfield}
\sigma(x):= \frac{d\theta}{d\norm{\theta}}(x), \ x\in X
\ee
that defines a positive semi-definite matrix field and satisfies $\trace{\sigma(x)}=1$ for $\norm{\theta}$-almost all $x\in X$.
\end{lemma}

\begin{proof}
    By applying the polar decomposition theorem \cite{functionofboundedvariation}*{Corollary 1.29} 
    \[
    d\theta = f \,dw
    \]
    for some Radon measure $w$ and a unique positive semi-definite matrix field $f$ such that $|f|_*\equiv 1$ for some norm $|\cdot|_*$ on $\R^{n\times n}$. Observe that $d\norm{\theta}(x)= \trace{f(x)} \,dw(x).$ We then define
    \[
    \sigma(x) = \frac{f(x)}{\trace{f(x)}}.
    \]
  Notice that $\sigma$ is well-defined because $f(x)$ is positive semi-definite and then $f(x)\ne 0$ if and only if $\trace{f(x)}\ne0$. The proof is done by checking that for each $1\le i,j\le n$ the component $\sigma_{ij}$ is exactly the Radon-Nikodym derivative $d\theta_{ij}/d\norm{\theta}$. 
\end{proof}

\subsubsection{Disintegration theorem}

The following theorem will be often referred to. For convenience we only present a specific version. See \cite{Ambrosio2008}*{Theorem 5.3.1} for a more general version and its proof.

\begin{theorem}[Disintegration theorem]\label{t.disintegrationtheorem}
   Let $\mu$ be a Radon measure on the product metric space $X\times Y$. Let $\pi$ be the projection $\pi:X\times Y\rta X$ that sends $(x,y)\in X\times Y$ to $x$. Then if we denote $\norm{\mu}:=\pi_\#\mu$, there is a family of probability measures $\rho_x$ on $Y$ such that
   \begin{itemize}
       \item for all Borel set $B\subset Y$ the function $x\rta \rho_x(B)$ is Borel measurable on $X$.
       \item for $\norm{\mu}$-almost all $x\in X$, we have $\rho_x(Y\setminus \pi^{-1}(x))=0$.
       \item for all Borel function $f$ on $X\times Y$ we have
       \[
       \int_{X\times Y} f(x,y) d\mu(x,y) = \int_X \int_{\pi^{-1}(x)} f(x,y) d\rho_x(y) d\norm{\mu}(x).
       \]
   \end{itemize}
\end{theorem}

\subsection{Geometric measure theory}

\subsubsection{Lipschitz map and function}

A {Lipschitz map} is a map $f:X\rta Y$ from a metric space $(X,d_X)$ to another metric space $(Y,d_Y)$ satisfying that there is a constant $C>0$ so that for all $x_1,x_2\in X$
\[
d_Y(f(x_1),f(x_2)) \le C d_X(x_1,x_2).
\]
A {Lipschitz function} is a Lipschitz map from $(X,d_X)$ to the real line. 

\subsubsection{Lipschitz path and length}
\label{subsubsection.lipschitzpathandlength}
A path is a continuous map $\gamma:[0,1]\rta X$.  Define the {multiplicity} of $\gamma$ at $x\in X$ as $m(\gamma,x):=\#\gamma^{-1}(x)$ (possibly $\infty$).  

Define the \emph{length} of a Lipschitz path $\gamma$ in $(a_0,b_0)\subset[0,1]$ as 
\[
\ell (\gamma, (a_0,b_0)):= \sup\lma \sum_{i=1}^{N-1} d_X(\gamma(a_i),\gamma(a_{i+1}))\rma<\infty
\]
with the supremum taken over all possible choices of $a_0\le a_1\le a_2 \le \cdots \le a_N \le b_0$. We write $\ell(\gamma):=\ell(\gamma,[0,1])$ as the total length.

Call a Lipschitz path $\gamma$ to have constant speed if $\ell(\gamma,(a_0,b_0))$ is proportional to $(b_0-a_0)$ regardless of $a_0$, $b_0$. It is known that any Lipschitz path having finite length admits a constant speed reparametrization. Notice that if $X$ is a Riemann manifold with Riemann metric $g$, then a constant speed Lipschitz path $\gamma$ satisfies
\[
\textup{Lip}(\gamma)=|\gamma'|_g=\ell(\gamma).
\]
See \cite{ALBERTI201735}*{Section 3} for more details.

\subsubsection{Hausdorff measure and dimension}
\label{subsubsection.hausdorff}
On a metric space $(X,d)$, the diameter of a set $E\subset X$ is defined as 
\[
\diam(E):=\sup_{x,y\in E} d(x,y).
\]
For $\alpha\ge 0$, the $\alpha$-dimensional {Hausdorff measure} of a subset $E\subset X$ is defined as
\[
\Ha^\alpha(E):=\sup_{\delta>0}\inf\left\{\sum_{i=1}^\infty\diam(U_i)^\alpha \ ; \ \diam(U_i)\le \delta\textup{ and } E\subset\bigcup_{i=1}^\infty U_i\right\}.
\]
The {Hausdorff dimension} of a set $E$ is defined as
\[
{\ \dim _{\Ha }(E)=\inf\{\alpha\geq 0:\Ha^{\alpha}(E)=0\}=\sup\{\alpha\geq 0:\Ha^{\alpha}(E)=\infty \},}
\]
where we take ${\displaystyle \inf \emptyset =+\infty }$ and ${\displaystyle \sup \emptyset =0}$. Notice that whenever $E$ has finite $\Ha^\alpha$ measure the restriction $\restr{\Ha^\alpha}{E}$ always define a Radon measure. See \cite{functionofboundedvariation}*{Section 2.8} for more discussions on Hausdorff measures.

\subsubsection{Rectifiable sets and measures}
\label{subsubsection.rectifiable}
On a metric space $X$, a subset $E\subset X$ is called $k$-\emph{rectifiable} for some integer $1\le k \le n$ if there are countably many Lipschitz maps $f_i:\R^k\rta X$ so that
 \[
 \Ha^k\lb E\setminus \bigcup_{i=1}^\infty f_i\lb \R^k\rb\rb=0.
 \]
A measure $\mu$ is $k$-rectifiable if there is a $k$-rectifiable set $E$ and a Borel function $f$ such that
\[
d\mu=f\, d\restr{\Ha^k}{E}.
\]
\subsubsection{Density}
\label{subsubsection.density}
For a measure $\mu$ on $\R^n$, the $k$-dimensional \emph{upper density} at $x\in \R^n$ is defined as
\[
\Theta^*(\mu,x)=\Theta_k^*(\mu,x):=\limsup_{r\rta0^+} \frac{\mu(B_r(x))}{\omega_kr^k},
\]
where $\omega_k$ is the volume of the $k$-dimensional unit ball. Similarly, the \emph{lower density} is defined as
\[
\Theta_*(\mu,x)=\Theta_*^k(\mu,x):=\liminf_{r\rta0^+} \frac{\mu(B_r(x))}{\omega_kr^k}.
\]
For a Borel set $E\subset \R^n$, the $k$-th upper and lower densities are defined by replacing $\mu$ with $\restr{\Ha^k}{E}$.

By \cite{functionofboundedvariation}*{Theorem 2.56}, we know that if $\mu$ is a Radon measure on $\R^n$ and $E$ is a Borel subset, then:
\begin{enumerate}[label=(\alph*)]
    \item If \( \Theta^*(\mu, x) < t \) whenever \( x \in E \), then
    \[
    \mu(E) \leq 2^k t \Ha^k(E).
    \]
    \item If \( E\subset U \) for some bounded open domain $U$ and \( \Theta^*(\mu, x) > t \) whenever \( x \in E \), then
    \[
    \mu(E) \geq t \Ha^k(E).
    \]
\end{enumerate}

A basic consequence of this result is that if $\Theta_k^*(\mu,x)$ is positive and finite for $\Ha^k$-almost every $x\in \spt{\mu}$ then there is a Borel function $f$ such that $d\mu = f \, d\restr{\Ha^k}{\spt{\mu}}$. See \cite{DeLellis2008}*{Proposition 2.16}.

\subsubsection{Tangent measures and rectifiability}\label{subsubsection.tangentmeasuresandrectifiability}
Let $\mu$ be a measure on $\R^n$ and $\alpha\ge 0$. Define an $\alpha$-\emph{tangent measure} $\nu_x$ of $\mu$ at $x\in \R^n$ to be a measure so that there is a sequence of positive numbers $r_i\downarrow0$ such that
\[
\frac{\mu(x+r_i (\cdot))}{\omega_\alpha r_i^\alpha} \overset{\ast}{\weakcv} \nu_x
\]
as $i\rta\infty$, where we have denoted $\omega_\alpha$ the Gamma function extension of volumes of unit balls. Denote by $\textup{Tan}_\alpha(\mu,x)$ the collection of all such $\nu_x$, and $\displaystyle\textup{Tan}(\mu,x):=\bigcup_{\alpha\ge0}\textup{Tan}_\alpha(\mu,x)$. 
\begin{theorem}[\cite{DeLellis2008}*{Theorem 4.8}]\label{t.tangentialcharacterizationofrectifiability}
    A measure $\mu$ on $\R^n$ is $k$-rectifiable for some integer $1\le  k\le n$ if and only if for $\mu$-almost all $x\in \R^n$, $\textup{Tan}(\mu,x)=\textup{Tan}_k(\mu,x)=\{\nu_x\}$ is a singleton, where $\nu_x$ is of the form $c_x\restr{\Ha^k}{\tau}$ for some $c_x>0$ and $\tau\subset\R^n$ is some $k$-dimensional subspace. A set $E\subset\R^n$ is $k$-rectifiable if and only if $\restr{\Ha^k}{E}$ is $k$-rectifiable.
\end{theorem}

\begin{definition}
\label{def.geometrictangentspace}
    For a Radon measure $w$, whenever for some integer $1\le k \le n$ the space
    \[
    \textup{Tan}(w,x) = \textup{Tan}_k(w,x)=\{\nu_x\}
    \]
    is a singleton, where $\nu_x = c_x \restr{\Ha^k}{\tau}$ for some $c_x>0$ and $\tau\subset \R^n$ a $k$-dimensional subspace, we call $\tau$ the tangent space of $w$ at $x$.
\end{definition}

\subsubsection{Varifold, area criticality and monotonicity formula}
\label{subsubsection.varifolds}
Let us now introduce the notion of varifolds. We only present the definitions and facts in $\R^n$ as they can be naturally extended to the case of manifolds such as torus. We refer to \cites{Allard1972,de2012allard,Umenne} for more discussions on this topic.

Denote for integers $1\le k\le n$ the Grassmannian manifold $G(k,n)$ consisting of $k$-dimensional subspaces of $\R^n$. Notice that $G(k,n)$ is a compact smooth manifold.

A $k$-varifold $\mu$ is a Radon measure on $\R^n \times G(k,n)$. Let $\pi$ be the projection that sends $(x,\tau)\in \R^n\times G(k,n)$ to $x$. Define the area distribution \(\norm{\mu}:=\pi_\#\mu\) (See \eqref{eq.pushforward} for definition).

Call a varifold $\mu$ to be \emph{stationary} in an open domain $U\subset\R^n$ if for all smooth vector fields $\Phi\in C_0^\infty(U;\R^n)$
\be\label{eq.stationaryvarifold}
\int_{U\times G(k,n)} \trace{P_\tau D\Phi(x)} d\mu(x,\tau) = 0,
\ee
where $P_\tau$ is the orthogonal projection matrix from $\R^n$ to the $k$-subspace $\tau\subset\R^n$. The following theorem states that stationary varifolds are critical points of the area functional.
\begin{theorem}[\cite{Allard1972}]
   A $k$-varifold $\mu$ on $\R^n$ is a critical point of the functional $\norm{\mu}(U)<\infty$ for some bounded open domain $U\subset\R^n$ if and only if $\mu$ is stationary in the sense of \eqref{eq.stationaryvarifold} in $U$.
\end{theorem}

Suppose $\mu$ is a stationary $k$-varifold on an open domain $U\subset \R^n$, then the following identity holds
\be\label{eq.monotonicityformulaforstationaryvarifold}
\frac{\norm{\mu}(B_r(x))}{r^k}-\frac{\norm{\mu}(B_s(x))}{s^k}=\int_{U\times G(k,n)} \frac{|(I-P_{\tau})(y-x)|^2}{|y-x|^{k+2}} d\mu(y,\tau)\ge0,
\ee
for any $x\in U$ and $0<s<r <\dist(x,\pt U)$. In particular, the density limit 
\[
\Theta(\norm{\mu},x):=\lim_{r\rta0^+}\frac{\norm{\mu}(B_r(x))}{\omega_kr^k}
\]
exists for every $x\in U$.

\subsubsection{Wa\.{z}ewski parametrization}
\label{subsubsection.wazuskiprel}
We now introduce the Wa\.{z}ewski parametrization of connected compact metric spaces (called {continuum}) that have finite length. This will allow us to do fine analysis on sets having finite length.

\begin{theorem}[\cite{ALBERTI201735}*{Theorem 4.4}]\label{t.existenceofanicelipschitzparam}
Let $X$ be a non-singleton continuum such that $\Ha^1(X)<\infty$, then $\Ha^1(X)>0$ and there is a Lipschitz mapping $\gamma:[0,1]\rta X$ satisfying the following properties:
\begin{enumerate}
    \item[(i)] $\gamma$ is closed, surjective and has degree zero.
    \item[(ii)] The multiplicity $m(\gamma,x)=2$ for $\Ha^1$-almost all $x\in X$.
    \item[(iii)] $\gamma$ has constant speed equaling to $2\Ha^1(X)$, and in particular, for any Borel function $f:X\rta [0,\infty]$ (and in particular any integrable Borel function)
    \be
    \dashint_X f (x) d\Ha^1(x) =  \int_0^1 f(\gamma(t)) dt.
    \ee
    
\end{enumerate}

In particular, a continuum $X$ is 1-rectifiable whenever $\Ha^1(X)<\infty$.    
\end{theorem}

\subsection{Torus and its topology}
\label{subsection.torusandtopology}
In the following, we sketch some important known topological properties of torus for audiences who are not familiar with them. Most of the materials in this section are covered in the textbook \cite{algebraictopology}. 

We write the standard flat torus $\T^n:=\R^n/\Z^n$. If not particularly mentioned, we denote
\[
\pi:\R^n\rta\T^n
\]
the standard projection that maps $x\in \R^n$ to the equivalence class of $x$ modulo $\Z^n$. As $\Z^n$ is a discrete subgroup of $\R^n$, standard theory shows that $\pi$ is also a locally isometric diffeomorphism.

\subsubsection{Path and path composition}

A path is a continuous map $\gamma:[0,1]\rta\T^n$. Given two paths $\gamma_1$ and $\gamma_2$ with $\gamma_1(1)=\gamma_2(0)$, define the path composition
\be
\label{eq.pathproduct}
\gamma_1\gamma_2(t):=\bca
\gamma_1(2t) & 0\le t <1/2\\
\gamma_2(2t-1) & 1/2\le t \le 1.
\eca
\ee
 It is not difficult to show that for two Lipschitz paths $\gamma_1$ and $\gamma_2$ with $\gamma_1(1)=\gamma_2(0)$ the path composition $\gamma_1\gamma_2$ is still Lipschitz and the total length satisfies
\be
\label{eq.lengthadditivewrtcomposition}
\ell(\gamma_1\gamma_2)=\ell(\gamma_1) + \ell(\gamma_2).
\ee

\subsubsection{Homotopy classes on torus}
 Call a path $\gamma$ to be closed if $\gamma(0)=\gamma(1)$. Two continuous closed paths $\gamma_1,\gamma_2:[0,1]\rta \T^n$ are homotopic based at $x_0=\gamma_1(0)=\gamma_2(0)$ if there exists a continuous map
\[
H:[0,1]^2 \rta \T^n
\]
such that $H(\cdot,0)=\gamma_1$ and $H(\cdot,1)=\gamma_2$, and $H$ also satisfies $H(0,t)=H(1,t)=x_0$ for all $t\in [0,1]$. By standard theory, the collection of homotopy classes in $\T^n$ based at any point $x_0$ endowed with the path composition forms a group, which we denote by $\pi_1(\T^n,x_0)$. There is a natural isomorphism from $\pi_1(\T^n,x_0)$ to $\Z^n$. We will further discuss this isomorphism in Section \ref{subsubsection.pathliftingproperty}. 

\subsubsection{Covering spaces of torus}
\label{subsubsection.coveringspacesoftorus}
A covering space $(g,X)$ of a topological space $X^*$ consists of a continuous surjection $g:X\rta X^*$ so that for each $x\in X^*$ there is an open neighborhood $U$ such that $g^{-1}(U)$ is a union of disjoint open sets in $X$, restricted on which $g$ is a homeomorphism onto $U$. In this paper we consider for each subgroup $G\le \Z^n$ the quotient space
\[
\T_G^n:=\R^n/G:=\{[x]\ ; \ x\sim y \textup{ if and only if }x-y\in G\}.
\]
By standard theory the spaces $\T_G^n$ are smooth manifolds and the projection $\pi:\R^n\rta \T^n$ can be factored as
\[
\pi=\pi_G\circ \pi^G,
\]
where $\pi$, $\pi^G:\R^n\rta\T_G^n$ and $\pi_G:\T_G^n\rta\T^n$ are locally diffeomorphic and isometric projections. 

\subsubsection{Path lifting property}
\label{subsubsection.pathliftingproperty}
Given a covering map $g:X\rta X^*$ and a path $\gamma:[0,1]\rta X^*$, a lift of the path is another path $\Tilde{\gamma}:[0,1]\rta X$ such that $\gamma=g\circ \Tilde{\gamma}$. The path lifting property states that for each lift $\Tilde{x}_0\in X$ of the starting point $x_0=g(0)$, there is a unique lift $\Tilde{\gamma}$ of $\gamma$ starting at $\Tilde{x}_0$. 

In the case of the covering map $\pi:\R^n\rta\T^n$, given a base point $x_0\in \T^n$ and its lift $y_0\in \pi^{-1}(x_0)$, each closed path $\gamma$ in $\T^n$ based at $\gamma(0)=x_0$ can be lifted to a unique path $\Tilde{\gamma}$ in $\R^n$ starting at $\Tilde{\gamma}(0)=y_0$. The vector $\Tilde{\gamma}(1)-y_0\in\Z^n$ turns out to be independent of $y_0$ and the choice of $\gamma$ in its homotopy class. The map from the homotopy class $[\gamma]$ to $\Tilde{\gamma}(1)-y_0$ is known to be an isomorphism, denoted by $i_{x_0}$, from $\pi_1(\T^n,x_0)$ to $\Z^n$. If the base point $x_0\in\T^n$ is unambiguous, we do not distinguish $\Z^n$ and the homotopy group based at $x_0$.

One also notice that because the covering map $\pi$ is a local isometry, the length $\ell(\gamma)$ is equal to the length of its lift $\ell(\Tilde{\gamma})$ whenever $\gamma$ is Lipschitz.

\subsubsection{Periodic extension of measures}
\label{subsubsection.periodicextensionofmeasures}
For any Radon measure $w$ on $\T^n$, we can extend $w$ uniquely to a periodic Radon measure $\Tilde{w}$ on $\R^n$. Define for any $\phi\in C_0^\infty(\R^n)$ the following linear functional
\[
L_w(\phi):= \int_{\T^n} \sum_{y\in \pi^{-1}(x)}\phi(y) dw(x).
\]
By the Riesz representation theorem, this defines a unique periodic Radon measure $\Tilde{w}$ on $\R^n$ such that
\[
\int_{\R^n} \phi \, d\Tilde{w}=L_w(\phi).
\]
In particular, for any $\alpha\ge0$ and subset $E\subset \T^n$ such that $\Ha^\alpha(E)<\infty$, the periodic extension of $\restr{\Ha^\alpha}{E}$ is exactly $\restr{\Ha^\alpha}{\pi^{-1}(E)}$. 

Similar periodic extension can be done for matrix-valued Radon measures. Specifically for a matrix-valued Radon measure of the form $d\theta = \sigma\, d\norm{\theta}$ we periodically extend the Radon measure $\norm{\theta}$ as above to be a periodic Radon measure $w$ on $\R^n$ and then define the periodic extension $\theta^*$ of $\theta$ as
\be
\label{eq.periodicextensionfomediumn}
d\theta^*(y):= \sigma(\pi(y)) \, dw(y).
\ee

\section{Conductive medium and its average behavior}\label{section.conductivemdium}

In this section we present some elementary results on mixture of conductive materials that are presented by (especially singularly supported) matrix-valued measures. We call such objects {medium} and define them precisely as follows.

\begin{definition}\label{def.singulartensorcomposite}
     Define a \emph{conductive medium} (or simply \emph{medium}) to be a positive semi-definite matrix-valued Radon measure on $\T^n$ (See Section \ref{subsubsection.matrixvaluedmeasures}). A medium $\theta$ is called \emph{isotropic} if $\theta(E)=\lambda(E)I_{n\times n}$ for any Borel set $E\subset \Omega$ and some number $\lambda(E)\ge0$. {Here $\lambda$ is in fact a Radon measure.}
\end{definition}

Define the mass distribution of a medium $\theta$ as its total variation with respect to the trace
\be\label{eq.totalvariationofmedium}
\norm{\theta}(E):=\trace{\theta(E)}.
\ee
According to the discussions in Section \ref{subsubsection.matrixvaluedmeasures}, we know that any medium $\theta$ can be uniquely decomposed as
\be\label{eq.normaldecomposition}
d\theta = \sigma \, d\norm{\theta},
\ee  
where $\norm{\theta}$ is a Radon measure and $\sigma$ is a positive semi-definite matrix field such that 
$$
\trace{\sigma(x)}=1 \textup{ for } \norm{\theta}\textup{-almost every } x\in \T^n.
$$ 
In particular, if $\theta$ is isotropic, we obtain that $\sigma\equiv \frac{1}{n}I_{n\times n}$. If not particularly mentioned, we always identify Radon measures with isotropic media.

Physically the matrix field $\sigma$ represents the local anisotropy of the material and the mass distribution $\norm{\theta}$ represents the local magnitude of conductance as a sum of all directions. Our main goal in this section is to present some basic facts on the average conductive behavior of a given medium.

\begin{definition}
\label{def.effectivetensor}
For a medium $\theta$, define the \emph{effective conductance tensor} (or simply \emph{effective tensor}) as the positive semi-definite matrix $Q(\theta)\in \R^{n\times n}$ that satisfies
\be\label{eq.effectivetensor}
p\cdot Q(\theta)p:=\inf_{\varphi\in C^\infty(\T^n)} \int_{\T^n} (\gd\varphi +p)\cdot \sigma (\gd\varphi +p) d\norm{\theta}
\ee
for every $p\in \R^n$ and we also define the \emph{mean conductance} as the average of eigenvalues
\be\label{eq.meanconductance}
M(\theta):=\frac{\trace{Q(\theta)}}{n}.
\ee
\end{definition}
See Lemma \ref{l.whyquadratic} for the proof that \eqref{eq.effectivetensor} is indeed a quadratic form. The quantity in \eqref{eq.effectivetensor} is known to be quadratic in periodic homogenization in \cite{Jikov1994}*{Chapter 1} and in ergodic versions in \cite{Armstrong2019}*{Chapter 1} for the case when $\norm{\theta}$ is absolutely continuous with respect to Lebesgue measure. The main difference is that $\theta$ in this paper can have singular support.

The main character of this paper is the following singular version that generalize the classical Wiener bound \cite{Wiener1912}.

\begin{lemma}[Singular Wiener bound]\label{l.singularwienerbound}
    The effective tensor $Q(\theta)$ satisfies
    \be\label{eq.singularwienerbound}
    0\le  Q(\theta) \le \theta(\T^n).
    \ee
    In particular, the mean conductance satisfies
       \be\label{eq.singularwienerboundformeanconductance}
   0\le M(\theta)\le \frac{1}{n} \trace{\theta(\T^n)}=\frac{1}{n} \norm{\theta}(\T^n).
    \ee
\end{lemma}

\begin{proof}
    The lower bound immediately follows from the definition of $Q(\theta)$. The upper bound is obtained by taking the special test function $\varphi=0$ in the formula. The bound for the mean conductance $M(\theta)$ is obtained by taking the average eigenvalue of $Q(\theta)$.
\end{proof}

Recall the classical Wiener bound \eqref{eq.classicalWIenrbound} when $d\theta(x)= A(x) d\Lme^n$ for some matrix field $A$, one has
\[
\lb \int_{\T^n} A^{-1}(x) d\Lme^n \rb^{-1} \le Q(\theta) \le \int_{\T^n} A(x) d\Lme^n.
\]
The upper bound in \eqref{eq.singularwienerbound} is a natural extension of the classical Wiener upper bound, but the lower bound in \eqref{eq.singularwienerbound}, although looks more trivial, is sharp in general. For example, one can consider any finite sum of Dirac delta
\[
\rho:= \sum_{i=1}^N \delta_{x_i},
\]
where $x_i\in \T^n$ for $i=1,\dots,N$, and then find that
\[
Q(\rho)=0
\]
no matter how we choose $N$ and the position $x_i$'s. In fact, we will show in Theorem \ref{t.dimensiononeofnontrivialmedium} that $Q(\theta)=0$ whenever $\spt{\theta}$ is totally disconnected. 

To classify the media that have different average behaviors, we now introduce some new terminologies.

\begin{definition}\label{def.differentmedia}
Call a medium \(\theta\) \emph{trivial} if \(\displaystyle Q(\theta) = 0\); \emph{nontrivial} if \(\displaystyle  Q(\theta) \ne 0\); \emph{positive} (or \emph{resilient}) if \(\displaystyle Q(\theta) > 0\) is positive definite; \emph{maximal} if $\displaystyle Q(\theta)=\theta(\T^n)$, i.e., $\theta$ achieves the upper bound in \eqref{eq.singularwienerbound}. 
\end{definition}

We remark here that a medium $\theta$ is trivial if and only if $M(\theta)=0$. It is maximal if and only if $M(\theta)$ reaches the upper bound in \eqref{eq.singularwienerboundformeanconductance}. Also note that a maximal medium is not necessarily positive.

In the following we show some basic properties of the effective tensor $Q$. In Section \ref{subsection.effectivequadraticform} we show that the quantity in \eqref{eq.effectivetensor} is indeed a quadratic form. In Section \ref{subsection.submedumandsuperadditivity}, we present basic properties of $Q(\theta)$ with respect to the addition of media. In Section \ref{subsection.faithfulconvergence} we show that $Q(\theta)$ is upper semi-continuous with respect to the weak* topology of media. In Section \ref{subsection.efficientsubmedium} we introduce the notion of efficient submedium. In Section \ref{subsection.elequation} we derive the Euler-Lagrange equation for maximal media.

\subsection{The effective quadratic form}\label{subsection.effectivequadraticform}
\begin{lemma}\label{l.whyquadratic}
    The quantity defined in \eqref{eq.effectivetensor} is indeed a quadratic form.

\end{lemma}

Let us write 
$$
    C_\theta(\varphi,p):= \int_{\T^n} (\gd \varphi+p)\cdot \sigma(\gd \varphi+p) d\norm{\theta} \textup{ and }C_\theta(p):=\inf_{\varphi\in C^\infty(\T^n)} C_\theta(\varphi,p).
$$
Before proving Lemma \ref{l.whyquadratic} we first show a technical lemma.

\begin{lemma}
    \label{l.techtoshowquadratic}
A sequence of functions $\varphi_j\in C^\infty(\T^n)$ satisfies
\be
C_\theta(\varphi_j,p) \downarrow C_\theta(p) \textup{ as } j\rta\infty
\ee
if and only if for every $\ep>0$ there is a $j_0>0$ such that for all $j>j_0$ and $\eta\in C^\infty(\T^n)$
\be\label{eq.approximatecorrectorequation}
\lw\int_{\T^n} (\gd \varphi_j+p)\cdot \sigma \gd \eta d\norm{\theta}\rw \le \ep \sqrt{\int_{\T^n} \gd \eta\cdot \sigma \gd \eta d\norm{\theta}}.
\ee
\end{lemma}

\begin{proof}
    To show the only if part, we take $\varphi_j$ to be a minimizing sequence and for some $j_0>0$ and all $j>j_0$ we have
\[
C_\theta(\varphi_j,p)\le C_\theta(p)+\ep.
\]
By minimality of $C_\theta(p)$ we have for all $\eta\in C^\infty({\T^n})$ and $t\ne0$
\[
C_\theta(\varphi_j,p)\le C_\theta(\varphi_j-t\eta,p) + \ep,
\]
which implies that
\[
2t\int_{{\T^n}} (\gd \varphi_j+p)\cdot \sigma \gd \eta d\norm{\theta} -{t^2}\int_{\T^n} \gd \eta\cdot \sigma \gd \eta d\norm{\theta} \le \ep.
\]
If $\int_{\T^n} \gd \eta\cdot \sigma \gd \eta d\norm{\theta}=0$ the inequality \eqref{eq.approximatecorrectorequation} is trivial. If $\int_{\T^n} \gd \eta\cdot \sigma \gd \eta d\norm{\theta}>0$ then by arbitrariness of $t\ne0$
\[
\frac{\lw\int_{{\T^n}} (\gd \varphi_j+p)\cdot \sigma \gd \eta d\norm{\theta}\rw^2}{{\int_{\T^n} \gd \eta\cdot \sigma \gd \eta d\norm{\theta}}} \le  \ep,
\]
regardless of the choice of $\eta$.

To show the if part, we observe that if there is a $\delta>0$ such that
\[
C_\theta(\varphi_j,p)\ge C_\theta(p)+\delta
\]
for all large $j$, then for all $j$ there will be a smooth test function $\eta_j$ such that
\[
C_\theta(\varphi_j,p)\ge C_\theta(\varphi_j+\eta_j,p)+\delta/2
\]
also for all large $j$. In other words, we have
\[
-2\int_{\T^n} (\gd \varphi_j+p)\cdot \sigma\gd \eta_j d\norm{\theta} -\int_{\T^n} \gd \eta_j \cdot \sigma \gd\eta_j d\norm{\theta} \ge \delta/2,
\]
while on the other hand, we have for all real number $t$
\[
2t\int_{\T^n} (\gd \varphi_j+p)\cdot \sigma\gd \eta_j d\norm{\theta} -{t^2}\int_{\T^n} \gd \eta_j \cdot \sigma \gd\eta_j d\norm{\theta}\le \frac{\lw\int_{{\T^n}} (\gd \varphi_j+p)\cdot \sigma \gd \eta_j d\norm{\theta}\rw^2}{{\int_{\T^n} \gd \eta_j\cdot \sigma \gd \eta_j d\norm{\theta}}} \le  \ep.
\]
By taking $t=-1$, we have
\[
-2\int_{\T^n} (\gd \varphi_j+p)\cdot \sigma\gd \eta_j d\norm{\theta} -\int_{\T^n} \gd \eta_j \cdot \sigma \gd\eta_j d\norm{\theta} \le \ep\ll \delta/2,
\]
as $j\rta\infty$ by the assumption \eqref{eq.approximatecorrectorequation}, which is a contradiction.
\end{proof}

\begin{proof}[Proof of Lemma \ref{l.whyquadratic}]
We first observe that for all real constant $\lambda$
\[
\begin{split}
         C_\theta(\lambda p)&=\inf_{\varphi\in C^\infty(\T^n)} C_\theta(\varphi,\lambda p)\\
         &=\lambda^2\inf_{\varphi\in C^\infty(\T^n)} C_\theta(\varphi/\lambda, p)\\
         &=\lambda^2 C_\theta(p).
\end{split}
\]  
It then suffice to show that
\[
(p_1,p_2)_\theta:= \frac{1}{2}\lb C_\theta(p_1+p_2)-C_\theta(p_1)-C_\theta(p_2)\rb
\]
is a bilinear form. By Lemma \ref{l.techtoshowquadratic}  
\[
    C_\theta(p_1+\lambda p_2+\mu p_3)=\lim_{j\rta\infty}C_\theta(\varphi_1^j+\lambda\varphi_2^j+\mu\varphi_3^j,p_1+\lambda p_2+ \mu p_3),
\]
where $\varphi_i^j$'s are minimizing sequences of $C_\theta(\cdot,p_i)$'s for each $i=1,2,3$. Now we know that
\[
\begin{split}
 (p_1,\lambda p_2+\mu p_3)_\theta+ o_j(1)&=\frac{1}{2} \lb C_\theta(\varphi_1^j+\lambda\varphi_2^j+\mu\varphi_3^j,p_1+\lambda p_2+ \mu p_3)-C_\theta(\varphi_1^j,p_1)\right.\\
 &\quad\left.-C_\theta(\lambda\varphi_2^j+\mu\varphi_3^j,\lambda p_2+\mu p_3) \rb \\
&= \int_{\T^n}(\gd\varphi_1^j +p_1)\cdot \sigma (\gd(\lambda\varphi_2^j+\mu\varphi_3^j)+\lambda p_2+ \mu p_3) d\norm{\theta} \\
&=\lambda \int_{\T^n}(\gd\varphi_1^j +p_1)\cdot \sigma (\gd\varphi_2^j+p_2) d\norm{\theta}  \\
&\quad+ \mu \int_{\T^n}(\gd\varphi_1^j +p_1)\cdot \sigma (\gd\varphi_3^j+p_3) d\norm{\theta}\\
&=\lambda (p_1,p_2)_\theta + \mu (p_1,p_3)_\theta + o_j(1).
\end{split}
\]
Sending $j\rta\infty$ on both sides, we obtain the bilinearity of $(p_1,p_2)_\theta$.

\end{proof}

\subsection{Submedium and super-additivity}\label{subsection.submedumandsuperadditivity}

In this subsection we discuss some basic behaviors of $Q(\theta)$ with respect to addition of media $\theta$'s. Let us start with the notion of submedia.

\begin{definition}\label{def.restrictionoftensorfolds}
    A \emph{submedium} of $\theta$ is a medium $\theta'$ such that for all Borel set $A\subset \T^n$
    \[
    \theta'(A)\le \theta(A).
    \]
We also denote as $\theta'\le \theta$. We further write $\theta'<\theta$ if $\theta'\ne \theta$. For a Borel set $E$, the \emph{restriction} of $\theta$ on $E$ is defined as 
    \[
    \theta_E(A) := \theta(E\cap A), \textup{ for every Borel set }A\subset\T^n.
    \]
\end{definition}
Notice that if ${\theta'}$ is a submedium of $\theta$ on $U$, then the subtraction
\[
\Tilde{\theta}:=\theta-{\theta'}
\]
also defines a submedium. We call $\Tilde{\theta}$ the \emph{complement} of $\theta'$ with respect to $\theta$. We are interested in how $Q(\theta)=Q(\theta'+\Tilde{\theta})$ behaves as we view $\theta$ as a sum of $\theta'$ and its complement. It turns out that we don't in general have additivity, but instead an inequality (See Lemma \ref{l.basicpropertiesofQ} \ref{condition.(3)} below). This non-additivity comes from the intersection of supports of each summand. Indeed, we can consider the sum of $\theta_1$ and $\theta_2$ defined on $\T^1$:
\[
d\theta_1:=1_{\{0\le s \le 1/2\}} ds \textup{ and }d\theta_2:=1_{\{1/2\le s \le 1\}} ds,
\]
where $s:[0,1]\rta \T^1$ is a unit speed parametrization. Notice that both 
\[
Q(\theta_1)=Q(\theta_2)=0,
\]
but $Q(\theta_1+\theta_2)=1>0$.

\begin{lemma}\label{l.basicpropertiesofQ}
We show the following basic properties of $Q(\theta)$.

\begin{enumerate}[label=(\arabic*)]
    \item\label{condition.(1)}  Let $\theta'$ be a submedium of $\theta$, then
    \[
    Q(\theta')\le Q(\theta) .
    \]
    \item\label{condition.(2)}  Suppose two media $\theta$ and $\theta'$ satisfy that $\spt{\theta}\cap\spt{\theta'}=\emptyset$, then
   \[
   Q(\theta+\theta')=Q(\theta)+Q(\theta).
   \]
    \item\label{condition.(3)} (Super-additivity) Given a sequence of media $\theta_i$ such that 
    \[
    \sum_{i=1}^\infty \norm{\theta_i}(\T^n)<\infty,
    \]
    then the following medium
    \[
    \theta_\infty:=\sum_{i=1}^\infty \theta_i
    \]
    is well-defined and we have the following super-additivity
    \[
    Q(\theta_\infty) \ge \sum_{i=1}^\infty Q(\theta_i).
    \]
    In particular, if all $\theta_i$ are maximal, then the equality holds and $\theta_\infty$ is also maximal.
\end{enumerate}
   
\end{lemma}

There is in fact no countable additivity even if the supports of the summand are disjoint from each other (See Example \ref{ex.failureofcountableadditivity}). Here we only present finite additivity under the disjointness assumption, but we will discuss in further details on such problems in Section \ref{section.nontrivialmedia}.

\begin{proof}

To prove property \ref{condition.(1)} we decompose $d\theta'=\sigma'\,d\norm{\theta'}$. For every $\varphi\in C^\infty(\T^n)$, we have
    \[
    \int_{\T^n}\lb\gd\varphi+p\rb\cdot \sigma'(x)\lb\gd\varphi+p\rb d\norm{\theta'}(x) \le  \int_{\T^n}\lb\gd\varphi+p\rb\cdot \sigma(x)\lb\gd\varphi+p\rb d\norm{\theta}(x).
    \]
    The inequality $ Q(\theta')\le Q(\theta)$ follows by taking infimum over all $\varphi\in C^\infty(\T^n)$.

 For property \ref{condition.(2)},  notice that the set $\spt{\theta'}$ has positive distance from $\spt{\theta}$, and both are compact by definition. This implies that 
the restriction of any test function $\varphi\in C^\infty(\T^n)$ on $\spt{\theta'}\cup \spt{\theta}$ can be represented as the restriction of $\varphi_1+\varphi_2$ with
$$
\varphi_1, \varphi_2 \in C^\infty(\T^n) \textup{ that satisfy }\spt{\theta}\subset \spt{\varphi_1}, \ \spt{\theta'}\subset \spt{\varphi_2}\textup{ and } \spt{\varphi_1}\cap \spt{\varphi_2}=\emptyset.
$$
This establishes the proof when one take the infimum over the test function $\varphi$.  

To show property \ref{condition.(3)}, we notice that $\sum_{i=1}^\infty \theta_i$ is a well-defined positive semi-definite matrix-valued Radon measure, and 
   \[
   \norm{\sum_{i=1}^\infty \theta_i}(\T^n)= \sum_{i=1}^\infty \norm{\theta_i}(\T^n)<\infty.
   \]
We write 
\[
d\theta_\infty:=\sigma_\infty\, d\norm{\theta_\infty} \textup{ and }d\theta_i:=\sigma_i\, d\norm{\theta_i} \textup{ for each }i\ge 1.
\]
To show the inequality we observe that for all $N>0$ and $\varphi\in C^\infty(\T^n)$
\[
\begin{split}
    \int_{\T^n} \lb \gd \varphi + p\rb\cdot{\sigma_\infty}\lb \gd \varphi + p\rb d\norm{\theta_\infty}&=\int_{\T^n}  \trace{\lb \gd \varphi + p\rb\otimes\lb \gd \varphi + p\rb d\theta_\infty}\\
    &=\sum_{i=1}^\infty \int_{\T^n}  \trace{\lb \gd \varphi + p\rb\otimes\lb \gd \varphi + p\rb d\theta_i}\\
    &\ge \sum_{i=1}^N \int_{\T^n}  \trace{\lb \gd \varphi + p\rb\otimes\lb \gd \varphi + p\rb d\theta_i}\\
    &=\sum_{i=1}^N\int_{\T^n} \lb \gd \varphi + p\rb\cdot \sigma_i\lb \gd \varphi + p\rb d \norm{\theta_i}.
\end{split}
\]
By taking the infimum over all smooth test functions $\varphi$ we obtain
\[
Q(\theta_\infty) \ge \sum_{i=1}^N Q(\theta_i).
\]
Sending $N\rta\infty$ establishes the inequality we desire. 

Suppose all $\theta_i$'s are maximal, then by the singular Wiener upper bound \eqref{eq.singularwienerbound}
\[
\sum_{i=1}^\infty Q(\theta_i)=\sum_{i=1}^\infty\theta_i(\T^n)=\theta_\infty(\T^n) \ge Q(\theta_\infty),
\]
which finishes the proof.

\end{proof}

\subsection{Weak* convergence, upper semi-continuity and faithful convergence}\label{subsection.faithfulconvergence}

In this subsection we study the continuity property of $Q(\theta)$ with respect to the weak* topology of media. Let us introduce the notion of weak* convergence.

\begin{definition}
    \label{def.weakstarconvergence}
    Say that a sequence of media $\theta_i$ {weakly* converges} to another medium $\theta_\infty$ if for any continuous function $\varphi\in C(\T^n)$ there is
    \[
    \lim_{i\rta\infty} \int_{\T^n} \varphi(x) d\theta_i(x) = \int_{\T^n} \varphi(x) d\theta_\infty(x).
    \]
    We use the notation $\theta_i \overset{*}{\weakcv} \theta_\infty$.
\end{definition}

The effective tensor $Q(\theta)$ is generally discontinuous with respect to the weak* topology of $\theta$ as above. Such phenomena occur often because of homogenization and have been studied in the context of $G$, $H$ and $\Gamma$-convergence theories \cites{DalMaso1993,MuratTartar1997,Braides2006,Defranceschi1993}. We are currently aware of at least two phenomena when the discontinuity occurs:
\begin{itemize}
    \item (Change of topology) When there is a significant topological difference between the supports of $\theta_\infty$ and $\theta_i$'s. To see this phenomenon, one can consider the medium 
    \[
     d\theta_\delta(s):= 1_{\{0\le s \le 1-\delta\}} ds, \textup{ for }s\in \T^1,
    \]
    where $s$ is a unit speed parametrization of $\T^1$ and $\delta >0$. Notice that for all $\delta>0$, $Q(\theta_\delta)=0$ but $Q(\theta_0)=1$. On the other hand, $\spt{\theta_\delta}$ does not contain any closed paths belonging to a nonzero homotopy class, but $\spt{\theta_0} = \T^1$ does contain one. 
    \item (Homogenization) When $\theta_i$ weakly* converge to a constant $\theta_\infty$, but $\theta_\infty$ does not faithfully represent the limit of the effective tensor. A very simple and well-known example is the 1-D periodic homogenization on $\T^1$ when $d\theta_\ep(s) = a( s/\ep) ds$. As $\ep\rta0^+$, the media $\theta_\ep\overset{*}{\weakcv} \avg{a}:=\int_0^1 a(s) ds$, but the effective tensor $Q(\theta_\ep) \rta \avg{a^{-1}}^{-1}$.
\end{itemize}

The characterization of discontinuities are left to future works. Although the discontinuity is ubiquitous we can still show that $Q(\theta)$ is upper semi-continuous.

\begin{lemma}\label{l.uppersemicontinuous}
    If $\theta_i$ weakly* converges to another medium $\theta_\infty$, then for any $p\in \R^n$ we have
    \[
    \limsup_{i\rta\infty} \ p\cdot Q(\theta_i)p \le p\cdot Q(\theta_\infty)p,
    \]
    and in particular
    \[
    \limsup_{i\rta\infty} M(\theta_i) \le M(\theta_\infty).
    \]
\end{lemma}

\begin{proof}
We write 
\[
d\theta_i = \sigma_i\, d\norm{\theta_i} \textup{ and }d\theta_\infty=\sigma_\infty\, d\norm{\theta_\infty}.
\]
We first show that for all $p\in \R^n$ the quantity $$p\cdot Q(\theta)p$$ is upper semi-continuous with respect to the weak* topology. Indeed, let $\psi$ be an arbitrary smooth test function, we have for every $i>0$
\[
p\cdot Q(\theta_i)p \le \int_{\T^n} (\gd \psi+p)\cdot \sigma_i(x)(\gd \psi+p) d\norm{\theta_i}(x)\overset{i\rta\infty}{\longrightarrow} \int_{\T^n} (\gd \psi+p)\cdot \sigma_\infty(x)(\gd \psi+p) d\norm{\theta_\infty},
\]
where the convergence is by applying the weak* convergence in Definition \ref{def.weakstarconvergence} component-wise. This implies that for every $\psi$
\[
\limsup_{i\rta\infty}p\cdot Q(\theta_i)p\le \int_{\T^n} (\gd \psi+p)\cdot \sigma_\infty(x)(\gd \psi+p) d\norm{\theta_\infty}.
\]
This completes the proof of the upper semi-continuity by taking an infimum in $\psi$. The upper semi-continuity of the mean conductance $M(\theta)$ follows immediately from the fact that
\[
M(\theta_i)=\frac{1}{n}\sum_{j=1}^n e_j \cdot Q(\theta_i) e_j,
\]
where $\{e_j\}_{j=1}^n$ forms an orthonormal basis of $\R^n$.
\end{proof}

One of the advantage of using singular matrix-valued measures is that sometimes they can reduce the dimension of the problem and give sharper insights. For example, when modeling the leaf vein networks (also read Example \ref{ex.nontrivialfaithfulconvergence} and Section \ref{section.leafvenation}), one can simply look at singularly supported measures instead of matrix fields that have high conductance in an $\ep$-tubular neighborhood of the network. An important aspect of this approach is that one has to check whether the singularization procedure faithfully indicate the conductive properties of the original nonsingular setup. To that end we introduce the following stronger notion of convergence.

\begin{definition}
      \label{def.faithfulconvergence}
Say that a sequence of media $\theta_i$ \emph{faithfully converge} to $\theta_\infty$ if $\theta_i\overset{*}{\weakcv} \theta_\infty$ and
\[
\lim_{i\rta\infty} Q(\theta_i) = Q(\theta_\infty).
\]
\end{definition}

Despite the general discontinuity of $Q$, we show a principle of faithful continuity when the limit medium is a submedium of all the media in the sequence.

\begin{lemma}\label{l.faithfulconvergencewhenshrinking}
    Suppose $\theta_i$ is a sequence of media weakly* converging to $\theta_\infty$, and for all $i\ge 1$ we have $\theta_\infty\le \theta_i$, then $\theta_i$ converge to $\theta_\infty$ faithfully.
\end{lemma}

\begin{example}
   \label{ex.nontrivialfaithfulconvergence}
In \cite{Huang2021}, the author proved an example of faithful convergence that is outside the assumptions in Lemma \ref{l.faithfulconvergencewhenshrinking}. Let $\Gamma\subset\T^2$ be a finite union of $C^2$ curves, having no tangential cusps at joining nodes. Define $\displaystyle\Gamma_\delta:=\bigcup_{x\in \Gamma}B_{\delta/2}(x)$ and
    \[
d\theta_\delta(x):=\lb1+ \frac{1_{\Gamma_\delta}}{\delta} \rb I_{2\times 2}\, d\Lme^2, \textup{ where }d\Lme^2 \textup{ is the Lebesgue measure on $\T^2$,}
    \]
    then $\theta_\delta$ converges faithfully to $I_{2\times 2}\,d\Lme^2 + I_{2\times 2}\,d\restr{\Ha^1}{\Gamma}$ as $\delta\rta0^+$. Notice that $I_{2\times 2}\,d\Lme^2 + I_{2\times 2}\,d\restr{\Ha^1}{\Gamma}$ is not a submedium of $\theta_\delta$ for any $\delta>0$.
\end{example}

\begin{proof}[Proof of Lemma \ref{l.faithfulconvergencewhenshrinking}]
It suffices to show that 
\[
\lim_{i\rta\infty} Q(\theta_i) = Q(\theta_\infty).
\]
By Lemma \ref{l.uppersemicontinuous}, we know that 
\[
\limsup_{i\rta\infty} Q(\theta_i) \le Q(\theta_\infty),
\]
and then it suffices to show the reverse inequality. Indeed, because $\theta_\infty \le \theta_i$ for all $i\ge1$, we have by Lemma \ref{l.basicpropertiesofQ} \ref{condition.(1)}, the following inequalities
\[
Q(\theta_i) \ge Q(\theta_\infty) \textup{ for all }i\ge 1. 
\]
This finishes the proof.

\end{proof}

\subsection{Efficient submedium}
\label{subsection.efficientsubmedium}
There is often redundancy in a medium from the viewpoint of its effective tensor, especially when the medium is singularly supported. To talk about such properties we define the {efficient part} of a medium.

\begin{definition}
    \label{def.efficientpart}
   Call a submedium $\theta_*$ to be \emph{efficient} in $\theta$ if $Q(\theta_*)=Q(\theta)$. Call $\theta_*$ to be \emph{minimally efficient} in $\theta$ if for any $\theta'\le \theta_*$ satisfying $Q(\theta')=Q(\theta_*)=Q(\theta)$, then $\theta'=\theta_*$. Call $\theta$ \emph{saturated} if it is equal to all its efficient submedium.
\end{definition}

We are aware of at least two types of proper efficient submedia.

\begin{itemize}
    \item (Dimension reduction) The medium $I_{2\times 2}\,d\Lme^2 + I_{2\times 2}\,d\restr{\Ha^1}{\Gamma}(x)$ in Example \ref{ex.nontrivialfaithfulconvergence} has proper efficient submedium, where $\Gamma$ is the union of finitely many $C^2$ curves with nondegenerate joining angles at nodes. Indeed, one can check that if at some $x_0\in\Gamma$ there is a $r_0>0$ such that $\Gamma\cap B_{r_0}(x_0)$ is a $C^2$ curve, then 
    $$
    I_{2\times 2}\,d\Lme^2 + P_{\tau_x}\,d\restr{\Ha^1}{\Gamma\cap B_{r_0}(x_0)}(x)+I_{2\times 2}\,d\restr{\Ha^1}{\Gamma\setminus B_{r_0}(x_0)}(x)
    $$ 
    is a proper efficient submedium, with $P_{\tau_x}$ the orthogonal projection from $\R^2$ to the 1-D tangent space $\tau_x$ of $\Gamma$ at $x\in \Gamma \cap B_{r_0}(x_0)$. This can be done by modifying the smooth test functions in \eqref{eq.effectivetensor}, which eliminates the normal contributions without costs in the tangential direction. As a consequence $ I_{2\times 2}\,d\Lme^2 + P_{\tau_x}\,d\restr{\Ha^1}{\Gamma}(x)$ is a proper efficient submedium if $\Gamma$ is, outside a finite set of singular points, the union of finitely many $C^2$ curves.
    \item (Trivial components) Suppose for a medium $\theta$ there is a component $E\subset\spt{\theta}$ such that the restriction $\theta_E$ is trivial, then by Lemma \ref{l.additivityforcomponentdecomposition}, $\theta-\theta_E$ is a proper efficient submedium of $\theta$.
\end{itemize}

Let us show the existence of minimally efficient parts for nontrivial media. We leave the characterization of minimally efficient submedia of a given medium (or equivalently the characterization of saturated media) to future works.

\begin{lemma}
    \label{l.efficientpart}
Nontrivial media always have nonzero minimal efficient submedium.
\end{lemma}

\begin{proof}
    Consider a nontrivial medium $\theta$ and the following set of submedia
    \[
   L:= \{\theta'\le\theta  \ ; \ Q(\theta')=Q(\theta)\}.
    \]
We aim to solve the minimization problem
\be\label{eq.minmizingmasstofindefficientpart}
\min_{\theta'\in L} \trace{\theta'(\T^n)}.
\ee
Suppose $\theta_i$ is a minimizing sequence, then by weak* compactness, we can find a converging subsequence, not relabeled, such that
\[
\theta_i\overset{*}{\weakcv} \theta_\infty.
\]
By the upper semi-continuity Lemma \ref{l.uppersemicontinuous}, we know that 
\[
Q(\theta_\infty) \ge \limsup_{i\rta\infty} Q(\theta_i)=Q(\theta).
\]
Notice that $\theta_\infty\le \theta$, and then we have $\theta_\infty\in L$. We also know that $\theta_\infty\ne0$ because $\theta$ is nontrivial.

Notice that $\theta_\infty$ is minimally efficient because if there is another $\Tilde{\theta}<\theta_\infty$ and $\Tilde{\theta}\in L$, then the fact that $\theta_\infty$ being a minimizer of \eqref{eq.minmizingmasstofindefficientpart} is violated.

\end{proof}

\subsection{The Euler-Lagrange equation for maximal media}
\label{subsection.elequation}
In this subsection we present the Euler-Lagrange equation for maximal media.

\begin{lemma}\label{l.equationformaximalmedia}
    A $d\theta=\sigma d\norm{\theta}$ is maximal if and only if for every $\Phi\in C^\infty(\T^n ; \R^n)$
    \be\label{eq.euequation}
    \int_{\T^n} \trace{\sigma(x)\gd\Phi(x)} d\norm{\theta}(x)=0.
    \ee
\end{lemma}

\begin{proof}
{We first recall that $\theta$ is maximal if and only if $Q(\theta)=\theta(\T^n)$. As both $Q(\theta)$ and $\theta(\T^n)$ are positive semi-definite, it is not difficult to find that $\theta$ is maximal if and only if its mean conductance (average eigenvalue of $Q(\theta)$) reaches the upper Wiener bound \eqref{eq.singularwienerboundformeanconductance}}
\be\label{eq.Mequalupperbound}
M(\theta)= \frac{1}{n} \trace{Q(\theta)}= \frac{1}{n} \norm{\theta}(\T^n).
\ee
Notice that if we write $\Phi=\sum_{i=1}^n\phi^ie_i$ with $e_i$'s being an orthonormal basis for $\R^n$
\be\label{eq.otherformofMthetha}
\begin{split}
    M(\theta)&=\frac{1}{n}\trace{Q(\theta)}\\
    &=\frac{1}{n}\sum_{i=1}^n e_i\cdot Q(\theta)e_i\\
    &=\frac{1}{n}\sum_{i=1}^n\inf_{\phi^i\in C^\infty(\T^n)} \int_{\T^n} (\gd \phi^i(x)+e_i)\cdot \sigma(x)(\gd \phi^i(x)+e_i) d\norm{\theta}(x)\\
    &=\frac{1}{n} \inf_{\Phi\in C^\infty(\T^n; \R^n)} \int_{\T^n} \trace{\lb\gd \Phi+I\rb^T\sigma(x)\lb\gd \Phi+I\rb} d\norm{\theta}(x).
\end{split}
\ee
Therefore {by applying \eqref{eq.Mequalupperbound}} the medium $\theta$ is maximal if and only if for all $\Phi\in C^\infty(\T^n; \R^n)$
\[
\int_{\T^n} \trace{\lb\gd \Phi+I\rb^T\sigma(x)\lb\gd \Phi+I\rb} d\norm{\theta}(x) \ge \norm{\theta}(\T^n).
\]
Simplifying the inequality we get
\[
\int_{\T^n} \trace{\gd \Phi^T\sigma(x)\gd \Phi} d\norm{\theta}(x) + 2 \int_{\T^n} \trace{\sigma(x)\gd \Phi} d\norm{\theta}(x) \ge 0.
\]
Replacing $\Phi$ by $h\Phi$ with $|h|>0$ we see that
\[
|h|\int_{\T^n} \trace{\gd \Phi^T\sigma(x)\gd \Phi} d\norm{\theta}(x) + 2 \frac{h}{|h|}\int_{\T^n} \trace{\sigma(x)\gd \Phi} d\norm{\theta}(x) \ge 0,
\]
which shows \eqref{eq.euequation} after sending $h\rta0$ from both sides. The reverse follows immediately when we plug \eqref{eq.euequation} into the last formula in \eqref{eq.otherformofMthetha}.

\end{proof}

\section{Maximal medium, stationary varifold and dimension bound}
\label{section.maximalmedium}

In this section we characterize maximal media by applying some results from geometric measure theory, especially the topic on varifolds and PDE for measures. Specifically we discuss Theorem \ref{t.forma2} and \eqref{eq.formalequivalenceintro}, which states the following formal identity
\be\label{eq.formalequality}
{\textup{Conductance Maximality}} \quad = \quad \textup{Area Criticality}.
\ee
On the left of the identity we mean the media that achieve the upper Wiener bound \eqref{eq.singularwienerbound}, and on the right we mean stationary varifolds, which are known to be critical points of the generalized area functional (See Section \ref{subsubsection.varifolds} for more details). Based on this identity we prove a pointwise dimension bound for maximal medium (See Theorem \ref{t.realizabledimensionandlocaldimension}).

At first glance it is far from being clear why the equality \eqref{eq.formalequality} makes sense. We explain this by introducing a mapping, first introduced in \cite{ambrosio1997}*{Remark 3.2}, from the space of varifolds to the space of media. From this transformation we view medium, a matrix-valued measure, as a varifold having \emph{variable} and \emph{fractional} dimension (See Lemma \ref{l.realizationmapcontinuity} and Definition \ref{def.realizabledimensionformedia}). Specifically we define the mapping $\Ta$ as follows: First disintegrate a $k$-varifold $\mu$, by Theorem \ref{t.disintegrationtheorem}
\[
d\mu(x,\tau) = d\rho_x(\tau) d\norm{\mu}(x), 
\]
where $\rho_x$ is a family of probability measures on the Grassmannian manifold $G(k,n)$ and $\norm{\mu}=\pi_\# \mu$ is the pushforward of $\mu$ with respect to the projection $\pi(x,\tau)=x$. Second, we define the matrix field
\be\label{eq.transformationofmatrixfield}
\sigma_\mu(x):=\frac{1}{k}\int_{G(k,n)} P_\tau d\rho_x(\tau),
\ee
where $P_\tau$ is the orthogonal projection matrix to the $k$-dimensional subspace $\tau\subset\R^n$. We then define $\Ta(\mu)$ as the following medium
\be\label{eq.realizationmap}
d\Ta(\mu):= \sigma_\mu d\norm{\mu}.
\ee
Notice that \(\trace{\sigma_\mu}=1 \textup{ and }\norm{\Ta(\mu)}=\norm{\mu}.\)

As we have seen in Lemma \ref{l.equationformaximalmedia} (See also the discussions in \cite{matsci1}*{Section 6}), the maximization of conductance is essentially solving the distributional PDE
\be\label{eq.formaldivergencefreematrixfield}
\gd \cdot  \theta  =0
\ee
under various conditions (two-phase, anisotropic, etc.) on the medium $\theta$. If $\theta = \Ta(\mu)$ for some varifold $\mu$, then plug this into \eqref{eq.formaldivergencefreematrixfield} we obtain
\be
0=\gd\cdot \lb \sigma_\mu(x) d\norm{\mu}(x) \rb = \frac{1}{k}\gd\cdot \lb \int_{G(k,n)} P_\tau d\rho_x(\tau) d\norm{\mu}(x) \rb ,
\ee
which is equivalently saying that $\mu$ is a stationary varifold (See Section \ref{subsubsection.varifolds}). This proves the second part of the following rigorous version of \eqref{eq.formalequality}.

\begin{theorem}\label{t.equivalence}
The mapping $\Ta$ as defined in \eqref{eq.realizationmap} is surjective and continuous with respect to the weak* topology of media and varifolds. Moreover, the following statements are equivalent for a fixed medium $\theta$:
\begin{enumerate}[label=(\alph*)]
    \item \label{condition.(a)} The medium $\theta$ is maximal.
     \item \label{condition.(a2)} The medium $\theta$ satisfies $\gd\cdot\theta=0$ in the distributional sense (See Lemma \ref{l.equationformaximalmedia}).
    \item \label{condition.(b)} All varifold realizations $\mu\in \Ta^{-1}(\theta)$ are stationary.
    \item \label{condition.(c)} There exists a stationary varifold $\mu\in \Ta^{-1}(\theta)$.
\end{enumerate}
\end{theorem}

\begin{proof}
    The proof of Theorem \ref{t.equivalence} is complete by Lemma \ref{l.realizationmapcontinuity} below, where we provide a more precise discussion on the surjectivity and continuity of the mapping $\Ta$.

\end{proof}

Our main contribution in this section is a pointwise dimension bound for maximal media. As we have pointed out in the introduction, for a medium in $\T^3$ of the form
\[
d\theta:= \sigma \, d\restr{\Ha^2}{\T^2\times\{0\}},
\]
where $\sigma$ is a constant positive semi-definite matrix, the rank of $\sigma$ satisfies the following bound when $\theta$ is maximal
\[
\textup{rank}(\sigma) \le \dim_{\Ha}\lb\T^2\times\{0\} \rb =2.
\]
More generally a \emph{global version} of such dimension bound was established in \cites{meas3} for $\mathcal{A}$-free measures, including maximal medium. At a heuristic level, the theory, in terms of maximal medium, states that the rank of the matrix field $\sigma=\frac{d\theta}{d\norm{\theta}}$ does not exceed the dimension of $\norm{\theta}$ (See \cite{meas3}*{Corollary 1.4 and Proposition 3.1} for more precise discussions). In this case, we call the rank of the matrix field $\sigma$ is bounded by the dimension of $\norm{\theta}$. 

To establish a \emph{pointwise} dimension bound, the compromise is that one should use a new slightly weaker notion of ``rank''. We define the new ``rank'' as follows.

\begin{definition}[Realizable dimension]
    \label{def.realizabledimension}
Given a positive semi-definite matrix $A\in \R^{n\times n}$, we define its \emph{realizable dimension} as
\[
\dimr(A):=\frac{\trace{A}}{\lambda_{\max} (A)},
\]
where $\lambda_{\max}(A)$ is the maximal eigenvalue of $A$.

\end{definition}

Note that $\dimr(A)\le \textup{rank}(A)$ with equality if and only if $A$ is an orthogonal projection matrix. The interest of this new notion is that it sharply characterizes whether a medium $\theta$ is in the range of $k$-varifolds under the transformation $\Ta$ (See Section \ref{subsection.varifoldrealization}). 

To state the dimension bound result more precisely, we define two notions of dimensions. The first is the semi-continuous version of realizable dimensions of the anisotropy matrix fields $\sigma$ in media.

\begin{definition}
    \label{def.realizabledimensionformedia}
    For a medium $d\theta=\sigma\, d\norm{\theta}$, we define its \emph{lower realizable dimension} at $x\in \spt{\theta}$ as
\be
\label{eq.lowerrealizabledimension}
\dimrlow(\theta)(x):=\sup_{\delta>0}\essinf\{\dimr(\sigma(y))\ ; \ |y-x|\le \delta,\, y\in \spt{\theta}\},
\ee
and symmetrically the \emph{upper realizable dimension}
\be
\label{eq.uprealizabledimension}
\dimrup(\theta)(x):=\inf_{\delta>0}\esssup\{\dimr(\sigma(y))\ ; \ |y-x|\le \delta,\, y\in \spt{\theta}\}.
\ee
Here we take \say{ess\,sup} and \say{ess\,inf} with respect to the Radon measure $\norm{\theta}$. Define the \emph{realizable dimension} of $\theta$ as
\be
\dimr(\theta)(x)=\dimrlow(\theta)(x)=\dimrup(\theta)(x)
\ee
if the equality holds. Note that $\dimrlow(\theta)(x)\ne \dimr(\sigma(x))$ and $\dimrup(\theta)(x)\ne \dimr(\sigma(x))$ in general.
\end{definition}

Besides the realizable dimension, we also consider the {local dimensions} from the viewpoints of fractal geometry. The following definition of local dimension can be found in \cites{Tricot_1982,Young_1982} and the book \cite{Bishop_Peres_2016}.

\begin{definition}\label{def.localdimension}
    We define the \emph{lower local dimension} of a medium $\theta$ as
    \[
    \dimlow (\theta)(x):=\liminf_{r\rta0^+} \frac{\log \norm{\theta}(B_r(x))}{\log r},
    \]
    and similarly the \emph{upper local dimension}
    \[
    \dimup (\theta)(x):=\limsup_{r\rta0^+} \frac{\log \norm{\theta}(B_r(x))}{\log r}.
    \]
    The \emph{local dimension} of $\theta$ at $x$ is defined as
    \[
    \dim_{\textup{loc}}(\theta)(x):=\lim_{r\rta0^+} \frac{\log \norm{\theta}(B_r(x))}{\log r},
    \]
    if the limit exists.
\end{definition}

The main theorem of this section describes the relation between the above two dimensions in general.

\begin{theorem}\label{t.realizabledimensionandlocaldimension}
    Suppose a medium $\theta$ is maximal, then for all $x\in \spt{\theta}$ we have
    \be\label{eq.dimensionboundtheorem}
     \dimrlow(\theta)(x) \le \dimlow (\theta)(x).
    \ee
   In particular, the lower local dimension $\dimlow (\theta)(x) \ge 1$ for all $x\in \spt{\theta}$.
\end{theorem}

The proof of this theorem does not depend on the theory as used in \cite{meas3}, but simply a \emph{fractional} version of the classical monotonicity formula for varifolds (See Lemma \ref{l.fracmonotonicityformula}). We do not know if \eqref{eq.dimensionboundtheorem} still hold when the realizable dimension is replaced by the standard rank.

Let us outline the structure of this section. In Section \ref{subsection.varifoldrealization} we characterize the media that can be realized by $k$-varifold for some $1\le k\le n$. In Section \ref{subsection.fractionalmonotonicityfomular}, we prove Theorem \ref{t.realizabledimensionandlocaldimension} by introducing a fractional version of monotonicity formula. We also present some examples to show the sharpness of the bound \eqref{eq.dimensionboundtheorem}. In Section \ref{subsection.applicationofstationaryvarifolds} we show some applications of existing theorems on stationary varifolds to maximal media.  In Section \ref{subsection.openquestion} we discuss Problem \ref{prob.intromaxvalency} in rigorous mathematical terms.

\subsection{Varifold realizations of media}\label{subsection.varifoldrealization}

In this subsection we show some basic properties of the transformation $\Ta$ as defined in \eqref{eq.realizationmap}. Let us first denote $\Ha$ the space of all varifolds on $\T^n$, and $\Fa$ the space of all media. We denote for each integer $1\le k \le n$ the space $\Ha_k$ of all $k$-varifolds and therefore we have $\displaystyle\Ha=\bigcup_{k=1}^n \Ha_k$.

\begin{definition}[Realization map]
    Call the mapping $\Ta:\Ha\rta \Fa$ as defined in \eqref{eq.realizationmap} the \emph{realization map}. For each $1\le k \le n$, call the restriction $\Ta_k:=\restr{\Ta}{\Ha_k}$ the $k$-th realization map.
\end{definition}

\begin{lemma}\label{l.realizationmapcontinuity}
For each $1\le k \le n$, the $k$-th realization map $\Ta_k$ is continuous with respect to the weak* topology on $\Ha_k$ and $\Fa$. A medium $\theta\in \Fa$ is in the range of $\Ta_k$ if and only if the matrix field $\sigma=\frac{d\theta}{d\norm{\theta}}$ satisfies
\[
\dimr(\sigma(x)) \ge k \textup{ for } \norm{\theta}\textup{-almost all }x\in \T^n.
\]
In particular, the first realization map $\Ta_1:\Ha_1\rta\Fa$ is surjective. 
\end{lemma}

To prove the lemma we observe that a medium $d\theta=\sigma \, d\norm{\theta}$ is in the range of $\Ta_k$ if and only if the matrix field $\sigma(x)$ can be written in the form \eqref{eq.transformationofmatrixfield} for some Borel choice in $x$ of probability measures $\rho_x$ on the Grassmannian manifold $G(k,n)$. In the following lemma we prove a sharp criterion for the existence of such representation for fixed $x$.

\begin{lemma}\label{l.realizablematrix}
    A positive semi-definite matrix $A\in \R^{n\times n}$ that satisfies $\trace{A}=1$ takes the form
    \be\label{eq.integralformofsigma}
    A=\frac{1}{k}\int_{G(k,n)} P_\tau d\rho(\tau)
    \ee
   where $P_\tau$ is the orthogonal projection to the $k$-dimensional subspace $\tau$ and $\rho$ is a probability measure on $G(k,n)$ if and only if the eigenvalues $\{\lambda_i\}_{i=1}^n$ of $A$ satisfy
    \[
    0\le \lambda_i \le 1/k , \textup{ for all }i=1,\dots,n.
    \]
    In particular, $A$ can always be written in the form \eqref{eq.integralformofsigma} when $k=1$.
\end{lemma}

\begin{proof}
It is not difficult to see the only if part because
\[
\trace{\int_{G(k,n)} \frac{1}{k}P_\tau d\rho(\tau)}=\int_{G(k,n)} \frac{1}{k}\trace{P_\tau} d\rho(\tau)=1,
\]
and for any unit vector $p$, we have
\[
\int_{G(k,n)} \frac{1}{k}|P_\tau(p)|^2 d\rho(\tau)\le 1/k.
\]
To show the reverse, we first rotate the coordinate and reduce to the case when $A=\textup{diag}(\lambda_1,\dots,\lambda_n)$ is a diagonal matrix. Observe that the set of $A$ that are of the integral form \eqref{eq.integralformofsigma} is convex. Therefore, by applying the Krein–Milman theorem, it suffices to show that the extreme points of the set 
$$
F:=\left\{0\le \lambda_i\le 1/k,\sum_{i=1}^n\lambda_i=1\right\}
$$
are all of the integral form \eqref{eq.integralformofsigma}. Let us first determine the extreme points of $F$. Notice that the extreme points of $F$ are also extreme points of the square $\{0\le \lambda_i\le 1/k\}$, which are of the form $\lambda_i=0$ or $1/k$. Because of the constraint $\sum_{i=1}^n\lambda_i=1$, we thus find that the extreme points of $F$ are permutations of
\[
\lambda_1=1/k,\dots,\lambda_k=1/k,\lambda_{k+1}=0,\dots,\lambda_n=0.
\]
On the other hand, each $k$-subset $J$ of $\{1,\dots,n\}$ such that $\lambda_i=1/k$ for $i\in J$ correspond to a $k$-dimensional space $V_J\subset\R^n$ spanned by the selection $\{e_i\}_{i\in J}$ of the orthonormal basis $\{e_i\}_{i=1}^n$. Therefore, each extreme point $A=\textup{diag}(\lambda_1,\dots,\lambda_n)$ takes the form
\[
A=\int_{G(k,n)} \frac{1}{k}P_\tau d\delta_{V_J}(\tau),
\]
for some $k$-subset $J$ of $\{1,\dots,n\}$.
  
\end{proof}

Denote $\mathcal{P}_k$ the space of probability measures on $G(k,n)$. Note that $\mathcal{P}_k$ is a Polish space under the weak topology, which coincides with the dual space $C_0^*(G(k,n))$ as $G(k,n)$ is compact \cite{Gaans_probability_measures}*{Section 2}. By Lemma \ref{l.compactnessofweakstar}, $\mathcal{P}_k$ is also compact with respect to the weak* topology.

\begin{proof}[Proof of Lemma \ref{l.realizationmapcontinuity}]
To show the realizability assertion, we first denote the full $\norm{\theta}$-measure set $$F:=\{x\in \spt{\theta}\ ; \ \dimr(\sigma(x)) \ge k\}$$ and consider
\[
K:=\lma(x,\rho)\in F\times \mathcal{P}_k\ ; \ \sigma(x)=\frac{1}{k}\int_{G(k,n)}P_\tau d\rho(\tau)\rma \subset F\times \mathcal{P}_k.
\]
Because the relation $\sigma(x)=\frac{1}{k}\int_{G(k,n)}P_\tau d\rho(\tau)$ is Borel in $x$ and continuous in $\rho$, the set $K$ is Borel. By Lemma \ref{l.realizablematrix}, for each $x\in F$ the slice 
$$
K_x:=\lma\rho \ ; \ \sigma(x)=\frac{1}{k}\int_{G(k,n)}P_\tau d\rho(\tau)\rma
$$
is compact and nonempty in $\mathcal{P}_k$. By \cite{Kechris1995}*{Theorem 28.8} and the Kuratowski-Ryll-Nardzewski measurable selection theorem \cites{KuratowskiRyllNardzewski1965}, there is a Borel selector $R:F\rta \mathcal{P}_k$ such that $R(x)\in K_x$ for all $x\in F$. This establishes the existence of a $k$-varifold in $\Ta^{-1}(\theta)$.

To show the continuity of the restriction $\Ta_k:\Ha_k\rta\Fa$, we first observe that for every continuous test function $\varphi$
\[
\psi(x,\tau):=\frac{\varphi(x)}{k} P_\tau 
\]
is continuous on $\T^n\times G(k,n)$. Suppose $\mu_l$ is a weakly* convergent sequence of $k$-varifolds with limit $\mu_\infty$, then if we write $d\Ta(\mu_l)=\sigma_l d\norm{\mu_l}$ and $d\Ta(\mu_\infty)=\sigma_\infty d\norm{\mu_\infty}$, we have as $l\rta\infty$
\[
\begin{split}
    \int_{\T^n} \varphi(x) \sigma_l(x) d\norm{\mu_l}(x)&=\int_{\T^n\times G(k,n)} \psi(x,\tau) d\mu_l(x,\tau)\\
    &\rta\int_{\T^n\times G(k,n)} \psi(x,\tau) d\mu_\infty(x,\tau)\\
    &=\int_{\T^n} \varphi(x) \sigma_\infty(x) d\norm{\mu_\infty}(x).
\end{split}
\]
As $\varphi\in C(\T^n)$ is arbitrary, we know that $\Ta(\mu_l)$ weakly* converges to $\Ta(\mu_\infty)$ in the sense of Definition \ref{def.weakstarconvergence}.

\end{proof}

\subsection{Dimensions and fractional monotonicity formula}
\label{subsection.fractionalmonotonicityfomular}

In this subsection we prove Theorem \ref{t.realizabledimensionandlocaldimension}, which is a fine estimate of the dimensions of a maximal medium.

The proof requires a \emph{fractional} monotonicity formula (See Lemma \ref{l.fracmonotonicityformula}) that is similar to that of a stationary varifold as we have seen in \eqref{eq.monotonicityformulaforstationaryvarifold}. Such monotonicity formula was first designed to study the fine structures of stationary varifolds \cites{lsimon,de2012allard,Allard1972}. In the context of medium the monotonicity formula gives pointwise dimension bound \eqref{eq.dimensionboundtheorem}.

Before presenting the monotonicity formula and the proof of Theorem \ref{t.realizabledimensionandlocaldimension}, let us first discuss the sharpness of the inequality \eqref{eq.dimensionboundtheorem} in the following two examples.
 
\begin{example}\label{ex.dimensionboundnotanidentity}
  The inequality \eqref{eq.dimensionboundtheorem} is in general not an {equality}.  Let $\theta$ be a medium on $\T^2$ defined as
    \[
    d\theta(x):= \lb(1-\lambda) e_1\otimes e_1 + \lambda e_2\otimes e_2\rb d\Lme^2(x),
    \]
    where $0\le \lambda\le1$, $\Lme^2$ is the Lebesgue measure on $\T^2$, $e_1$ and $e_2$ form an orthonormal basis for $\R^2$. {For each $0\le \lambda\le 1$ the medium is a maximal medium because for any $\Phi\in C^\infty(\T^2;\R^2)$ we have by integraion by parts
    \[
    \begin{split}
         \int_{\T^2} \trace{\lb(1-\lambda) e_1\otimes e_1 + \lambda e_2\otimes e_2\rb\gd \Phi} d\Lme^2(x) &= 
   \int_{\T^1} \int_{\T^1}(1-\lambda)\pt_1 \Phi_1(x_1,x_2) +\lambda \pt_2 \Phi_2(x_1,x_2) dx_1dx_2\\
   &=0.
    \end{split}
    \]
    The local dimension of this medium coincides with the Lebesgue measure and is always 2.} However, its realizable dimension is 
    \[
   \dimr(\theta)\equiv \dimr\lb(1-\lambda) e_1\otimes e_1 + \lambda e_2\otimes e_2\rb=\frac{1}{\max\{\lambda,1-\lambda\}},
    \]
    which can be any real number in the interval $[1,2]$.
\end{example}

\begin{example}\label{ex.upperrealdimstrictlygreaterthanlocaldim}
   The upper realizable dimension can be strictly greater than the local dimension. This example also shows that the upper realizable dimension does not capture the lower dimensional structures in a medium. To see this, let $\theta=\theta_1+\theta_2$ be the sum of two maximal media on $\T^2$ that are of the form
    \[
    d\theta_1(x):= \frac{1}{2}I_{2\times2} \, d\Lme^2(x),
    \]
    and 
    \[
    d\theta_2(x)=e_1\otimes e_1d\restr{\Ha^1}{\{x_2=0\}}(x).
    \]
    By Lemma \ref{l.basicpropertiesofQ}, we know that $\theta$ is maximal and the local dimension of $\theta$ on $\{x_2=0\}$ is exactly 1. On the other hand, we know that the matrix field $\sigma=\frac{d\theta}{d\norm{\theta}}$ satisfies
    \[
    \sigma(x)=\bca
    \frac{1}{2}I, & x_2\ne0,\\
    e_1\otimes e_1, & x_2 = 0,
    \eca
    \]
    which has upper realizable dimension 2 everywhere on $\T^2$. Notice that $\sigma(x)$ has exactly lower realizable dimension 1 on $\{x_2=0\}$, while 2 elsewhere.
\end{example}

Let us now present the fractional version of monotonicity formula.

\begin{lemma}\label{l.fracmonotonicityformula}
    Given a medium $d\theta=\sigma\, d\norm{\theta}$. If $\theta$ is maximal and $x_0\in \spt{\theta}$, then the following quantity
    \be\label{eq.quotientinmonotone}
    \frac{\norm{\theta}(B_r(x_0))}{r^\alpha}
    \ee
    is monotone nondecreasing in $0<r<1/2$ for $\alpha=1$. If $\kappa=\dimrlow (\theta)(x_0)>1$, then for all $\alpha\in [1,\kappa)$ there exists $r_\alpha>0$ such that for all $0<r<r_\alpha$, the quotient \eqref{eq.quotientinmonotone} is monotone nondecreasing. If $\dimrlow (\theta)(y) \ge m$ for some constant $m>1$ and all $y\in U\cap \spt{\theta}$ in an open neighborhood $U\ni x_0$ then \eqref{eq.quotientinmonotone} is monotone nondecreasing for $\alpha=m$. In particular, this gives the standard monotoncity formula when $m$ is an integer.
\end{lemma}

\begin{proof}
We may without loss assume that $x_0=0$ and apply a special $\Phi\in C^\infty(\T^n;\R^n)$ to the E-L equation \eqref{eq.euequation}. To this end we consider for small $\delta>0$ an auxiliary function $\eta=\eta_{r,\delta}\in C_0^\infty([0,\infty))$ that satisfies $\eta(s)=1$ for $0\le s \le r-\delta$, $\eta(s)=0$ for $s\ge r$ and $\eta$ is strictly decreasing on $r-\delta< s <r$. By plugging $\Phi(x)=x\eta(|x|)$ into \eqref{eq.euequation} and because $\trace{\sigma}=1$ we obtain
\[
\begin{split}
    \int_{B_r} \eta d\norm{\theta}(x)&=-\int_{B_r}\lb\sigma(x)x \rb\cdot \gd \lb \eta(|x|)\rb d\norm{\theta}(x)\\
&=-\int_{B_r} \frac{x\cdot\sigma(x)x}{|x|} \eta'(|x|) d\norm{\theta}(x).
\end{split}
\]
By definition of realizable dimension, we know that 
\be\label{eq.intermediateinequalityinprovingfractionalmono}
\begin{split}
    \int_{B_r}  \eta d\norm{\theta}(x)&=-\int_{B_r} \frac{x\cdot\sigma(x)x}{|x|} \eta'(|x|) d\norm{\theta}(x)\\
&\le r\int_{B_r} \frac{\trace{\sigma(x)}}{\dim_\textup{r}(\sigma(x))}  |\eta'|(|x|) d\norm{\theta}(x).
\end{split}
\ee
Because we always have $\dimrlow (\theta)(x) \ge1$
\be\label{eq.monotonwehn1d}
 \int_{B_r}  \eta d\norm{\theta}(x) \le r\int_{B_r}   |\eta'|(|x|) d\norm{\theta}(x).
\ee
If we write 
$$
J_\delta(r) = \int_{B_r} \eta d\norm{\theta}(x),
$$
the inequality \eqref{eq.monotonwehn1d} implies that
\[
J_\delta(r)\le r J_\delta'(r),
\]
which shows that 
\[
\frac{J_\delta(r)}{r}
\]
is monotone nondecreasing in $r>0$. Sending $\delta\rta0$, we know that $J_\delta/r$ converges monotonically to \(\frac{\norm{\theta}(B_r)}{r}\). This shows that $\frac{\norm{\theta}(B_r)}{r}$ is monotone nondecreasing in $0<r<1/2$.

Let us now discuss the case when $\kappa=\dimrlow(\theta)(0)>1$. By the definition of $\dimrlow(\theta)(x)$ (See Definition \ref{def.realizabledimension}) and the inequality \eqref{eq.intermediateinequalityinprovingfractionalmono}, we know that for every $\alpha\in[1,\kappa)$ there is an $r_\alpha>0$ such that for all $0<r<r_\alpha$ we have
\[
\int_{B_r}  \eta d\norm{\theta}(x) \le \frac{r}{\alpha}\int_{B_r} \trace{\sigma(x)}  |\eta'|(|x|) d\norm{\theta}(x).
\]
Following the same argument as in the case for $\alpha=1$ in \eqref{eq.monotonwehn1d}, we obtain
\[
\frac{\norm{\theta}(B_r)}{r^\alpha}
\]
is monotone nondecreasing in $0<r<r_\alpha$.

In the case $\dimrlow (\theta)(y) \ge m$ for some constant $m>1$ and all $y$ in an open neighborhood $U\ni x_0$, we obtain that 
\[
\dim_\textup{r}(\sigma(y)) \ge m
\]
for $\norm{\theta}$-a.e. $y\in U$. Plug this into \eqref{eq.intermediateinequalityinprovingfractionalmono} and follow the same argument we obtain the monotonicity of \eqref{eq.quotientinmonotone} with $\alpha=m$.
\end{proof}

\begin{proof}[Proof of Theorem \ref{t.realizabledimensionandlocaldimension}]
We without loss focus on the case that $x_0=0\in \spt{\theta}$ and {argue by contradiction that} 
\be\label{eq.contradimbound}
\dimrlow(\theta)(0)>\dimlow(\theta)(0).
\ee
We choose a number $\alpha\in(\dimlow(\theta)(0),\dimrlow(\theta)(0))$. {Because $\alpha<\dimrlow(\theta)(0)$,} by the monotonicity formula in Lemma \ref{l.fracmonotonicityformula}, we know that there is an $r_\alpha>0$ such that
\[
\frac{\norm{\theta}(B_r)}{r^\alpha}
\]
is monotone nondecreasing for $0<r<r_\alpha$. Because $\norm{\theta}$ is a Radon measure {on $\T^n$ and hence a finite measure}, we know that 
\be\label{eq.realizabledensityquotientbound}
\frac{\norm{\theta}(B_r)}{r^\alpha} \le C
\ee
for some $C>0$ and all $0<r<r_\alpha$. However, on the other hand, because 
$$
{\liminf_{r\rta0^+}\frac{\log \norm{\theta}(B_{r})}{\log r} =\dimlow(\theta)(0)<\alpha}
$$ 
there exists a sequence of $r_j\rta0^+$ such that
\[
\limsup_{j\rta\infty} \frac{\log \norm{\theta}(B_{r_j})}{\log r_j} < \alpha,
\]
which implies that for a small $\delta>0$ and all large $j$
\[
r_j^{\alpha-\delta}\le \norm{\theta}(B_{r_j}).
\]
Therefore we have
\[
{\frac{1}{r_j^{\delta}}\le \frac{\norm{\theta}(B_{r_j})}{r_j^{\alpha}}.} 
\]
This contradicts the bound \eqref{eq.realizabledensityquotientbound}. {Therefore, there should be no such an $\alpha$ contained in the interval $(\dimlow(\theta)(0),\dimrlow(\theta)(0))$. This contradicts \eqref{eq.contradimbound} and hence proves the theorem.}
\end{proof}

\subsection{Applications of stationary varifolds}
\label{subsection.applicationofstationaryvarifolds}

In this subsection we make clear the applications of the known theory of stationary varifolds in the characterization of maximal media. Most of the materials in this subsection are known. Specifically we apply the celebrated rectifiability theorem on stationary varifolds given by Allard \cite{Allard1972}. There are many recent results in this topic, which we refer to \cites{de2012allard,Umenne} for more discussions and references. 

The following theorem is a result of Allard's rectifiability theorem \cite{Allard1987}*{Theorem 14}.

\begin{theorem}
    \label{t.characterizerectifiablemaximalmedium}
Suppose a medium $\theta$ is maximal and satisfies the following lower density inequality for some integer $1\le k \le n-1$
\be
\label{eq.lowerdensityconstraintk}
\liminf_{r\rta0^+} \frac{\norm{\theta}(B_r(x))}{r^k} >0
\ee
for $\norm{\theta}$-almost all $x\in \T^n$. Then $\norm{\theta}$ is $k$-rectifiable, and there is a unique realization $\mu_\theta$ that is a rectifiable stationary $k$-varifold. 
\end{theorem}

For dimension $k=1,n$ there are stronger results. First we have a result by Allard-Almgren.

\begin{theorem}[Allard-Almgren \cite{AllardAlmgren1976}]\label{t.allardalmgren}
    Suppose $\mu$ is a stationary 1-varifold on $U\Subset\R^n$ that satisfies the lower density bound for $\norm{\mu}$-almost every $x$
    \[
    \liminf_{r\rta0^+} \frac{\norm{\mu}(B_r(x))}{r} \ge \delta>0
    \]
    for some $\delta>0$, then $\mu$ is 1-rectifiable with
    \[
    d\mu(x,\tau) = \xi(x) d\delta_{T_x\mu} d\norm{\mu}(x),
    \]
    where the density $\xi(x)=\lim_{r\rta0^+}\frac{\norm{\mu}(B_r(x))}{2r}$ is $\norm{\mu}$-almost everywhere well-defined, $T_x\mu\subset T_xM$ is the 1-dimensional tangent space of the 1-rectifiable support $\textup{Spt}\norm{\mu}$ at $x$ and $\mu$ further satisfies
    \begin{enumerate}
        \item the support $\textup{Spt}\norm{\mu}$ is, up to an $\Ha^1$-null closed set $S$, a countable union of straight line segments, which are open relative to $\textup{Spt}\norm{\mu}$;
        \item on each geodesic line segment there is a constant $c\ge\delta$ such that $\xi\equiv c $;
        \item at every $x\in S$, there exists a unique stationary tangent cone consisting of finitely many half lines with densities (See Figure \ref{fig:localcone});
        \item if the density $\xi$ is discretely valued, then for every compact subset $K\csubset U$ the number of line segements that have nontrivial intersection with $K$ is finite.
    \end{enumerate}
\end{theorem}

The above theorem implies the following result.
\begin{corollary}
    \label{cor.1dmaximalmedia}
    Let $\theta$ be a maximal media that satisfies the 1-dimensional lower density bound
\be\label{eq.1dlowerdensityboyundfortensor}
\liminf_{r\rta0^+} \frac{\norm{\theta}(B_r(x))}{r} \ge \delta >0
\ee
for $\norm{\theta}$-almost every $x\in \T^n$. Then $\theta$ admits a unique 1-rectifiable stationary varifold realization as described in Corollary \ref{cor.applyallardalmgren}.
\end{corollary}

For isotropic media, or equivalently media that have realizable dimension $k=n$, we have the following characterization.

\begin{theorem}
    Let $\theta$ be an isotropic maximal medium, then {$d\norm{\theta} = c \, d\Lme^n$} for some constant $c>0$.
\end{theorem}

This result for the case when $\norm{\theta}$ is absolutely continuous with respect to the Lebesgue measure {has been established earlier, see \cite{Jikov1994}*{Section 1.6}.}

\begin{proof}
   It is not difficult to see that 
   \[
   d\mu(x,\tau)=\frac{1}{n}d\delta_{\R^n}(\tau) d\norm{\theta}(x) 
   \]
   is the unique $n$-realization of $\theta$. By the monotonicity formula (See Section \ref{subsubsection.varifolds}), we know that 
    \[
    \frac{\norm{\theta}(B_r(x))}{r^n}
    \]
    is monotone nondecreasing in $r>0$ for every $x\in \T^n$. By the Lebesgue-Besicovitch differentiation theorem (See \cite{DeLellis2008}*{Theorem 2.10}), $\norm{\theta}$ is absolutely continuous with respect to the Lebesgue measure, that is, for some $\xi\ge0$ and $\xi\in L^1(\T^n)$
    \[
    d\norm{\theta}(x)=\xi(x) d\Lme^n.
    \]
    Now, by Theorem \ref{t.equivalence} we have for every $\Phi\in C_0^\infty(U)$
    \[
   0= \int_U \gd\cdot\Phi(x) d\norm{\theta}(x)=\int_U \gd\cdot\Phi(x)\xi(x) d\Lme^n,
    \]
    which shows that $\xi(x)\equiv c$ for some constant $c>0$ and all $x\in U$.
\end{proof}

\subsection{Stationary networks and the question of the maximal valency of leaf vein patterns}\label{subsection.openquestion}

As we have established the connection between the stationary varifolds and leaf vein patterns in Theorem \ref{t.equivalence}, it is natural to ask what geometric properties of the leaf vein patterns can we derive from this relation. In this subsection, we provide an example question on stationary 1-varifold on $\T^2$ with density exactly 1 almost everywhere. In this example we formulate mathematically the question on the maximal valencies, that is, the maximal number of edges joining at one node, of leaf vein patterns.

By the Allard-Almgren characterization in Theorem \ref{t.allardalmgren} and Corollary \ref{cor.1dmaximalmedia}, such objects are networks composed of finitely many nodes and edges $(\Na,\Ea)$, where $\Na\subset\T^2$ and $\Ea$ is a collection of straight line segments in $\T^2$ with endpoints in $\Na$ satisfying that for every node $x\in \Na$ there is an integer $k\ge 2$ such that the edges joining at $x$ can be written as $e_1,\dots,e_k\in\Ea$, and the edges always satisfy, according to \eqref{eq.mcone}, the balance condition
\be\label{eq.balancecondition}
\sum_{i=1}^k T_{i} =0,
\ee
where for each $1\le i\le k$, the vector $T_i$ is the unit tangent vector of the line segment $e_i$ starting from the endpoint $x$.
\begin{definition}
    \label{def.stationarynetworks}
    We define a \emph{periodic planar stationary network} (simply call \emph{stationary network} in the following context) as $\Gamma=(\Na,\Ea)$ as described above that satisfies \eqref{eq.balancecondition} for all $x\in \Na$. We do not distinguish $\Gamma$ and the closed set consisting of all points in $\Na$ and edges in $\Ea$.
    For a stationary network $\Gamma$, define the \emph{valency} $\textup{V}_\Gamma(x)$ at $x\in \T^2$ as
    \[
\textup{V}_\Gamma(x):=\lim_{r\rta0^+} \# \partial B_r(x) \cap \Gamma.
    \]
\end{definition}

To talk about the maximal valency problem, we need another concept, which is the irreducibility of a stationary network. 

\begin{definition}
    \label{def.irreducible}
    A stationary network is called \emph{irreducible} if it can not be written as the union of two distinct stationary networks.
\end{definition}

Our question on the maximal valency states as follows.

\begin{problem}[Maximal valency]\label{prob.maxval}
    Is there a universal constant $C>0$ such that for any irreducible stationary network $\Gamma$, the maximal valency
    \[
    V(\Gamma):=\max_{x\in\Gamma} V_\Gamma(x) \le C?
    \]
    If the bound $C$ exists, is $C =4 ,5$ or 6 ?
\end{problem}

The reason of using the notion \say{irreducible} is Lemma \ref{l.genericunionproperty} below, which states that a reducible stationary network can be perturbed infinitesimally without losing stationarity. This indicates that the maximal valency of reducible stationary networks is unstable in the sense that it changes after small and  simple perturbations of the network while the stationarity is preserved (see Figure \ref{fig:perturornot}).
    \begin{figure}[hbpt]
\centering
\begin{subfigure}{.3\textwidth}
  \centering
  \includegraphics[width=.5\linewidth]{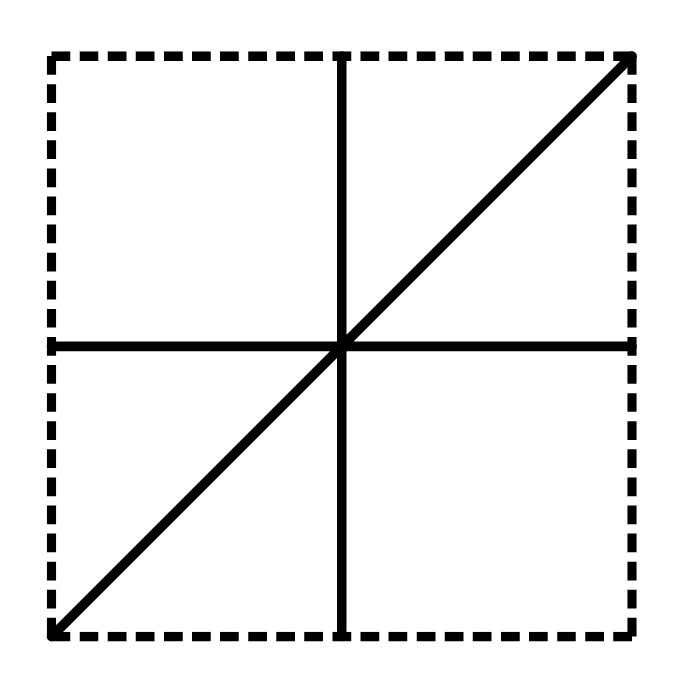}
  \caption{A reducible stationary network with maximal valency 6.}
  \label{pertbefore}
\end{subfigure}%
\hspace{0.2cm}
\begin{subfigure}{.3\textwidth}
  \centering
  \includegraphics[width=.5\linewidth]{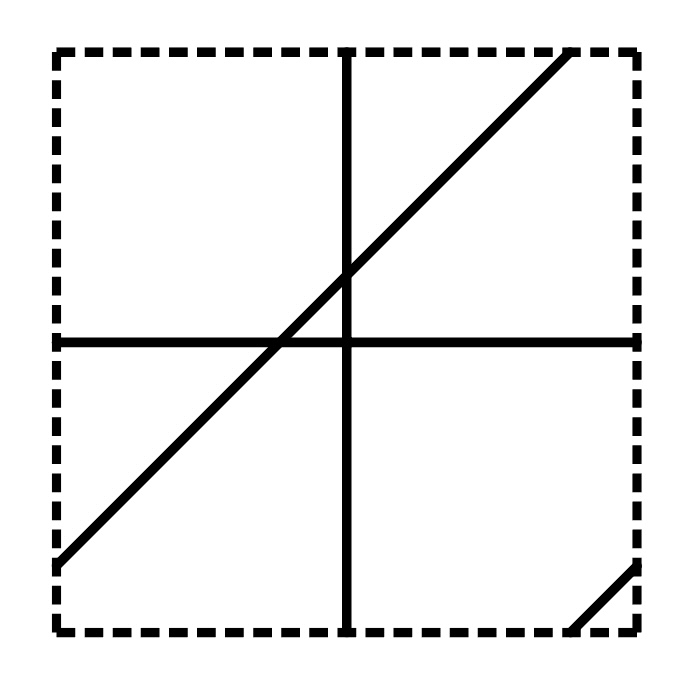}
  \caption{After small perturbations the maximal valency becomes 4.}
  \label{afterpert}
\end{subfigure}%
\caption{Instability of maximal valencies of reducible stationary networks. }
\label{fig:perturornot}
\end{figure}

\begin{lemma}
    \label{l.genericunionproperty}
    Given two stationary networks $\Gamma_1$ and $\Gamma_2$, for any $\ep>0$ there is a vector $|p|=1$ such that 
    \[
    \Gamma_1 \cup \Gamma_2 + \ep p
    \]
    is a stationary network.
\end{lemma}

\begin{proof}
    Note that $\Gamma_1 \cup \Gamma_2$ is a stationary network whenever the intersection of the edges is at most finite. This can always be achieved by a small perturbation.
\end{proof}

One might also ask whether there are only finitely many such irreducible objects. In the following we prove that there are infinitely many irreducible stationary networks. Note that the following construction does not enumerate all the irreducible stationary networks.

\begin{lemma}\label{l.infinitelymanyirreducible}
    There are infinitely many irreducible \(\Z^2\)-periodic stationary networks. All of them have maximal valency 4.
\end{lemma}

\begin{proof}
    We prove this by constructing infinitely many such patterns. Let \(k>0\) be an arbitrary odd number and we put \(k\) equilateral {parallelograms} with inner angles \(\pi/3,2\pi/3\) (diamonds) inside a unit square \((0,1]^2\). In Figure \ref{fig:infinitemanyirredu}, we put \(k=7\) such parallelograms along the line \(y=0.5\). There is no obstacle to set $k$ to be arbitrarily large as long as $k$ is odd.

    \begin{figure}[htbp]
        \centering
        \includegraphics[width=0.3\linewidth]{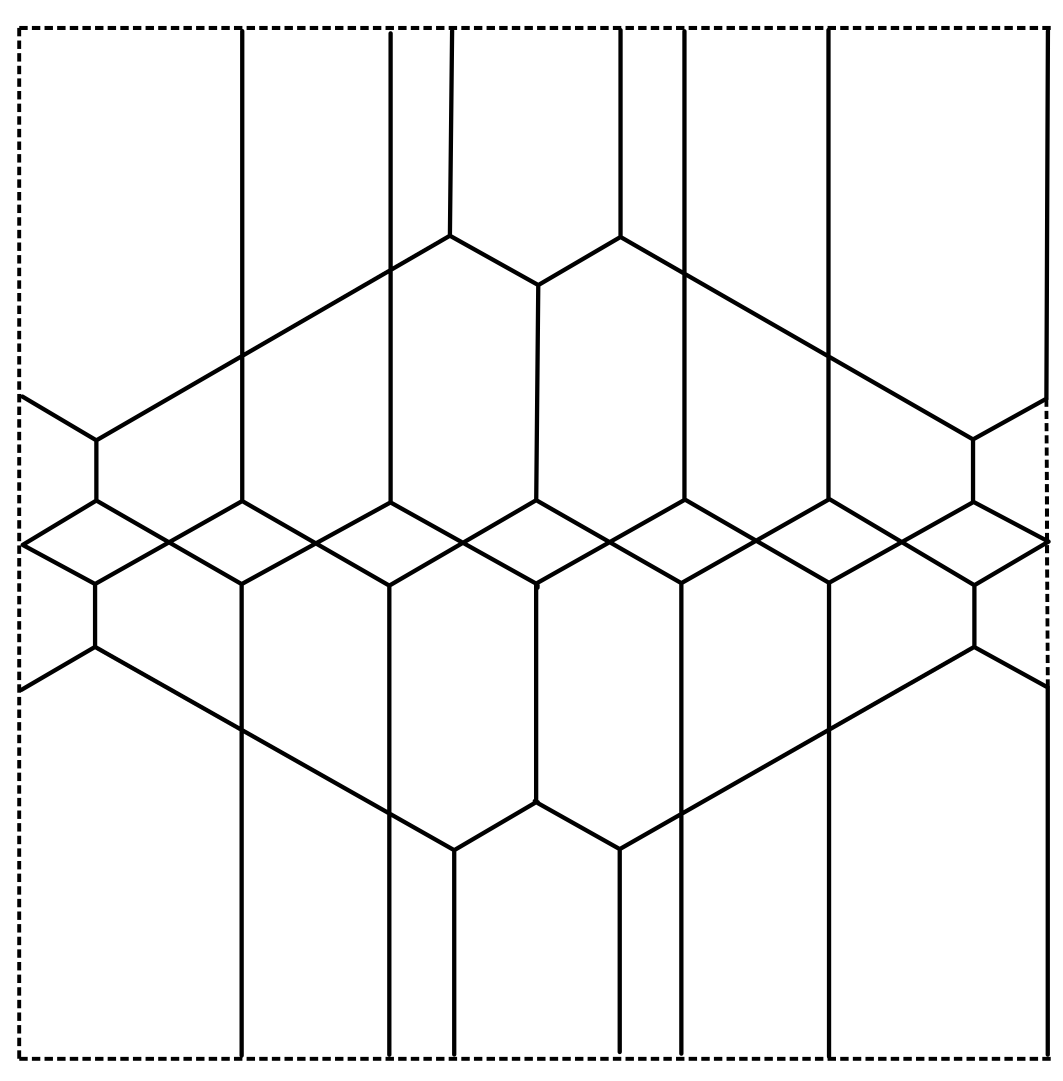}
        \caption{An irreducible stationary network that has $k=7$ parallelograms in the middle. This example was first made in \cite{Huang2021}.}
        \label{fig:infinitemanyirredu}
    \end{figure}

\end{proof}

\section{The support of nontrivial media}
\label{section.nontrivialmedia}

In this section we present three theorems about the support of a nontrivial medium. These results present the interplay between the geometry of the support $\spt{\theta}$ and the nonzero effective tensor $Q(\theta)\ne0$ under different assumptions. They are also important preliminary results for the discussion of the 1-D attainability of the lower Wiener bound in Section \ref{section.leafvenation}. In Section \ref{subsection.examplesontrivialmeida} we also present several examples that can show the sharpness of these theorems.

Our first result concerns about the dimension of a nontrivial medium and shows that the support of a nontrivial medium is not totally disconnected, and hence has at least Hausdorff dimension 1. This result is sharp in the sense that the dimension of the support of a nontrivial medium can be any real number in $[1,n]$, see Example \ref{ex.characterizenontrivialmediabutnopointwisecontrol} for more details.

\begin{theorem}
    \label{t.dimensiononeofnontrivialmedium}
    Let $\theta$ be a nontrivial medium, then the support $\spt{\theta}$ is not totally disconnected, that is, there is a nonsingleton component in $\spt{\theta}$. This implies that the 1D Hausdorff measure
    \be\label{eq.positiveH1measure}
    \Ha^1(\spt{\theta}) >0.
    \ee
    In particular, the Hausdorff dimension $\dim_{\Ha}(\spt{\theta})\ge 1$.
\end{theorem}

Our second result concerns about the decomposition of $\theta$ into the countable sum of its restrictions to the components of $\spt{\theta}$. In general there is no countable additivity due to Example \ref{ex.failureofcountableadditivity}. We also show the existence of a maximal medium having uncountably many connected components in Example \ref{ex.uncountablymanycomponentsmaximalmedia}.

\begin{theorem}
    \label{t.countabledecompositionofcomponents}
    Let $\theta$ be a nontrivial medium and suppose $\spt{\theta}$ have countably many components $E_i$ with
    \[
   E_i\subset \spt{\theta}, \ \spt{\theta} = \bigcup_{i=1}^\infty E_i \textup{ and } E_i\cap E_j=\emptyset \textup{ for }i\ne j.
    \]
   Define the restrictions $\theta_i:=\restr{\theta}{E_i}$ then
    \[
    Q(\theta)=\sum_{i=1}^\infty Q(\theta_i).
    \]
    In particular, if $\theta$ is saturated then all $\theta_i$'s are nontrivial.
\end{theorem}

Our third result concerns about the countable decomposition of $\theta$ solely when $\spt{\theta}$ has finite $\Ha^1$ measure. The key improvement of this theorem is that it does not require the medium to have only countably many components.

\begin{theorem}\label{t.countabledecompositionof1dmedium}
    Let $\theta$ be a nontrivial medium such that $\Ha^1(\spt{\theta})<\infty$. Then there exist countably many 1-rectifiable mutually disjoint connected components $E_i\subset\spt{\theta}$ such that the submedia $\theta_i:=\restr{\theta}{E_i}$ satisfy
   \be
   \label{eq.countabledecomp}
   Q(\theta)= \sum_{i=1}^\infty Q(\theta_i) .
   \ee
\end{theorem}

We postpone the proofs of Theorem \ref{t.dimensiononeofnontrivialmedium}, Theorem \ref{t.countabledecompositionofcomponents} and Theorem \ref{t.countabledecompositionof1dmedium} to Section \ref{subsection.proofof5.15.2} and \ref{subsection.proofof5.3}. In Section \ref{subsection.examplesontrivialmeida} we present some examples.

\subsection{Examples on trivial and nontrivial media}
\label{subsection.examplesontrivialmeida}

In this subsection we present some examples to show the sharpness of the above theorems. The nontrivial media are weaker notions than the maximal ones, which means that we should expect less regularity properties of such media than, for example, Theorem \ref{t.characterizerectifiablemaximalmedium} and Theorem \ref{t.allardalmgren}.  We list several examples (or non-examples) to show that why Theorem \ref{t.dimensiononeofnontrivialmedium}, Theorem \ref{t.countabledecompositionofcomponents} and Theorem \ref{t.countabledecompositionof1dmedium} are sharp. There are also some examples that are related to the results obtained in the previous sections.

\begin{example}\label{ex.reverseofnontrivialmediatheorem}
    In this example we show that the reverse of Theorem \ref{t.dimensiononeofnontrivialmedium} is generally false, that is, a medium can still be trivial even if the support has positive $\Ha^1$ measures. We first parametrize $\T^1$ with unit speed coordinate $x\in [-1/2,1/2)$. Define the medium $\theta$ as
    \[
    d\theta(x)= |x|^\alpha d\Lme^1(x)=|x|^\alpha dx,
    \]
    where $\alpha>1$. Notice that $\spt{\theta}=\T^1$ but it is not difficult to check that $\theta$ is a trivial medium.

\end{example}

\begin{example}
    \label{ex.failureofcountableadditivity}
    In Lemma \ref{l.basicpropertiesofQ} \ref{condition.(2)} we proved additivity for two media $\theta_1$ and $\theta_2$ satisfying $\spt{\theta_1}\cap \spt{\theta_2}=\emptyset$. In this example we show that there is generally no countable additivity for a sequence of media $\theta_i$ even if we assume $\spt{\theta_i}\cap \spt{\theta_j}=\emptyset$ for $i\ne j $. Indeed, on $\T^1$ we define the media 
    \[
    \theta_j:= \restr{\Ha^1}{\overline{D}_j},
    \]
    where $D_j$ is the union of open subintervals of $(0,1)$, with each interval having length $1/3^j$, obtained by taking the deleted intervals at the $j$-th step in the construction of the standard Cantor set. {More precisely, for $j\ge 1$ the sets $D_j$ satisfy
    \[
  D_1=(1/3,2/3)\textup{ and }D_{j+1}=\bigcup _{k=0}^{3^{j}-1}\left({\frac {3k+1}{3^{j+1}}},{\frac {3k+2}{3^{j+1}}}\right)\setminus D_j.
    \]
   See Figure \ref{fig:cantor} for the corresponding $D_j$'s when $j=1,2,3$. Note that $\spt{\theta_j}=\overline{D}_j$ and it is not difficult to show that 
   $$
  \dist(\spt{\theta_i},\spt{\theta_j}) =  \dist(\overline{D}_i,\overline{D}_{j})=1/3^{\max\{i,j\}}>0
   $$ 
   for all $i\ne j$.} Now we can find that
    \[
    Q(\theta_j)=0 \textup{ for all }j\ge 1,
    \]
    while $\sum_{j=1}^\infty \theta_j$ is exactly the Lebesgue measure of $\T^1$, which is a maximal medium.
\end{example}

        \begin{figure}[htbp]
        \centering
        \includegraphics[width=0.8\linewidth]{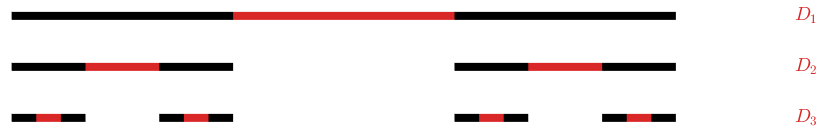}
        \caption{The first three steps in constructing the standard Cantor set. The red intervals correspond to the deleted intervals $D_j$ in step $j=1,2$ and $3$.}
        \label{fig:cantor}
    \end{figure}

\begin{example}\label{ex.characterizenontrivialmediabutnopointwisecontrol}
In this example we show that it is difficult to go beyond Theorem \ref{t.dimensiononeofnontrivialmedium} and prove a similar pointwise theory like Theorem \ref{t.realizabledimensionandlocaldimension} and the fractional montonicity formula like Lemma \ref{l.fracmonotonicityformula} for maximal media. In fact we show that a medium can have \say{arbitrarily} rough support even if they are close to the upper Wiener bound \eqref{eq.singularwienerbound}. We start with the medium $\theta$ on $\T^n$ as defined below
    $$
    d\theta:=e_1\otimes e_1\, d\restr{\Ha^1}{\{x'=0\}},
    $$
  where we parametrize $\T^n$ by $x=(x_1,x')\in[0,1)\times[0,1)^{n-1}$ and $e_1=\gd x_1$ is the unit vector on $x_1$-axis. Notice that $\theta$ is maximal. Now let $\xi$ be an arbitrary medium, then by Lemma \ref{l.basicpropertiesofQ} \ref{condition.(1)}, for every $\ep>0$, the medium $\theta_\ep:=\theta+\ep\xi$ is nontrivial. As $\ep>0$ is arbitrary, the medium $\theta_\ep$ can be made arbitrarily close to the upper Wiener bound \eqref{eq.singularwienerbound}. However, the dimension of $\spt{\theta_\ep}$ can be any number between $1$ and $n$ as $\xi$ is chosen arbitrarily. 

\end{example}

\begin{example}\label{ex.uncountablymanycomponentsmaximalmedia}

In this example we present a maximal medium $\theta$ such that $\spt{\theta}$ has uncountably many components. Let $\beta=\frac{\log2}{\log3}$, and consider on $(x_1,x_2)\in\T^2$ the following Radon measure
\[
dw(x_1,x_2):=d\restr{\Ha^\beta}{C}(x_1) dx_2,
\]
where $C\subset\T^1$ is the standard Cantor set. By standard result (See \cite{Falconer_1985}*{Theorem 1.14}), we know that $\restr{\Ha^\beta}{C}$ is a probability measure with support equal to $C$. This shows that $w$ is a probability measure supported on $C\times \T^1$. 

We define the medium $\theta$ to be
\[
\norm{\theta}=w \textup{ and } \sigma:= \frac{d\theta}{d\norm{\theta}}= e_2\otimes e_2,
\]
where $e_1=\gd x_1$ and $e_2=\gd x_2$ form an orthonormal basis for $\R^2$. By using Fubini's theorem, it is not difficult to check that $\theta$ is maximal on $\T^2$, but $\spt{\theta}$ is composed of uncountably many components.

\end{example}

\subsection{Proof of Theorem \ref{t.dimensiononeofnontrivialmedium} and Theorem \ref{t.countabledecompositionofcomponents}}

\label{subsection.proofof5.15.2}

The proofs of both Theorem \ref{t.dimensiononeofnontrivialmedium} and Theorem \ref{t.countabledecompositionofcomponents} hinge on the following fact from topology. This lemma adapts from the arguments in \cite{arhangelskij2008topological}*{Proposition 3.1.7}.

\begin{lemma}\label{l.approximateclopenset}
    Let $X$ be a closed subset of $\T^n$, then for every component $E\subset X$ and its open neighborhood $ V\subset \T^n$ there is another open subset $K$ of $\T^n$ such that $E\subset K \subset V$, and $K\cap X$ is closed in $\T^n$.
\end{lemma}

The argument for proving this lemma also works for compact components in a closed set of $\R^n$.

\begin{proof}
We consider the set $W=V\cap X$, and if $W$ is also closed then we are done. Otherwise $\lb\pt V\rb\cap X$ is non-empty. Notice that $E\subset W$ is a component of $X$ and therefore it is also a component for $\overline{W}$. By \cite{engelking1978dimension}*{Lemma 1.4.4}, $E$ is also a quasi-component of $\overline{W}$. {Recall that a quasi-component in $\overline{W}$ is the intersection of all subsets of $\overline{W}$ that contain a fixed $x\in X$ and are both open and closed in the subset topology.} Therefore for every $y\in \lb\pt V\rb\cap X$, {because $y\notin E$} there is an open subset $H_y\subset \T^n$ such that $E\subset H_y$, $y\not\in H_y$ and $H_y\cap \overline{W}$ is also closed. This means that $\overline{W}\setminus H_y$ form an open cover for $\lb\pt V\rb\cap X$ in the subset topology of $\overline{W}$. As $\lb\pt V\rb\cap X$ is compact, we may choose finitely many $y_1,\dots,y_m\in \lb\pt V\rb\cap X$ such that $\overline{W}\setminus H_{y_i}$, $i=1,\dots,m$ cover $\lb\pt V\rb\cap X$. Let $H=\cap_{i=1}^m H_{y_i}$, then $H$ by the previous construction is still open and contains $E$, and $H\cap \overline{W}$ is closed. Moreover, we know that $ H\cap \overline{W} \cap \lb\pt V\rb\cap X =\emptyset$, which means that $H\cap \overline{W}\csubset V$, and hence we may finish the proof by just choosing \(K = V\cap H.\)

\end{proof}

\begin{proof}[Proof of Theorem \ref{t.dimensiononeofnontrivialmedium}]
We argue by contradiction and assume that the compact set 
$$
X:=\spt{\theta}\subset\T^n
$$ 
is totally disconnected. We claim that in this case the mean conductance 
\[
M(\theta) = 0.
\]

To prove the claim, we apply Lemma \ref{l.approximateclopenset} to the totally disconnected set $X$. Notice that all components of $X$ are singletons, and therefore for every small $\delta>0$ and every $x\in X$ there is an open subset $K_\delta(x)\subset \T^n$ such that $x\in K_\delta(x)$, $\diam(K_\delta)< \delta$ and $K_\delta\cap X$ is closed. Now by compactness of $X$ we can choose finitely many $x_i$ for $i=1,\dots,m$ such that the open sets $K_\delta(x_i)$'s cover $X$ and each $I_i:=K_\delta(x_i)\cap X$ is closed in $\T^n$. In other words, $I_i$'s are both closed and open in $X$ in the subset topology. After finitely many times of intersections and complementations of $I_i$'s, we obtain for $X$ a new finite cover $J_k$'s that are disjoint from each other and also both open and closed in $X$ in the subset topology. Moreover the diameters $\diam(J_k)<\delta$.

For each $k$, let $\psi_k\in C^\infty(\T^n)$ be a smooth function such that $0\le \psi_k\le 1$, $\psi_k \equiv 1$ in a neighborhood of $J_k$ and $\psi_k=0$ outside a slightly larger neighborhood. By the previous constructions on $J_k$'s we can further allow
\[
\diam(\spt{\psi_k})<\delta \textup{ and }\psi_k\psi_l=0 \textup{ for all }k\ne l.
\]
Because $\T^n$ is a flat torus for each $k$ we may choose a smooth isometric chart $\xi_k$ from a $\delta$ geodesic ball containing $J_k$ to the ball $B_{\delta}(0)$ in $\R^n$. Now we define
\[
\Phi_0(x) = -\sum_{k} \xi_k \psi_k\in C^\infty(\T^n;\R^n),
\]
and see that, if we write $d\theta=\sigma d\norm{\theta}$ and apply \eqref{eq.otherformofMthetha}
\[
M(\theta) \le \frac{1}{n} \int_{\T^n} \trace{(\gd \Phi_0 + I)^T\sigma(\gd \Phi_0 + I)} d\norm{\theta}(x) = 0.
\]
{This contradicts the assumption that $\theta$ is nontrivial and implies that the support $\spt{\theta}$ must contain at least one non-singleton connected component $E$. We finish the proof of this theorem by applying the following lemma.
\begin{lemma}\label{l.albert211}
    Let $E$ be a connected subset of a continuum $X$. Then $\Ha^1(E)\ge \diam(E)$.
\end{lemma}
The proof of this lemma can be found in \cite{ALBERTI201735}*{Lemma 2.11}.
}

\end{proof}

To prove Theorem \ref{t.countabledecompositionofcomponents} we require the following lemma that improves Lemma \ref{l.basicpropertiesofQ} \ref{condition.(2)}.

\begin{lemma}
\label{l.additivityforcomponentdecomposition}
    Let $\theta$ be a medium and $E\subset \spt{\theta}$ be a component. If we write $\theta_E$ to be the restriction of $\theta$ on $E$ then
    \[
    Q(\theta)=Q(\theta_E)+ Q(\theta-\theta_E).
    \]
\end{lemma}

\begin{proof}
    The difficulty here is that the support of $\theta_E$ may intersect the support of its complement $\theta-\theta_E$. To overcome this, we apply Lemma \ref{l.approximateclopenset} to the component $E$. We then obtain for each $\delta>0$ an open subset $K_\delta\subset \T^n$ such that $E\subset K_\delta$, $\dist(E,\T^n\setminus K_\delta)<\delta$ and $K_\delta\cap \spt{\theta}$ is closed in $\T^n$. We denote $J_\delta:=K_\delta\cap \spt{\theta}$, and define the restrictions
    \[
    \theta_\delta:=\restr{\theta}{J_\delta}.
    \]
    Notice that $\theta_\delta\overset{*}{\weakcv} \theta_E$ and $\theta_E\le \theta_\delta$. By applying Lemma \ref{l.faithfulconvergencewhenshrinking}, we know that $\theta_\delta$ faithfully converges to $\theta_E$ as $\delta\rta0^+$, that is, we have
    \[
\lim_{\delta\rta0^+}Q(\theta_\delta)=Q(\theta_E).
    \]
    On the other hand, by the construction of $\theta_\delta$, the support $\spt{\theta_\delta}$ is disjoint from the support of its complement $\spt{(\theta-\theta_\delta)}$. By Lemma \ref{l.basicpropertiesofQ} \ref{condition.(2)}, we have for every $\delta>0$
    \be\label{eq.intermediatedecompositionatdelat}
    Q(\theta)=Q(\theta_\delta) + Q(\theta-\theta_\delta).
    \ee
    Because $\theta-\theta_\delta$ weakly* converges to $\theta-\theta_E$, by sending $\delta\rta0$ in \eqref{eq.intermediatedecompositionatdelat} and applying Lemma \ref{l.uppersemicontinuous}
    \[
    Q(\theta)=Q(\theta_E)+\limsup_{\delta\rta0^+} Q(\theta-\theta_\delta) \le Q(\theta_E) + Q(\theta-\theta_E).
    \]
We then finish the proof by applying Lemma \ref{l.basicpropertiesofQ} \ref{condition.(3)}.
    
\end{proof}

\begin{proof}[Proof of Theorem \ref{t.countabledecompositionofcomponents}]
Notice that if $E$ and $F$ are disjoint components of $\spt{\theta}$, then $F$ is also a component of $\spt{(\theta-\theta_E)}$, where $\theta_E$ is the restriction of $\theta$ on $E$. 

Let $E_i$ be the countable collection of components of $\spt{\theta}$ and $\theta_i:=\restr{\theta}{E_i}$ be the restrictions of $\theta$ on $E_i$. By iterating Lemma \ref{l.additivityforcomponentdecomposition}, we obtain the identity
\[
Q(\theta)= \sum_{i=1}^N Q(\theta_i) + Q\lb\theta-\sum_{i=1}^N \theta_i\rb,
\]
for all $N\ge1$. In particular, we have for all $N\ge1$ the following identity for the mean conductances
\[
M(\theta)= \sum_{i=1}^N M(\theta_i) + M\lb\theta-\sum_{i=1}^N \theta_i\rb.
\]
Notice that by Lemma \ref{l.basicpropertiesofQ} \ref{condition.(3)} we always have $Q(\theta)\ge \sum_{i=1}^\infty Q(\theta_i)$, and if the inequality is strict, then there is a constant $c>0$ such that for all $N\ge1$
\[
 M\lb\theta-\sum_{i=1}^N \theta_i\rb =M(\theta)- \sum_{i=1}^N M(\theta_i) \ge c>0.
\]
This is impossible because $\lb\theta-\sum_{i=1}^N \theta_i\rb\overset{*}{\weakcv}0$ as $N\rta\infty$, and by Lemma \ref{l.uppersemicontinuous} we have
\[
0\ge \limsup_{N\rta\infty}M\lb\theta-\sum_{i=1}^N \theta_i\rb \ge c>0,
\]
which is a contradiction.    
\end{proof}

\subsection{Proof of Theorem \ref{t.countabledecompositionof1dmedium}}
\label{subsection.proofof5.3}
In this proof we will have to select carefully a countable family of connected components based on the only assumptions that $Q(\theta)\ne0$ and $\Ha^1(\spt{\theta})<\infty$. 

We begin with the following quantitative version of the inequality \eqref{eq.positiveH1measure}.

\begin{lemma}
    \label{l.quantiposih1measure}
    Suppose a medium $\theta$ is nontrivial, then
    \[
   {\Ha^1(\spt{\theta}) \ge 1/2.} 
    \]
\end{lemma}

\begin{proof}
{We argue by contradiction and assume that $\Ha^1(\spt{\theta})<1/2$. We claim that in this case, the set $\spt{\theta}$ can be decomposed as a finite collection of subsets that have positive distance from each other and have diameters smaller than 1/2.}

First of all, by definition (See Section \ref{subsubsection.hausdorff}), we know that for every $\delta>0$
\[
\inf\lma \sum_{j=1}^\infty \diam(U_j) \ ; \ \diam(U_j)\le \delta,\, \spt{\theta}\subset \bigcup_{j=1}^\infty U_j\rma < 1/2.
\]
This shows that, by using the compactness of $\spt{\theta}$, we can find $N\in \N_+$ open balls $B_j$ such that
\[
\sum_{j=1}^N \diam(B_j) < 1/2 \textup{ and } \spt{\theta}\subset \bigcup_{j=1}^N B_j=: K.
\]
{We finish the proof of the claim by showing that the diameter of each component of $K$ is less than $1/2$. To prove this we denote $K^*\subset K$ a connected component of $K$. As $K$ is a finite union of balls, the component $K^*$ is also a finite union of balls $\{B_{j_l}\}\subset \{B_j\}_{j=1}^N$. By the triangle inequality
\[
\diam(K^*) \le \sum_{l} \diam(B_{j_l}) \le \sum_{j=1}^N \diam(B_j) < 1/2.
\]
This proves the claim.
}
 
{Now, because each connected component $K^*$ of $K$ has diameter at most $1/2$, $K^*$ as a subset of $\T^n$ must be isometrically diffeomorphic to a connected open domain in $\R^n$. Indeed, the preimage $\pi^{-1}(B_r)$ of any open ball $B_r\subset \T^n$ of radius $0<r\le1/2$ under the standard projection $\pi:\R^n\rta\T^n$ is a countable union of disjoint balls in $\R^n$ of the same radius, and $\pi$ restricted on any of these balls defines an isometric diffeomorphism onto $B_r$. Applying this observation to a fixed ball of radius $1/2$ that covers $K^*$ (there is always such a ball because $\diam(K^*)<1/2$), we can find a smooth isometric diffeomorphism $\phi:K^*\rta \R^n$. The proof is then done by a similar construction of the smooth test vector field $\Phi_0$ as in the proof of Theorem \ref{t.dimensiononeofnontrivialmedium}.}

\end{proof}

\begin{proof}[Proof of Theorem \ref{t.countabledecompositionof1dmedium}]
We construct $E_i$'s inductively and denote $\theta_i=\restr{\theta}{E_i}$. Because $\theta$ is nontrivial, by Theorem \ref{t.dimensiononeofnontrivialmedium} we know that there is always a connected component of $\spt{\theta}$ having positive $\Ha^1$ measure. Therefore, we can find $E_1$ to satisfy
\[
\Ha^1(E_1) \ge \frac{1}{2}\sup\{\Ha^1(F)\ ; \ F \textup{ is a connected component of }\spt{\theta}\}>0.
\]
By Lemma \ref{l.approximateclopenset}, for every $\ep>0$ we can find an open set $J_\ep$ such that $J_\ep\subset J_{\delta}$ whenever $\delta>\ep>0$ and
\be\label{eq.propertyofJep}
E_1\subset J_\ep, \ \dist(E_1,\pt J_\ep) \le \ep, \ \Ha^1(J_\ep\cap \spt{\theta}\setminus E_1)\le \ep
\ee
and $J_\ep\cap\spt{\theta}$ is closed. We claim that when $\ep>0$ is sufficiently small, we have
\[
Q(\restr{\theta}{J_\ep})=Q(\theta_1).
\]
Indeed, on one hand we have by Lemma \ref{l.faithfulconvergencewhenshrinking} the identity
$\displaystyle\lim_{\ep\rta0^+}Q(\restr{\theta}{J_\ep})=Q(\theta_1).$ On the other hand, we have by Lemma \ref{l.additivityforcomponentdecomposition}
\[
Q(\restr{\theta}{J_\delta}) = Q(\restr{\theta}{J_\ep}) + Q(\restr{\theta}{J_\delta\setminus J_\ep})
\]
for $0<\ep<\delta\ll 1$. When both $\delta,\ep\rta0^+$ with $\delta>\ep$, by \eqref{eq.propertyofJep}, we know that $\Ha^1(J_\delta\cap \spt{\theta}\setminus J_\ep) \rta0 $. This implies that according to Lemma \ref{l.quantiposih1measure}
\[
Q(\restr{\theta}{J_\delta}) = Q(\restr{\theta}{J_\ep}) 
\]
for all $\delta>\ep>0$ whenever both are small. Therefore $Q(\restr{\theta}{J_\ep})=Q(\theta_1)$ for sufficiently small $\ep>0$. 

We define $K_1$ to be such an open set $J_\ep$ with sufficiently small $\ep>0$ that satisfies
\[
E_1\subset K_1, \ Q(\restr{\theta}{K_1})=Q(\theta_1)
\]
and $\spt{\theta}\cap K_1$ is closed. If $\theta-\restr{\theta}{K_1}$ is trivial, then by Lemma \ref{l.basicpropertiesofQ} we are done. Otherwise by Theorem \ref{t.dimensiononeofnontrivialmedium} there is a connected component $E_2$ of $\spt{\theta}$ that is disjoint from $K_1$ and
\[
\Ha^1(E_2) \ge \frac{1}{2}\sup\{\Ha^1(F)\ ; \ F \textup{ is a connected component of }\spt{\theta} \setminus K_1\}>0.
\]
Note that by the construction of $K_1$, all the connected components of $\spt{\theta} \setminus K_1$ is also a component of $\spt{\theta}$. Suppose we have obtained $E_1,\dots,E_i$ for some $i\ge 2$. For $E_i$ we can find by a similar argument as before, an open set $K_i$ such that
\[
E_i\subset K_i, \ Q(\restr{\theta}{K_i})= Q(\theta_i) , \ K_i\cap \spt{\theta} \textup{ is closed}
\]
and $K_j\cap K_{j'}=\emptyset$ for $1\le j < j'\le i$.  If $\theta-\sum_{j=1}^i \restr{\theta}{K_j}$ is trivial, then we are done. Otherwise by applying Theorem \ref{t.dimensiononeofnontrivialmedium} to $\theta-\sum_{j=1}^i \restr{\theta}{K_j}$, we can find a connected component $E_{i+1}$ of $\spt{\theta}$ disjoint from all $K_j$ for $j\le i$ and
\be\label{eq.countableinequaltiy}
\Ha^1(E_{i+1}) \ge \frac{1}{2}\sup\{\Ha^1(F)\ ; \ F \textup{ is a connected component of }\spt{\theta} \setminus \bigcup_{j=1}^i K_j\}>0.
\ee
We finish the proof by showing that the mean conductance satisfies
\[
\lim_{i\rta\infty} M\lb\theta-\sum_{j=1}^i \restr{\theta}{K_j}\rb = 0.
\]
Suppose for some $c>0$ and all large $i$ we have
\[
M\lb\theta-\sum_{j=1}^i\restr{\theta}{K_j}\rb \ge c >0.
\]
Sending $i\rta\infty$ and denote 
\[
\Tilde{\theta}=\theta-\sum_{j=1}^\infty \restr{\theta}{K_j}.
\]
By applying the singular Wiener bound \eqref{eq.singularwienerbound} and the upper semi-continuity in Lemma \ref{l.uppersemicontinuous}, we know that 
\[
\frac{1}{n}\norm{\Tilde{\theta}}(\T^n) \ge M(\Tilde{\theta})(\T^n) \ge c >0.
\]
As $\Tilde{\theta}$ is a nontrivial medium, by Theorem \ref{t.dimensiononeofnontrivialmedium} again, there is always a component $K\subset \spt{\Tilde{\theta}}$ such that $\Ha^1(K)>0$. On one hand, we know that 
\[
\spt{\Tilde{\theta}} \subset \spt{\theta} \setminus \bigcup_{j=1}^\infty K_j,
\]
and therefore $K\cap E_j=\emptyset$ for all $j\ge 1$. On the other hand, for all $i\ge 1$ there is a connected component $\Tilde{K}_i$ of $\spt{\theta}\setminus\bigcup_{j=1}^iK_j$ that contains $K$ and satisfies by \eqref{eq.countableinequaltiy}
\[
\Ha^1(E_{i+1}) \ge \frac{1}{2}\Ha^1(\Tilde{K}_i) \ge \frac{1}{2}\Ha^1(K)>0.
\]
This contradicts the fact that $\sum_{i=1}^\infty\Ha^1(E_i) \le \Ha^1(\spt{\theta})<\infty$.

\end{proof}

\section{Loopiness, reticulation and the lower Wiener bound on networks}
\label{section.leafvenation}
In this section we characterize the lower attainability property of the singular lower Wiener bound \eqref{eq.singularwienerbound} for \emph{network-like} media. To be more specific, we focus on media of the form $d\theta = I_{n\times n}\, dw$, where the Radon measure $w$ satisfies
\be\label{eq.1dregular}
0<\limsup_{r\rta0^+}\frac{w(B_r(x))}{r}<\infty
\ee
for $w$-almost every $x\in \T^n$ and the following coercivity condition 
\be
\label{eq.coercive}
\limsup_{r\rta0^+}\frac{w(B_r(x))}{r}>c>0
\ee
for $\Ha^1$-almost every $x\in \spt{w}=\spt{\theta}$. Note that the subtle point here is that the coercivity condition \eqref{eq.coercive} is satisfied for $\Ha^1$-almost every $x\in \spt{w}$ instead of for $w$-almost every $x$. By the discussions in Section \ref{subsubsection.density}, there is an equivalent definition for such medium
\be\label{eq.networklikemedium}
d\theta(x)= a(x) I_{n\times n} \, d\restr{\Ha^1}{\Gamma}(x),
\ee
where $\Gamma=\spt{w}=\spt{\theta}$ is a closed subset of $\T^n$ and $\Ha^1(\Gamma)<\infty$, and the coercivity condition on $a$ takes the form
\be
\label{eq.ellipticfora}
a\in L^1\lb\restr{\Ha^1}{\Gamma}\rb \textup{ and } a(x) \ge \Lambda^{-1} \textup{ for }\restr{\Ha^1}{\Gamma}\textup{-almost every }x\in \T^n,
\ee
for some constant $\Lambda>0$. We refer to Example \ref{ex.reverseofnontrivialmediatheorem} for the necessity of the coercivity condition.

Such medium arise as one considers in classical theory the matrix field $\Bar{\sigma}$ of the form with small $\delta>0$ 
\[
\Bar{\sigma}(x):=\bca
\frac{a(x)}{\delta}I_{n\times n} & \dist(x,\Gamma)<(\delta/\omega_{n-1})^{1/(n-1)}\\
\delta I_{n\times n} & \textup{elsewhere}.
\eca
\]
That is, when the high conductive material concentrate on the 1D set $\Gamma$, one should study medium of the form \eqref{eq.networklikemedium} as an effective model. We refer to \cite{Huang2021} for a justification of the faithful convergence of $\Bar{\sigma}(x) \, d\Lme^2$ to $d\theta$ as in \eqref{eq.networklikemedium} as $\delta\rta0$ in the special case that the dimension $n=2$ and $\Gamma$ is a finite union of $C^2$ curves.

Most of the results in this section are motivated by the modeling of leaf venation patterns. Specifically we view $\T^n$ as a small piece of leaves, $\Gamma=\spt{\theta}$ the geometric structure of the leaf veins on $\T^n$ and $a(x)$ the local conductance of veins at $x\in \Gamma$. The goal is to show, as pointed out in the introduction, that a periodic planar network is resilient to fluctuations if and only if it is reticulate. This formal theorem will be stated in a rigorous way in Theorem \ref{t.Qtohomotopy} and Theorem \ref{t.characterizationoflowerboundn=2}.

It turns out naturally that the answer is hidden in the study of the effective tensor $Q(\theta)$, and in particular, the positive definiteness of $Q(\theta)$. By standard homogenization theory, the effective conductance of the medium $\theta$ in direction $p\in \R^n$ should be 
\[
p\cdot Q(\theta)p.
\]
Here the vector $p$ can be viewed as the effective exterior pressure gradient applied on the medium $\theta$. Suppose the medium experience a random fluctuation in $p$, then one would expect that $p\cdot Q(\theta)p$ to be as large as possible in all direction $p$. However, due to the vast variety of leaf veins, we do not expect a quantitative way to analyze, but at least we can allow $p\cdot Q(\theta)p$ to be positive in all $p$. This is equivalently saying that $Q(\theta)$ has to be positive definite, or in other words, the medium $\theta$ has to be positive. 

Although the above simple explanation is clear, there is still a gap between this argument and the original setting in \cites{Corson2010,Katifori2010}, where the minimization of the total dissipation under random fluctuations was considered. In Appendix \ref{appendix.formalderivation} we try to fix this gap by deriving the positivity of the effective tensor $Q(\theta)$ from a continuum version of the dissipation minimization problem in \cites{Corson2010,Katifori2010}. 

To characterize positive medium $\theta$ of the form \eqref{eq.networklikemedium}, we need to introduce some topological terminologies. We begin with the set $H_E$ that collects all the homotopy classes in the subset $E\subset\T^n$.

\begin{definition}
    \label{def.definehomotopygroupofsets}
    For any subset $E\subset \T^n$, we denote $\pi_1(E,x_0)\le \pi_1(\T^n,x_0)$ the subgroup of the homotopy classes containing closed paths in $E$ based at $x_0\in E$. We denote the collection of homotopy classes in $E$ as the set
    \be\label{eq.definehe}
    H_E:=\bigcup_{x_0\in E} i_{x_0}\lb \pi_1(E,x_0)\rb \subset\Z^n,
    \ee
    where $i_{x_0}:\pi_1(\T^n,x_0)\rta\Z^n$ is the isomorphism as introduced in Section \ref{subsubsection.pathliftingproperty}. Call $\Z H_E:=\textup{Span}_\Z(H_E)$, $\R H_E:=\textup{Span}_\R (H_E)$ and $H_E^\perp := \lb\R H_E\rb^\perp$ the real orthogonal complement of $\R H_E$ in $\R^n$. 
\end{definition}

\begin{definition}
    \label{def.topologynetwork}
    Let $E\subset\T^n$. Denote $H_E\subset \Z^n$ the collection of homotopy classes in $E$ as defined in Definition \ref{def.definehomotopygroupofsets}. Call $E$ to be trivial if $H_E=0$. Call $E$ to be \emph{loopy} if the real span $\R H_E=\R^n$. Call $E$ to be \emph{reticulate} if it has a loopy connected component. Call $E$ to be \emph{quasi-laminate} if $H_E=\textup{Span}_\Z (p)$ for some $p\in \Z^n$.
\end{definition}

\begin{remark}\label{r.equiloopyandreticulate}
    The terms ``loopy'' and ``reticulate'' generally share the same meaning for leaf venation patterns in biological literature \cites{RothNebelsick2001,sack2013leaf}. In dimension $n=2$ they are indeed the same mathematically (See Lemma \ref{l.characterizeloopysetn=2}). In higher dimensions they are different notions because the loops can be detached in the extra dimensions. A quick example in $\T^3$ is the following set
    \be
\label{eq.exampledetachn=3}
{\T^1\times\{0\}\times\{1/2\}\cup\{0\}\times\{1/2\}\times\T^1\cup\{1/2\}\times \T^1 \times \{0\} .}
    \ee
    {This detaching argument does not work in dimension $n=2$, because in $\T^2$ two closed paths that belong to nonparallel homotopy classes will eventually intersect with each other and form a reticulate set, see Lemma \ref{l.characterizeloopysetn=2} and Lemma \ref{l.intersectionofnonparallelpaths}. Intuitively this is related to the simple fact that on $\R^2$ any two nonparallel straight lines will always intersect with each other, while in higher dimensions $n\ge 3$, two nonparallel lines generically do not intersect with each other, and if they do they must be on the same plane. In this higher dimensional scenario, one can construct a loopy set like \eqref{eq.exampledetachn=3}, in which there are three linearly independent closed paths that do not intersect with each other, i.e. the union is not connected and therefore not reticulate.}
\end{remark}

{\begin{remark}
    \label{r.onthedefinitonofloopy}
It is natural to ask whether $E$ being loopy can be equivalently defined as $H_E$ spanning $\Z^n$ instead of $\R^n$. It turns out that in dimension $n=2$, the two ways of defining loopy are equivalent. This can be seen by Lemma \ref{l.characterizeloopysetn=2} \ref{condition.(d)loopy}. However, in higher dimensions $n\ge 3$, $H_E$ spanning $\Z^n$ is a stronger property than loopiness. 
 Here let us present an explicit example of a loopy set $F\subset \T^3$, of which the homotopy class collection $H_F$ does not span $\Z^3$: we parametrize $\T^3$ by $(x,y,z)\in [0,1)^3$, and define
\be
\label{eq.spanRnotZ}
\begin{split}
    F&:=\{2y=x \textup{ mod }\Z\textup{ and }z=0\} \cup \{x=0, \ z=1/2\} \cup \{x=y=1/2\}\\
    &=:F_1\cup F_2 \cup F_3.
\end{split}
\ee
Note that $H_{F_1}=\textup{Span}_{\Z}\{(2,1,0)\}$, $H_{F_2}=\textup{Span}_{\Z}\{(0,1,0)\}$ and $H_{F_3}=\textup{Span}_{\Z}\{(0,0,1)\}$ (these sets can be easily derived by using the path lifting property and the simple structure of $F_i$'s). Because $F_1$, $F_2$ and $F_3$ are mutually disjoint, we have $H_F=H_{F_1}\cup H_{F_2}\cup H_{F_3}$ and
\[
\Z H_F = (2\Z)\times\Z^2 <\Z^3 \textup{, while }\R H_F = \R^3.
\]
The intuition here is very similar to what we have just discussed in Remark \ref{r.equiloopyandreticulate}. In dimension $n=2$, even if we do not initially observe the existence of $(1,0)$ and $(0,1)$ in $H_E$, the closed paths having nonparallel homotopy classes in $E$ will have to intersect with each other, resulting in the occurrence of $(1,0)$ and $(0,1)$ in $H_E$ that span the whole $\Z^2$. However, in higher dimensions, the closed paths can be detached and contained in different components (such as the components $F_1$, $F_2$ and $F_3$ in $F$ as in \eqref{eq.spanRnotZ}).
\end{remark}
}

Our main theorem of this section is to show the direct relation between the kernel of $Q(\theta)$ and the collection of homotopy classes $H_\Gamma\subset \Z^n$ in the support $\Gamma=\spt{\theta}$.

\begin{theorem}
    \label{t.Qtohomotopy}
    Let $\theta$ be a network-like medium as defined in \eqref{eq.networklikemedium} with the weight $a$ satisfying the coercivity condition \eqref{eq.ellipticfora}, then we have the following identity
    \be
\label{eq.Qtohomotopy}
H_{\Gamma}^\perp  = \ker Q(\theta).
    \ee
    In particular, $Q(\theta)$ is positive definite if and only if $\Gamma$ is loopy. 
\end{theorem}

In dimension $n=2$, we can make some further improvements.

\begin{theorem}
    \label{t.characterizationoflowerboundn=2}
    Suppose $n=2$ and assume the same as Theorem \ref{t.Qtohomotopy}. We characterize the topological properties of $\Gamma$ for different forms of $Q(\theta)$:
    \begin{itemize}
        \item $Q(\theta)=0$ if and only if $\Gamma$ is trivial.
        \item $Q(\theta)=q\otimes q$ for some $q\in \R^2\setminus\{0\}$ if and only if $q\in \R\Z^2$ and $\Gamma$ is quasi-laminate in direction $q$.
        \item $Q(\theta)$ is positive definite if and only if $\Gamma$ is reticulate. 
    \end{itemize}
\end{theorem}

Theorem \ref{t.characterizationoflowerboundn=2} is a corollary of Theorem \ref{t.Qtohomotopy} combined with the fact that loopiness is equivalent to reticulation in dimension $n=2$ in Lemma \ref{l.characterizeloopysetn=2}.

 The proof of Theorem \ref{t.Qtohomotopy} consists of three steps. First, in Section \ref{subsection.decompose} we use Theorem \ref{t.countabledecompositionof1dmedium} to decompose $\theta$ as a countable sum of its restrictions to its connected components. This step reduces the problem to the case when $\Gamma=\spt{\theta}$ is a 1-rectifiable non-singleton compact connected set with finite $\Ha^1$-measure.
 
 Second, in Section \ref{subsection.oneside} we prove the inclusion $H_{\Gamma}^\perp \subset\ker Q(\theta)$. The key is to construct in a neighborhood of $\Gamma$ the linear function $p\cdot x$ for $p\in H_\Gamma^\perp$. This is accomplished by analyzing the unique lift in the covering space $\R^n/\Z H_\Gamma$ of a surjective Wa\.{z}ewski parametrization $\gamma$ of $\Gamma$.

  In the last step, we first show a change of variable result in Section \ref{subsection.wazewskiparam}. This result helps us to estimate $Q(\theta)$ and then we can finish the proof of Theorem \ref{t.Qtohomotopy} by showing the following lower bound
 \[
 p\cdot Q(\theta) p \ge \frac{|p|^2}{\Lambda c_{n,\Gamma}},
 \]
for all $p\in H_\Gamma$, where $\Lambda>0$ is the coercivity constant in \eqref{eq.ellipticfora} and $c_{n,\Gamma}>0$ depends only on $n$ and the geometry of $\Gamma$. As we have pointed out in the introduction, this is accomplished by a modification procedure that improves the estimation of lengths and multiplicities of Lipschitz closed paths in $\Gamma$ without changing the homotopy classes. See Section \ref{subsection.proofoftheoremqhomotopy} for more details.

\subsection{Decomposition of medium}
\label{subsection.decompose}
Let us show that, to prove Theorem \ref{t.Qtohomotopy}, the assumption \eqref{eq.networklikemedium} on $\theta$ can be reduced to media of the form
\be
\label{eq.simpleformofnetwork}
d\theta(x) = a(x) I_{n\times n}\, d\restr{\Ha^1}{\Gamma}(x),
\ee
where $a$ is the positive function as before in \eqref{eq.ellipticfora}, but $\Gamma$ is a 1-rectifiable closed connected set in $\T^n$ such that $\Ha^1(\Gamma)$ is positive and finite. 

To be more precise, we prove the following lemma.
\begin{lemma}
    \label{l.decompositionof1dmedium}
    Suppose Theorem \ref{t.Qtohomotopy} is correct for $\theta=\theta_w$ of the form \eqref{eq.simpleformofnetwork} with $a$ satisfying \eqref{eq.ellipticfora} and the support $\Gamma:=\spt{\theta}=\spt{w}$ being a closed connected 1-rectifiable set with $0<\Ha^1(\Gamma)<\infty$. Then Theorem \ref{t.Qtohomotopy} is correct.

\end{lemma}

\begin{proof}

    Let us start with a nontrivial medium $\theta$. By Theorem \ref{t.dimensiononeofnontrivialmedium} and Theorem \ref{t.countabledecompositionof1dmedium}, we know that there are $\theta_j \le \theta$ such that
    \be\label{eq.estimatecountabledecomp}
    Q(\theta)=\sum_{j=1}^\infty Q(\theta_j),
    \ee
    where $\theta_j$ are the restrictions of $\theta$ on $\Gamma_j:=\spt{\theta_j}\subset\Gamma$, which are mutually disjoint closed connected sets that satisfy $0<\Ha^1(\Gamma_j)<\infty$. By Theorem \ref{t.existenceofanicelipschitzparam}, each $\Gamma_j$ is also 1-rectifiable.

    Suppose $q\in H_\Gamma^\perp$, then because $\Gamma_j\subset\Gamma$, we know that for all $j$, the vector $q\in H_{\Gamma_j}^\perp$. By the assumption, we know that $q\in \ker Q(\theta_j)$ for all $j$. This implies that by \eqref{eq.estimatecountabledecomp}
    \[
    Q(\theta) q =\sum_{j=1}^\infty Q(\theta_j) q =0.
    \]
    That is, $q\in \ker Q(\theta)$. 
    
    {To prove the reverse inclusion, we first assume $q\in \ker Q(\theta)$,} and then it suffices to show that for any connected component $E\subset\spt{\theta}$, we have {$q\in H_E^\perp$}. {Indeed, because $\Ha^1(\Gamma)<\infty$, by \cite{ALBERTI201735}*{Proposition 3.4}, each connected component $E\subset\Gamma$ is also path-connected and hence for every $x\in E$
    \[
    \pi_1(\Gamma,x) = \pi_1(E,x).
    \]
    Therefore,  by definition \eqref{eq.definehe} we have
    \be\label{eq.computecomponentwise}
    H_\Gamma = \bigcup_{x\in \Gamma} i_x(\pi_1(\Gamma,x))= \bigcup_{E} \bigcup_{x\in E} i_x(\pi_1(\Gamma,x)) =  \bigcup_{E} \bigcup_{x\in E} i_x(\pi_1(E,x)) = \bigcup_{E} H_E,
    \ee
    where the union is taken with respect to all the connected components $E\subset \Gamma$. This implies that 
    $$ 
    q\in \bigcap_{E} H_E^\perp =H_\Gamma^\perp.
    $$}

    To prove the claim we may without loss assume that $E$ is a nonsingleton set, as otherwise $H_E^\perp$ is trivially $\R^n$. {Because $\theta$ takes the form \eqref{eq.networklikemedium}, it is not difficult to derive that
    \[
   d\restr{\theta}{E} = a(x) I_{n\times n} d\restr{\Ha^1}{E}.
    \]
By the discussions in the prior paragraph, $E$ is both compact and path-connected, and hence by Lemma \ref{l.albert211}, we also obtain that for any $x\in E$ the following inequality
\[
\liminf_{r\rta0^+} \frac{\theta(E\cap B_r(x))}{r} \ge\liminf_{r\rta0^+} \frac{\Ha^1(E\cap B_r(x))}{\Lambda r}I_{n\times n}\ge \frac{I_{n\times n}}{\Lambda} >0
\]
for $\Lambda>0$ the coercivity constant. This shows that
\[
\spt{\restr{\theta}{E}} = E.
\] 
Therefore the support of $\restr{\theta}{E}$ is a continuum. Also because $0<\Ha^1(E)<\infty$, by Theorem \ref{t.existenceofanicelipschitzparam}, the medium $\restr{\theta}{E}$ satisfies the assumption in Lemma \ref{l.decompositionof1dmedium} with $\Gamma$ replaced by $E$, and therefore $q\in H_E^\perp$.} 

   Let us finish the proof by discussing the case that $\theta$ is trivial. By Theorem \ref{t.dimensiononeofnontrivialmedium}, we do not worry about the case that the support $\Gamma=\spt{\theta}$ is also totally disconnected. In this case we automatically have the identity \eqref{eq.Qtohomotopy}. Suppose otherwise the support $\Gamma$ has a nonsingleton connected component $E$. By Lemma \ref{l.basicpropertiesofQ} \ref{condition.(1)} we have $Q(\theta)=Q(\restr{\theta}{E})=0$. As $0<\Ha^1(E)\le \Ha^1(\Gamma)<\infty$, we have by Theorem \ref{t.existenceofanicelipschitzparam} the set $E$ is also 1-rectifiable. Therefore by the assumption, we have $H_E^\perp=\ker Q(E)=\R^n$. The proof is then finished by applying the identity \eqref{eq.computecomponentwise}.
\end{proof}
\subsection{One side inclusion}
\label{subsection.oneside}
In this subsection we show one side inclusion $H_\Gamma^\perp\subset \ker Q(\theta)$. This is accomplished by applying the Wa\.{z}ewski parametrization theorem in Theorem \ref{t.existenceofanicelipschitzparam}.

\begin{lemma}
    \label{l.firstinclusion}
    Assume the same as Theorem \ref{t.Qtohomotopy}, then we have
    \[
    H_\Gamma^\perp \subset \ker Q(\theta).
    \]
\end{lemma}

The proof of this inclusion does not depend on the coercivity constant $\Lambda$ as in \eqref{eq.ellipticfora}.

\begin{proof}
    By Lemma \ref{l.decompositionof1dmedium}, it suffices to prove for medium $\theta$ of the form \eqref{eq.simpleformofnetwork} and $\Gamma=\spt{\theta}$ a closed 1-rectifiable connected set in $\T^n$ such that $0<\Ha^1(\Gamma)<\infty$. As $\Gamma$ satisfies the conditions in Theorem \ref{t.existenceofanicelipschitzparam}, there is a surjective constant speed closed Lipschitz path $\gamma:[0,1]\rta \Gamma$.

    Let $p\in H_\Gamma^\perp$ be a nonzero vector, then by definition \eqref{eq.effectivetensor}
    \[
    p\cdot Q(\theta)p:=\inf_{\varphi\in C^\infty(\T^n)} \int_{\T^n} |\gd\varphi(x) +p|^2 a(x) d\restr{\Ha^1}{\Gamma}(x)
    \]
    To show that $Q(\theta)p=0$ it suffices to show that in an open neighborhood $U$ of $\Gamma$ here is a smooth function $\varphi_p\in C_0^\infty(U)$ such that $\gd \varphi_p = -p$ on $\Gamma$.

    To that end, we first write $H=\Z H_\Gamma$ and recall the following covering maps (See Section \ref{subsubsection.coveringspacesoftorus})
    \[
    \pi=\pi_H\circ \pi^H, \ \pi^H:\R^n \rta \T_H^n=\R^n/H \textup{ and }\pi_H:\T_H^n \rta \T^n.
    \]
    As $p\perp H$, {we know that for any $x,y\in \R^n$ such that $x-y\in H$, $p\cdot x = p\cdot y$. This shows that there is a unique function $h_p$ on $\T_H^n$ such that
    \be\label{eq.definehpppp}
    p\cdot x = h_p(\pi^H(x)) \textup{ for }x\in \R^n.
    \ee
    We claim that the function $h_p\in C^\infty(\T_H^n)$. First notice that because $H\le \Z^n$ is a discrete subgroup of $\R^n$, the projection $\pi^H$ is a locally isometric diffeomorphism onto $\T_H^n$, and for sufficiently small radius $r>0$ the preimage $(\pi^H)^{-1}(B_r)$ of each geodesic ball $B_r\subset \T_H^n$ is a countable union of disjoint balls in $\R^n$ having the same radius $r$. Restricting $\pi^H$ on a fixed component in the preimage $(\pi^H)^{-1}(B_r)$ defines an isometric diffeomorphism from the component to $B_r$, which by \eqref{eq.definehpppp} shows the differentiability of $h_p$ in $B_r$. As the position of $B_r$ can be chosen arbitrarily, we have proved that $h_p\in C^\infty(\T_H^n)$.
    }
    
    {Notice that it is impossible to find a function $h$ on $\T^n$ such that $p\cdot x = h(\pi(x))$ for all $x\in \R^n$. Otherwise one would obtain
    \[
   0\ne p\cdot p = h(\pi(p))=h(\pi(0)) = p\cdot 0 =0,
    \]
    which is impossible. This is one of the main reasons why we invoke the intermediate space $\T_H^n=\R^n/H$.}
    
    Let $\hat{\gamma}:[0,1]\rta \T_H^n$ denotes the unique path lifting of $\gamma$ (starting at some fixed point lift) and write $\hat{\Gamma}$ to be the image of $\hat{\gamma}$. We claim that in a small neighborhood $\hat{U}$ of $\hat{\Gamma}$ the covering map $\pi_H:\T_H^n\rta\T^n$ is an isometric diffeomorphism onto a neighborhood $U$ of $\Gamma$. This claim finishes the proof because we can define 
    \[
    \varphi_p:=-\eta \cdot \lmb h_p\circ \lb\restr{\pi_H}{\hat{U}}\rb^{-1} \rmb\in C^\infty(\T^n),
    \]
    where $\eta\in C^\infty(\T^n)$ satisfies $0\le \eta\le 1$, $\eta\equiv1$ near $\Gamma$ and $\eta\equiv 0$ outside a compact subset of $U$.

   To prove the claim, we observe that $\pi_H$ is a locally diffeomorphic isometry, and therefore it suffices to show that $\restr{\pi_H}{\hat{U}}$ is injective. Let us first show that $\restr{\pi_H}{\hat{\Gamma}}$ is injective. Suppose there are $0\le t_1<t_2\le1$ such that
   \[
   \hat{\gamma}(t_1)\ne \hat{\gamma}(t_2) \textup{ but }\pi_H\lb\hat{\gamma}(t_1)\rb=\pi_H\lb\hat{\gamma}(t_2)\rb.
   \]
   Let $q_1$ and $q_2$ be vectors in $\R^n$ such that we can write $\hat{\gamma}(t_1)$ and $\hat{\gamma}(t_2)$ respectively as
   \[
   \hat{\gamma}(t_1)=q_1+H \textup{ and } \hat{\gamma}(t_2)=q_2+H.
   \]
 {Let $z=q_2-q_1\in \R^n$, then }by the assumption that $\hat{\gamma}(t_1)\ne \hat{\gamma}(t_2)$, {we know that $z\ne 0$ and $z+H\cap H = \emptyset$}. Because $\pi_H\lb\hat{\gamma}(t_1)\rb=\pi_H\lb\hat{\gamma}(t_2)\rb$, {we obtain that the vector $z\in \Z^n$} and $\pi_H\circ \hat{\gamma}:[t_1,t_2]\rta \Gamma$ is a closed path that belongs to the homotopy class 
   \[
   z+h \in \Z^n\setminus H
   \]
   for some $h\in H$. This contradicts the assumption that $H=\Z H_\Gamma\supset H_\Gamma$ includes all the homotopy classes in $\Gamma$.

   To prove that $\restr{\pi_H}{\hat{U}}$ is injective for some small neighborhood $\hat{U}$, we argue by contradiction. We assume that there is a constant $c>0$ and for all small $\delta>0$ there are $x_\delta$, $y_\delta$ in $\T_H^n$ having distance to $\hat{\Gamma}$ bounded by $\delta$ such that
   \[
   \dist(x_\delta,y_\delta) \ge c>0
   \]
   independent of $\delta>0$ but $\pi_H(x_\delta)=\pi_H(y_\delta)$. Notice that $\hat{\Gamma}$ is a compact subset in $\T_H^n$ because $\pi_H$ is a local isometry and hence $\hat{\gamma}'(t)=\gamma'(t)$ is bounded and therefore $\hat{\Gamma}$ is bounded and closed in $\T_H^n$. The proof is done by a compactness argument as we send $\delta\rta0$, which will lead to a contradiction to our previous claim that $\restr{\pi_H}{\hat{\Gamma}}$ is injective.

\end{proof}

\subsection{Change of variables}
\label{subsection.wazewskiparam}
Another technical lemma is the change of variables using Lipschitz paths. Let us first recall a standard result for Lipschitz paths. We also refer to Section \ref{subsubsection.lipschitzpathandlength} for more discussions on the preliminary concepts that are required in this subsection.

\begin{lemma}[\cite{ALBERTI201735}*{Remark 3.6}]\label{l.reparametrization}
    Suppose $\Gamma$ is the image of a Lipschitz path $\gamma:[0,1]\rta \Gamma\subset\T^n$, then for any nonnegative Borel function $f$ (or integrable ones) on $\Gamma$ we have
    \[
    \int_0^1 f(\gamma(t)) |\gamma'| dt = \int_\Gamma f(x) m(\gamma,x) d\Ha^1,
    \]
    where $m(\gamma,x):=\#\gamma^{-1}(x)$ is the multiplicity of $\gamma$ at $x$.
\end{lemma}

We have the following technical results.

\begin{lemma}\label{l.computationwithinequalityofQ}
    Suppose $\gamma:[0,1]\rta\Gamma$ is a constant speed Lipschitz path with 
    \[
    m_\gamma:=\sup_{x\in \Gamma}\#\gamma^{-1}(x) <\infty \textup{ and path length }\ell(\gamma)=|\gamma'|,
    \]
    then we have
 \be
   p\cdot Q(\theta) p \ge   \frac{1}{m_\gamma \ell(\gamma)} \inf_{\varphi\in C^\infty(\T^n)}\int_0^1 \lw \frac{d}{dt} \lmb\varphi(\gamma(t)) \rmb+ p\cdot \gamma'(t)\rw^2 a(\gamma(t))d t.
  \ee
\end{lemma}

\begin{proof}
Let $K:=\gamma([0,1])$ be the image. Note that for every $\varphi\in C^\infty(\T^n)$ we have by Lemma \ref{l.reparametrization}
\[
\begin{split}
    \int_\Gamma |\gd \varphi+p|^2 a\, d\Ha^1& \ge  \int_K |\gd \varphi+p|^2 a\, d\Ha^1\\
    &\ge \frac{\ell(\gamma)}{m_\gamma} \int_0^1 |P_{\tau_{\gamma(t)}}(\gd \varphi(\gamma(t))+p)|^2 a(\gamma(t)) dt\\
    &=\frac{1}{\ell(\gamma)m_\gamma}\int_0^1 \lw \frac{d}{dt} \lmb\varphi(\gamma(t)) \rmb+ p\cdot \gamma'(t)\rw^2 a(\gamma(t)) dt,
\end{split}
\]
where $\tau_{\gamma(t)}$ is the 1-D tangent space of $\gamma$ at $t$. The proof is done by taking the infimum over $\varphi$.
\end{proof}

\subsection{Proof of Theorem \ref{t.Qtohomotopy}}
\label{subsection.proofoftheoremqhomotopy}
We require the following technical lemmas for proving the reverse inclusion $H_\Gamma^\perp\supset \ker Q(\theta)$. We refer to Section \ref{subsection.torusandtopology} for the preliminaries of some concepts in this subsection.

\begin{lemma}
    \label{l.selectionoflipschitzpath}
    Let $E\subset \T^n$ be a subset such that $\Ha^1(E)<\infty$. For any nontrivial closed path $\gamma^*:[0,1]\rta E$, there is another constant speed Lipschitz closed path $\gamma:[0,1]\rta E$ such that $\gamma$ belongs to the same homotopy class as $\gamma^*$, the path length $\ell(\gamma)<\infty$ and $\ell(\gamma)\le \ell(\gamma^*)$, and the lift $\Tilde{\gamma}$ of $\gamma$ on $\R^n$ starting at any $\Tilde{y}\in \pi^{-1}(\gamma(0))$ is injective.
\end{lemma}

\begin{proof}
   Denote $x_0=\gamma(0)=\gamma^*(0)$ and take some $\Tilde{y}\in \pi^{-1}(x_0)$. Let $\Tilde{\gamma}^*$ be the unique lift of $\gamma^*$ on $\R^n$ starting at $\Tilde{\gamma}^*(0)=\Tilde{y}$, and denote the homotopy class as $q=\Tilde{\gamma}^*(1)-\Tilde{y}\in \Z^n\setminus\{0\}$. Because $\Tilde{\gamma}^*$ is continuous, its image is compact and connected in $\R^n$. On the other hand, the periodic extension (See Section \ref{subsubsection.periodicextensionofmeasures}) $\restr{\Ha^1}{\pi^{-1}(E)}$ of $\restr{\Ha^1}{E}$ is locally finite on $\R^n$. This implies that 
    $$
    \Ha^1(\Tilde{\gamma}^*([0,1]))=\restr{\Ha^1}{\pi^{-1}(E)}(\Tilde{\gamma}^*([0,1]))<\infty.
    $$ 
    Now we know that $\Tilde{\gamma}^*([0,1])$ is a connected compact set in $\R^n$ that has finite $\Ha^1$ measure. By applying \cite{ALBERTI201735}*{Proposition 3.4}, there is an injective Lipschitz path $\Tilde{\gamma}:[0,1]\rta \Tilde{\gamma}^*([0,1])$ such that $\Tilde{\gamma}(0)=\Tilde{y}$ and $\Tilde{\gamma}(1)=q+\Tilde{y}$. The proof is done by setting $\gamma=\pi\circ \Tilde{\gamma}$. Notice that by \cite{ALBERTI201735}*{Proposition 3.4} and Lemma \ref{l.reparametrization}
    \[
    \ell(\gamma) \le \Ha^1(\Tilde{\gamma}^*([0,1])) \le \ell(\Tilde{\gamma}^*)=\ell(\gamma^*).
    \]
    On the other hand we also have $\ell(\gamma) \le \Ha^1(\Tilde{\gamma}^*([0,1]))<\infty$.
\end{proof}

\begin{lemma}
\label{l.homotopygroup}
   Let $E\subset \T^n$ be a connected closed subset such that $\Ha^1(E)<\infty$. Then for any point $x_0\in E$ we have
   \[
   H_E = \Z H_E = i_{x_0}(\pi_1(E,x_0)),
   \]
   where $i_{x_0}:\pi_1(\T^n,x_0)\rta\Z^n$ is the isomorphism defined in Section \ref{subsubsection.pathliftingproperty}.
\end{lemma}

\begin{proof}
    Let $q=i_{y_0}([\xi]_{y_0})$ be a vector in $H_E$, where $y_0\in E$ and $[\xi]_{y_0}$ is a homotopy class of a closed path $\xi$ in $E$ based at $y_0$, and $i_{y_0}$ is the isomorphism as defined in Section \ref{subsubsection.pathliftingproperty}. It then suffices to show that $q\in i_{x_0}(\pi_1(E,x_0))$. The case $x_0=y_0$ is trivial. In the case $x_0\ne y_0$, as $E$ is connected compact and has finite $\Ha^1$ measure, by \cite{ALBERTI201735}*{Proposition 3.4} there is an injective constant speed Lipschitz path $\eta:[0,1]\rta E$ such that $\eta(0)=x_0$ and $\eta(1)=y_0$. Consider the path composition
    \[
    \hat{\xi}:=\eta^{-1}\xi\eta,
    \]
    which defines a closed path in $E$ based at $x_0$. Note that because $\eta$ is injective, it is not difficult to compute the homotopy class $i_{x_0}([\hat{\xi}]_{x_0}) = i_{y_0}([\xi]_{y_0})=q$. This shows that $q\in i_{x_0}(\pi_1(E,x_0))$.
\end{proof}

\begin{lemma}
    \label{l.selectionofgoodpath}
    Let $E\subset \T^n$ be a connected closed subset such that $\Ha^1(E)<\infty$. Then there is a constant $c=c_{n,E}>0$ such that for each $q\in {H_E\setminus\{0\}}$ we can find a Lipschitz closed path $\gamma_q$ in $\Gamma$ that belongs to $q$, having injective lift in $\R^n$ and
    \[
  1\le   |q| \le \ell (\gamma_q) \le c |q|.
    \]
\end{lemma}

Notice that the key difficulty of this lemma is that the length $\ell(\gamma)$ of a closed path $\gamma$ in $E$ may not be bounded by $\Ha^1(E)$ because of multiplicity.

\begin{proof}
By Lemma \ref{l.homotopygroup}, we can view $H_E$ as the isomorphic image of $\pi_1(E,x_0)$ in $\Z^n$ for some $x_0\in E$. From now on in this proof we do not distinguish $\pi_1(E,x_0)$ and $H_E$. Let $u_1,\dots,u_k$ be a minimal collection of the generators of the group $H_E\le \Z^n$. By Lemma \ref{l.selectionoflipschitzpath} the following quantity is well-defined
    \[
    \lambda_E:=\max_{1\le j\le k}\inf\lma \ell(\gamma_j) \ ; \ \gamma_j \textup{ is a closed path belonging to }u_j\rma<\infty.
    \]
For each $1\le j \le k$ we denote $\gamma_j$ a closed path in $E$ based at $x_0$ that belongs to $u_j$ and
\[
\ell(\gamma_j) \le 2 \inf\lma \ell({\gamma}_j^*) \ ; \ {\gamma}_j^* \textup{ is a closed path belonging to }u_j\rma.
\]

Now for each $q\in H_E$, there are $c_i\in \Z$ such that
\[
q=\sum_{i=1}^k c_i u_i.
\]
Correspondingly we define 
\[
\gamma_q^*:=\prod_{i=1}^k {\gamma}_i^{c_i}
\]
as the path composition based at $x_0$. Notice that $\gamma_q^*$ is a closed path in $\Gamma$ that belongs to $q$. By Lemma \ref{l.selectionoflipschitzpath} we can find a Lipschitz path $\gamma_q:[0,1]\rta \gamma_q^*([0,1])$ that belongs to $q$ and has a unique injective lift $\Tilde{\gamma}_q$ on $\R^n$ starting at some $\Tilde{\gamma}_q(0)=\Tilde{y}\in \pi^{-1}(\gamma_q(0))$. Because the lift $\Tilde{\gamma}_q$ connects $\Tilde{y}$ and $\Tilde{y}+q$, and $\gamma_q'=\Tilde{\gamma}_q'$, we have
\[
\ell(\gamma_q)=\ell(\Tilde{\gamma}_q)\ge |q|.
\]
On the other hand, by the construction of $\gamma_q$
\be\label{eq.intermediateupperboundonlength}
\ell(\gamma_q) \le \ell(\gamma_q^*) = \sum_{i=1}^k |c_i| \ell({\gamma}_i) \le 2 \lambda_E \sum_{i=1}^k|c_i|.
\ee
It then suffices to show that $\sum_{i=1}^k|c_i|$ is uniformly bounded by $|q|$. Indeed, the set of vectors $u_i\in \Z^n$ is $\Z$-linearly independent and therefore they are also $\R$-linearly independent, {see Lemma \ref{l.ZindRind}}. 
This implies that $\sum_{i=1}^k|c_i|$ defines another norm on $\R H_E$. As we live in a finite dimensional space there is always a constant $c=c_{n,E}>0$ such that
\[
c^{-1}|q| \le \sum_{i=1}^k|c_i| \le c|q|.
\]
By combining \eqref{eq.intermediateupperboundonlength} we finish the proof.

\end{proof}

\begin{proof}[Proof of Theorem \ref{t.Qtohomotopy}]
By Lemma \ref{l.decompositionof1dmedium} and Lemma \ref{l.firstinclusion}, we know that it suffices to prove the following inclusion
\[
\ker Q(\theta) \subset H_\Gamma^{\perp},
\]
where $\theta$ takes the form \eqref{eq.simpleformofnetwork} and $\Gamma=\spt{\theta}$ is a closed connected subset of $\T^n$ such that $0<\Ha^1(\Gamma)<\infty$. 

To prove this inclusion we claim that there is a constant $c>0$ such that
\[
\inf_{q\in H_\Gamma\setminus\{0\}} \frac{q\cdot Q(\theta)q}{|q|^2} \ge c>0.
\]
This claim proves the inclusion $\ker Q(\theta) \subset H_\Gamma^{\perp}$ because if $p\in \ker Q(\theta)$ then we can write $p=p_1+p_2$ with $p_1\in \R H_\Gamma$ and $p_2 \perp H_\Gamma$. By Lemma \ref{l.firstinclusion}, we know that
\[
0=Q(\theta)p=Q(\theta)p_1.
\]
This implies that $p_1\in \ker Q(\theta)$. {If $p_1\ne0$, then by Lemma \ref{l.homotopygroup}, $ H_\Gamma= \Z H_\Gamma$ is the $\Z$-span of a finite collection of vectors, and therefore the set $\{q/|q|\ ; \ q\in H_\Gamma\setminus\{0\}\}$ is dense in $\pt B_1(0)\cap \R H_\Gamma$ by Lemma \ref{l.densityofnormalizedintegerpoints}, which implies that}
\[
\frac{p_1\cdot Q(\theta)p_1}{|p_1|^2} \ge \inf_{q\in H_\Gamma\setminus\{0\}} \frac{q\cdot Q(\theta)q}{|q|^2} \ge c>0,
\]
which is a contradiction to the fact that $p_1\in \ker Q(\theta)$.

To prove the claim we observe that for $q\in H_\Gamma\setminus\{0\}$, by Lemma \ref{l.computationwithinequalityofQ} and Lemma \ref{l.selectionofgoodpath} we can find a constant speed closed Lipschitz path $\gamma_q$ in $\Gamma$ that belongs to $q$ and
\[
\begin{split}
   q\cdot Q(\theta)q & \ge \inf_{\varphi\in C^\infty(\T^n)} \frac{1}{m_{\gamma_q} \ell(\gamma_q)} \int_0^1 \lw \frac{d}{dt} \lmb\varphi(\gamma_q(t)) \rmb+ q\cdot \gamma_q'(t)\rw^2 a(\gamma_q(t))d t\\
    &\ge \inf_{\varphi\in C^\infty(\T^n)} \frac{1}{\Lambda m_{\gamma_q} \ell(\gamma_q)} \lb \int_0^1 \frac{d}{dt} \lmb\varphi(\gamma_q(t)) \rmb+ q\cdot \gamma_q'(t) dt \rb^2\\
    &= \inf_{\varphi\in C^\infty(\T^n)} \frac{1}{\Lambda m_{\gamma_q} \ell(\gamma_q)} \lb q\cdot\int_0^1   \gamma_q'(t) dt \rb^2\\
    &=\frac{|q|^4}{\Lambda m_{\gamma_q} \ell(\gamma_q)},
\end{split}
\]
where $m_{\gamma_q}$ is the maximal multiplicity of $\gamma_q$. Notice that the lift $\Tilde{\gamma}_q$ of $\gamma_q$ starting at some $\Tilde{\gamma}_q(0)=\Tilde{y}\in \pi^{-1}(\gamma_q(0))$ is an injective map, and $\gamma_q=\pi\circ \Tilde{\gamma}_q$. This shows that
\be\label{eq.boundmaxmgammq}
m_{\gamma_q} \le \max_{x\in [0,1)^n} \# \lma \Tilde{\gamma}_q([0,1]) \cap (\Z^n +x)\rma \le \lceil \ell (\Tilde{\gamma}_q)\rceil= \lceil \ell ({\gamma}_q)\rceil\le \ell(\gamma_q)+1 \le 2\ell(\gamma_q).
\ee
Combining Lemma \ref{l.selectionofgoodpath}, we can finish the proof by observing
\be\label{eq.fourthpowerofq}
q\cdot Q(\theta)q\ge \frac{|q|^4}{\Lambda m_{\gamma_p} \ell(\gamma_q)}  \ge \frac{|q|^2}{\Lambda c_{n,\Gamma}}
\ee
for some $c_{n,\Gamma}>0$ depending only on the dimension $n$ and $\Gamma=\spt{\theta}$.
    
\end{proof}
{\begin{remark}
    \label{r.fourthpower}
One might wonder why there is a fourth power in $|q|$ in the intermediate step of \eqref{eq.fourthpowerofq}. The main reason is that when we estimate the integral in Lemma \ref{l.computationwithinequalityofQ} we allow the closed paths $\gamma_q$ to have multiplicities, especially when $q=\lambda \Tilde{q}$ for some $\lambda\in \Z_+$ and $\Tilde{q}\in H_\Gamma$ and $\gamma_q=\gamma_{\Tilde{q}}^\lambda$ is the path composition of $\gamma_{\Tilde{q}}$ with itself for $\lambda$ times. In this case, $\ell(\gamma_q)$ and $m_{\gamma_q}$ are not constants, but linear functions of $|q|$. Generally, the efforts in Lemma \ref{l.selectionofgoodpath} and \eqref{eq.boundmaxmgammq} are just made to bound the effects of the length and multiplicity of $\gamma_q$ when we use the formula in Lemma \ref{l.computationwithinequalityofQ}. It turns out that when $\Gamma$ is connected, closed and $\Ha^1(\Gamma)<\infty$, one can find closed paths $\gamma_q$ in $\Gamma$ so that the growth of both lengths $\ell(\gamma_q)$ and the multiplicities $m_{\gamma_q}$ are indeed linear in $|q|$, which then leads to the lower bound \eqref{eq.fourthpowerofq}.
\end{remark}}

\subsection{Loopiness and reticulation in dimension \texorpdfstring{$n=2$}{das}}

In this subsection we prove Theorem \ref{t.characterizationoflowerboundn=2}. Let us begin with the following characterization of loopy sets in dimension $n=2$.

\begin{lemma}
    \label{l.characterizeloopysetn=2}
    Let $E\subset \T^2$ be a {closed} subset such that $\Ha^1(E)<\infty$ and denote $\Tilde{E}=\pi^{-1}(E)$ be the periodic extension. Denote $e_1=(1,0)$ and $e_2=(0,1)$ as the standard orthonormal basis for $\R^2$. The following statements are equivalent:
    \begin{enumerate}[label=(\alph*)]
        \item\label{condition.(a)loopy} $E$ is loopy.
        \item\label{condition.(b)loopy} $E$ is reticulate.
        \item\label{condition.(c)loopy} There is a reticulate compact connected component $E^*\subset E$ such that $E\setminus E^*$ is trivial.
        \item\label{condition.(d)loopy} There is a point $x\in \Tilde{E}$ such that for each $i=1,2$ there is a Lipschitz path $\gamma_i:[0,1]\rta \Tilde{E}$ with $\gamma_i(0)=x$ and $\gamma_i(1)=x+e_i$.
    \end{enumerate}
\end{lemma}

To prove this lemma we require the following technical lemma.

\begin{lemma}\label{l.intersectionofnonparallelpaths}
    Let $\gamma_1$ and $\gamma_2$ be two closed paths in $\T^2$ and their homotopy classes are not parallel to each other, then the images of $\gamma_1$ and $\gamma_2$ must intersect at some point.
\end{lemma}

By the same proof of this lemma we can show the following corollary.

\begin{corollary}
    \label{cor.intersectionoflinelikesets}
    Suppose $L_i(t):=y_i +t q_i\in \R^2$ are two straight lines with $q_1$, $q_2$ linearly independent, then for any continuous functions $\Gamma_i:\R\rta\R^2$ such that
    \[
    \sup_{t\in \R, \, i=1,2} \dist(\Gamma_i(t),L_i(t)) <\infty
    \]
    we have the images of $\Gamma_1$ and $\Gamma_2$ must have nontrivial intersection.
\end{corollary}

\begin{proof}[Proof of Lemma \ref{l.intersectionofnonparallelpaths}]
We denote the nonzero vectors $(n_1,m_1)$ and $(n_2,m_2)$ in $\Z^2$ to be the homotopy classes of $\gamma_1$ and $\gamma_2$ respectively. By the nonparallel assumption we know that $|n_1m_2-n_2m_1|>0$. Let $\widetilde{\gamma}_1$ and $\widetilde{\gamma}_2$ be the lift of $\gamma_1$ and $\gamma_2$ respectively on $\R^2$ such that the starting points are $y_i\in \pi^{-1}(\gamma_i(0))$ and
\[
\widetilde{\gamma}_i(1)=y_i+(n_i,m_i),
\]
for all $i=1,2$.

Now we construct $\Gamma_i:\R\to\R^2$ by defining
\[
\Gamma_i(t)=\widetilde{\gamma}_i(t-\lfloor t\rfloor) + \lfloor t\rfloor (n_i,m_i)
\]
for $i=1,2$ and all $t\in \R$. Notice that $\Gamma_i$'s are continuous on $\R$ because of the definition of $\widetilde{\gamma}_i$'s. 

Notice that $\gamma_1,\gamma_2$ have intersection points if and only if the images of $\Gamma_i$'s have intersection points. It then suffices to consider the intersection of $\Gamma_1$ and $\Gamma_2$. We know that because $\Gamma_i$'s are continuous and $\Gamma_i(t+l)-\Gamma_i(t)=l(n_i,m_i)$ for all $t\in\R$, $l\in\Z$ and $i=1,2$, there is a constant $C>0$ such that 
\[
\dist(\Gamma_i(t),\left(y_i+t(n_i,m_i)\right)) \le C
\]
for all $i=1,2$ and $t\in \R$. Let $L_i(t) = y_i + t(n_i,m_i)$ be the linearized paths. Define the matrix $A = \begin{pmatrix}n_1 & -n_2 \\ m_1 & -m_2\end{pmatrix}$ with $\det A = n_2m_1 - n_1m_2 \neq 0$. Consider the continuous map $\Phi: \R^2 \to \R^2$ given by:
\[
\Phi(t,s) = A^{-1}\left(y_2 - y_1 + (\Gamma_2(s) - L_2(s)) - (\Gamma_1(t) - L_1(t))\right).
\]
The boundedness $\dist(\Gamma_i(t),L_i(t)) \le C$ implies $\Phi$ maps some closed ball $B_R \subset \R^2$ to itself. By Brouwer's fixed-point theorem, there exists $(t_0,s_0) \in B_R$ with $\Phi(t_0,s_0) = (t_0,s_0)$, yielding:
\[
\Gamma_1(t_0) = \Gamma_2(s_0).
\]
Projecting via $\pi$ gives $\gamma_1(t_0) = \gamma_2(s_0)$ in $\T^2$, completing the proof.

\end{proof}

\begin{proof}[Proof of Lemma \ref{l.characterizeloopysetn=2}]
By Lemma \ref{l.intersectionofnonparallelpaths} we know that \ref{condition.(d)loopy} $\implies$ \ref{condition.(c)loopy} $\implies$ \ref{condition.(b)loopy} $\implies$ \ref{condition.(a)loopy}. By the same reason and also Lemma \ref{l.homotopygroup} we deduce that \ref{condition.(a)loopy} $\implies$ \ref{condition.(c)loopy}, and therefore \ref{condition.(a)loopy}, \ref{condition.(b)loopy} and \ref{condition.(c)loopy} are equivalent. It then suffice to show that \ref{condition.(d)loopy} can be implied by other three conditions. We begin with \ref{condition.(c)loopy} and assume there are two closed paths $\gamma_1$ and $\gamma_2$ based at the same point $x_0\in E^*$. It suffices to show that $E^*$ contains homotopy classes of the form $(1,0)$ and $(0,1)$.

Denote $(a,b)$ and $(c,d)$ as the nonzero homotopy classes of $\gamma_1$ and $\gamma_2$ respectively, then by the assumption \ref{condition.(c)} we can assume that $g:=ad-bc\ne0$. Notice that by taking integer combinations
 \[
    (-d) \cdot (a, b) + b \cdot (c, d) = (bc - ad, 0), \quad (-c) \cdot (a, b) + a \cdot (c, d) = (0, ad - bc),
\]
we obtain classes of the form $q_1=(-g,0)$ and $q_2=(0,g)$. To show the existence of classes of the form $(1,0)$ (the class $(0,1)$ can be deduced symmetrically) we denote $\gamma_i^*$ as some closed paths based at the common $x_0$ that belong to $q_i$ for $i=1,2$. Denote $\Tilde{\gamma}_i^*$ as the lift of $\gamma_i^*$ based at some common $y_0\in \pi^{-1}(x_0)$ and define
\[
\Gamma_i(t)=\Tilde{\gamma}_i^*(t-\lfloor t\rfloor) + \lfloor t\rfloor q_i.
\]
The proof is done by observing that $\Gamma_1$ must intersect $\Gamma_2+e_1$ according to Corollary \ref{cor.intersectionoflinelikesets}.
    
\end{proof}

\begin{proof}[Proof of Theorem \ref{t.characterizationoflowerboundn=2}]
    The proof is done by combining Theorem \ref{t.Qtohomotopy} and Lemma \ref{l.characterizeloopysetn=2}.
\end{proof}

\appendix

\section{A formal derivation of resilience from the expected total dissipation}\label{appendix.formalderivation}

In this section we derive the notion of resilience introduced in Theorem \ref{t.forma1}, that is, the positivity of the effective tensor $Q(\theta)$, from the minimization of expected total dissipation in the hydraulic system of a leaf under random fluctuations \cites{Corson2010,Katifori2010}.

Let us model a small piece of leaf by a smooth bounded planar domain $\Da\subset\R^2$ along with a positive definite matrix field $\sigma=L+Q :\Da \rta \R^{2\times 2}$ that represents the local conductance of the leaf. Here $L$ represents the conductance of the lower order veins, which are supported near a tree-like network. The field $Q\ge Q_0$ is a constant positive definite matrix, regarded as the effective tensor of $\theta +  Q_0\,d\Lme^2$ where $\theta$ is a network-like medium as in Theorem \ref{t.forma1} and $Q_0$ is a small positive definite matrix that represents the background medium. That is, we are at a scale where the higher order veins are considered well-mixed with the background materials. We are concerned about the positive definiteness of $Q-Q_0\approx Q(\theta)$ as the leaf grows.

By Darcy's law, the velocity field $j$ satisfies the following relation with the pressure function $\phi$:
\be
\label{eq.darcy}
j=-\sigma\gd \phi.
\ee
Suppose the source distribution is $m_i$ and the sink distribution is $m_o$, where $m_i$ and $m_o$ are nonnegative finite measures on $\pt\Da$ and $\Da$ respectively. The mass conservation indicates that $m_o(\Da)=m_i(\pt\Da)$. Combining the Darcy's law \eqref{eq.darcy} with the Gauss law, one obtains the Neumann problem
\be
\bca
  \gd\cdot \lb {\sigma} \gd \phi\rb(x)=m_o(x) &x\in \Da\\
 \sigma\gd \phi\cdot\Vec{n}(x)=m_i(x) &x\in\partial \Da,
\eca\label{eq.gaussdarcy}
\ee
where $\Vec{n}(x)$ is the unit outer normal of $\pt\Da$ at $x$. Under random fluctuations in $m_i$ and $m_o$, the total dissipation in this scenario takes the form
\be
P(\sigma) = \Ept\lmb\int_{\Da} \gd\phi\cdot\sigma \gd \phi \rmb,
\ee
where the pressure function $\phi$ satisfies \eqref{eq.gaussdarcy} and the expectation is taken with respect to $m_i$ and $m_o$. As the Gauss law explicitly controls the velocity field $j$, the total dissipation $P(\sigma)$ represents the effective resistance of the construction $\sigma$ under a random choice of the source-sink distributions $m_i$ and $m_o$. Therefore, for a leaf to maintain its function, it is natural to minimize the effective resistance, or equivalently the total dissipation $P(\sigma)$.

We are especially interested in the effects of the higher order veins, which leads us to compute the gradient of $P(\sigma)=P(L+Q)$ as a function of $Q$. Indeed, let $B$ be any symmetric matrix, we have
\be
\begin{split}
    \restr{\frac{d}{d t}}{t=0} P\lb \sigma+tB\rb &=-  \Ept\lmb\int_\Da \gd \phi\cdot B\gd\phi \rmb\\
    &=-\trace{B K},
\end{split}
\ee
where $K:=\Ept\lmb\int_{\Da} \gd\phi\otimes \gd\phi\rmb$ and $\phi$ satisfies the equation \eqref{eq.gaussdarcy}. This implies that the gradient flow of $Q$ takes the form
\[
\frac{d}{dt}Q=K=\Ept\lmb\int_{\Da} \gd\phi\otimes \gd\phi\rmb.
\]
Note that if $m_i$ and $m_o$ are random, then the tensor $K$ is generically positive definite, which means that in this scenario the effective tensor $Q_t-Q_0$ of the higher order veins regardless of the background has to be positive definite.

\section{Periodic homogenization of media}
\label{appendix.periodichomog}
For a positive semi-definite matrix-valued Radon measure $d\xi=\beta\,d\norm{\xi}$ on a bounded open domain $U\Subset\R^n$, one can consider the following mesoscopic effective tensor
\[
p\cdot Q(\xi ; U)p := \inf_{\phi\in C_0^\infty(U)} \dashint_U (\gd \phi+p)\cdot \beta (\gd \phi+p) d\norm{\xi}=\inf_{u\in p\cdot x + C_0^\infty(U)} \dashint_U \gd u\cdot \beta \gd u d\norm{\xi},
\]
where $\dashint_U:=\frac{1}{|U|}\int_U$ with $|U|$ the Lebesgue measure of $U$.

In this section we show the following periodic homogenization result. 
\begin{lemma}
    \label{l.periodichomog}
Let $\theta^*$ be the periodic extension (See Section \ref{subsubsection.periodicextensionofmeasures}) of a medium $\theta$ on $\T^n$, then 
\[
\lim_{R\rta\infty} Q(\theta^*; (-R,R)^n) = Q(\theta).
\]
\end{lemma}
Such homogenization results are known for the case when $\theta$ is absolutely continuous with respect to the Lebesgue measure, see \cite{Jikov1994}*{Chapter 1} and an ergodic version in \cite{Armstrong2019}*{Chapter 1}. A singular version of homogenization can be found in \cite{bouchitte2001}. Here we present a proof for media including general anisotropy. The proof does not use either the BBS tangent space theory as in \cite{bouchitte2001} or any Sobolev estimates on the corrector equation.

Let us start with an auxiliary lemma. First we denote 
$$
N_R(x):=\# (-R,R)^n \cap \pi^{-1}(x),
$$ 
where $\pi:\R^n\rta\T^n$ is the standard projection. Note that
\[
N_{R,-}(x):=\inf_{x\in \T^n} N_R(x) \textup{ and }N_{R,+}(x):=\sup_{x\in \T^n}N_R(x)
\]
satisfy
\[
\lim_{R\rta\infty}\frac{N_{R,\pm}}{(2R)^n} = 1.
\]

\begin{lemma}
    \label{l.auxillary1forperiodichomog}
    For every $R>0$ one has
    \[
    Q(\theta) \le \liminf_{R\rta\infty} Q(\theta^* ; (-R,R)^n).
    \]
\end{lemma}

\begin{proof}
    For every $\phi\in C_0^\infty((-R,R)^n))$ we define 
    \[
    v_\phi(x):=\frac{1}{(2R)^n}\sum_{y\in \pi^{-1}(x)} \phi(y).
    \]
    Notice that 
    \[
    \gd v_\phi(x) = \frac{1}{(2R)^n}\sum_{y\in \pi^{-1}(x)} \gd \phi(y).
    \]
    In particular by Cauchy-Schwartz inequality we have
    \be\label{eq.tech.b2}
    \gd v_\phi(x)\cdot \sigma(x)\gd v_\phi(x)\le \frac{N_R(x)}{(2R)^{2n}}\sum_{y\in \pi^{-1}(x)} \gd \phi(y)\cdot \sigma(x) \gd \phi (y).
    \ee
    By the definition of periodic extension in Section \ref{subsubsection.periodicextensionofmeasures}, the periodic extension of the medium $d\theta(x)=\sigma(x)\,d\norm{\theta}(x)$ takes the form
    $$
    d\theta^*(y)=\sigma(\pi(y)) \, d w(y),
    $$
where the Radon measure $w$ is the periodic extension of $\norm{\theta}$. This implies that
    \[
    \begin{split}
        \dashint_{(-R,R)^n} |\gd \phi + p|_{\sigma(\pi(y))}^2 dw(y)&= \int_{\T^n} \frac{1}{(2R)^n}\sum_{y\in \pi^{-1}(x)} |\gd \phi(y)|_{\sigma(x)}^2 d\norm{\theta}(x) + 2 \int_{\T^n} \gd v_\phi \cdot \sigma p\,  d\norm{\theta} \\
        &\quad+ \dashint_{(-R,R)^n}  p\cdot \sigma(\pi(y))p\,  dw(y)\\
        &\ge \int_{\T^n} \frac{1}{c_R(x)} |\gd v_\phi|_\sigma^2 d\norm{\theta}(x) + 2 \int_{\T^n} \gd v_\phi(x) \cdot \sigma(x) p\,  d\norm{\theta}(x) \\
        &\quad+ \int_{\T^n} c_R(x) |p|_{\sigma}^2\,  d\norm{\theta}(x)+ o_R(1)\int_{\T^n} |p|_{\sigma}^2\,  d\norm{\theta}(x)\\
        &=\int_{\T^n} \frac{1}{c_R(x)} \lw  \gd v_\phi + c_R(x)p\rw_{\sigma}^2 d\norm{\theta}+ o_R(1),
    \end{split}
    \]
    where $|q|_\sigma^2:=q\cdot \sigma q$, $c_R(x):= \frac{N_R(x)}{(2R)^n}$ and we have used \eqref{eq.tech.b2}. By using triangle inequality
    \[
    \dashint_{(-R,R)^n} |\gd \phi + p|_{\sigma(\pi(y))}^2 dw(y) \ge \int_{\T^n} \frac{1}{c_R(x)} \lb\lw  \gd v_\phi + p\rw_{\sigma} -|1-c_R(x)||p|_\sigma \rb^2 d\norm{\theta}+o_R(1).
    \]
    This shows that for any $\ep>0$
    \[
     \dashint_{(-R,R)^n} |\gd \phi + p|_{\sigma(\pi(y))}^2 dw(y) \ge \int_{\T^n} \frac{1-\ep}{c_R(x)} \lw  \gd v_\phi + p\rw_{\sigma}^2 +\lb1-\frac{1}{\ep}\rb\frac{|1-c_R(x)|^2}{c_R(x)}|p|_\sigma^2 d\norm{\theta} +o_R(1).
    \]
    Therefore, for any $\ep>0$ and $R>0$
    \[
   p\cdot Q(\theta^*;(-R,R)^n)p \ge \frac{(2R)^n(1-\ep)}{N_{R,+}}p\cdot Q(\theta)p +\lb1-\frac{1}{\ep}\rb \int_{\T^n} \frac{|1-c_R(x)|^2}{c_R(x)}|p|_\sigma^2 d\norm{\theta} +o_R(1)
    \]
    Sending $R\rta\infty$ we have $\displaystyle\lim_{R\rta\infty} c_R(x)=1$ uniformly in $x\in \T^n$ and hence for all $\ep>0$
    \[
    \liminf_{R\rta\infty}\, p\cdot Q(\theta^*;(-R,R)^n)p  \ge (1-\ep) p\cdot Q(\theta)p.
    \]
    This finishes the proof by sending $\ep\rta0$.
\end{proof}

\begin{proof}[Proof of Lemma \ref{l.periodichomog}]
    By Lemma \ref{l.auxillary1forperiodichomog}, it suffices to show that 
    \[
    \limsup_{R\rta\infty} Q(\theta^*;(-R,R)^n) \le Q(\theta).
    \]
    To this end, we denote $\eta_R\in C_0^\infty((-R,R)^n)$ a cut-off function satisfying
    \begin{itemize}
        \item $0\le\eta_R\le 1$ and $|\gd \eta_R|\le 2$
        \item $\eta_R \equiv 1$ on $[-\lfloor R\rfloor+1,\lfloor R\rfloor-1]^n$ and $\eta_R \equiv0$ near the boundary $\pt (-R,R)^n$.
    \end{itemize}
    By using this function we observe that any $\psi\in C^\infty(\T^n)$ can be extended to $C_0^\infty((-R,R)^n)$ through
    \[
    f_\psi^R(y):=\eta_R(y) \psi(\pi(y)).
    \]
    Notice that 
    \[
    \begin{split}
         \dashint_{(-R,R)^n} |\gd f_\psi^R+p|_{\sigma(\pi(y))}^2 dw(y) &= \frac{(2\lfloor R\rfloor-2)^n}{(2R)^n} \int_{\T^n} |\gd \psi + p|_\sigma^2 d\norm{\theta} \\
         &\quad+ \frac{1}{(2R)^n} \int_{(-R,R)^n\setminus [-\lfloor R\rfloor+1,\lfloor R\rfloor-1]^n}  |\gd f_\psi^R+p|_{\sigma(\pi(y))}^2 dw(y).
    \end{split}
    \]
Observe that for any fixed $\psi$, the second term is $o_R(1)$, and therefore
\[
\frac{(2\lfloor R\rfloor-2)^n}{(2R)^n} \int_{\T^n} |\gd \psi + p|_\sigma^2 d\norm{\theta} +o_R(1) \ge p\cdot Q(\theta^*; (-R,R)^n) p 
\]
Sending $R\rta\infty$ we obtain
\[
\int_{\T^n} |\gd \psi + p|_\sigma^2 d\norm{\theta} \ge \limsup_{R\rta\infty}\, p\cdot Q(\theta^*; (-R,R)^n) p, 
\]
which finishes the proof by taking the infimum over all $\psi\in C^\infty(\T^n)$.

\end{proof}

\section{Some technical facts}

In this section we present some technical facts that we use in the proofs but are not in the scope of the main topics of this article.

\begin{lemma}
    \label{l.densityofnormalizedintegerpoints} 
    Suppose $H$ is a discrete additive subgroup of $\R^n$. Then the set of points $$\left\{\frac{q}{|q|} \,;\, q\in H \setminus \{0\}\right\}$$ is dense in the unit sphere of $\textup{Span}_\R H$.
\end{lemma}

\begin{proof}

Let $G = \textup{Span}_\R H$. 
Denote by $S_G = \{ x \in G : |x| = 1 \}$ the unit sphere in $G$. Since $H$ is discrete, it is a lattice in $G$. 
Thus there exists $r = \textup{dim}_\R \ G$ and linearly independent vectors $v_1, \dots, v_r \in G$ such that 
$H = \Z v_1 + \cdots + \Z v_r$. 
The set $\{v_1, \dots, v_r\}$ is an $\R$-basis for $G$.
Define a linear isomorphism $T : G \to \R^r$ by $T(v_i) = e_i$ (the standard basis vectors). 
Then $T(H) = \Z^r$. 
The Euclidean inner product on $\R^n$ induces an inner product on $\R^r$ via
\[
\langle x, y \rangle_T = \langle T^{-1}x, T^{-1}y \rangle \textup{ and } \norm{x}_T = \sqrt{\langle x, x \rangle_T}.
\]
Then $T$ is an isometry: $|x| = \norm{T(x)}_T$ for all $x \in G$.

For $q \in H \setminus \{0\}$, set $z = T(q) \in \Z^r \setminus \{0\}$. Then
\[
T\left( \frac{q}{|q|} \right) = \frac{z}{\norm{z}_T}.
\]
Hence it suffices to prove that $\{ z / \norm{z}_T : z \in \Z^r \setminus \{0\} \}$ is dense in the 
$T$-unit sphere $$S_T = T(S_G) =\{ x \in \R^r : \norm{x}_T = 1 \}.$$

Let $u \in S_T$ and $\varepsilon > 0$. By Dirichlet's simultaneous approximation theorem (see for example \cite{schmidt1991diophantine}*{Theorem 1B}), for any integer $N > 0$ there exist $q_N \in \Z$ with $1 \le q_N \le N$ and $p_N \in \Z^r$ such that
\[
\norm{q_N u - p_N}_2 \le \frac{\sqrt{r}}{N^{1/r}}.
\]
Then because all norms on $\R^r$ are equivalent
\[
\norm{q_N u - p_N}_T \le \frac{C\sqrt{r}}{N^{1/r}},
\]
for some constant $C>0$. Choose $N$ large enough so that $C\sqrt{r}/N^{1/r} < \varepsilon$. 
For such $N$, we have $p_N \neq 0$ (otherwise $\norm{q_N u}_T = q_N \ge 1$ contradicts the inequality for large $N$). 

Now because $p_N/q_N \to u$ in $\norm{\cdot}_T$ and $\norm{p_N/q_N}_T \to \norm{u}_T = 1$, we obtain 
\[
\norm{ \frac{p_N}{\norm{p_N}_T} - u }_T 
= \norm{ \frac{p_N/q_N}{\norm{p_N/q_N}_T} - \frac{u}{\norm{u}_T} }_T \to 0
\]
as $N \to \infty$. Thus $u$ is approximated by elements of $\{ z/\norm{z}_T : z \in \Z^r \setminus \{0\} \}$. By the isometry $T$, the original set $\{ q/|q| : q \in H \setminus \{0\} \}$ is dense in $S_G$.

\end{proof}

\begin{lemma}
    \label{l.ZindRind}
    Suppose $V=\{v_1,\dots, v_k\}$ is a sequence of $\Z$-linearly independent vectors in $\Z^n$, then $V$ is also $\R$-linearly independent.
\end{lemma}

\begin{proof}
We argue by contradiction and assume that there are real numbers $c_1,\dots, c_k$, not all zeros, such that
\[
c_1v_1 +\cdots + c_k v_k =0.
\]
By writing $c=(c_j)$ and $M=(M_{ij})$, where $M_{ij}=(v_j)_i\in \Z$ the $i$-th component of the vector $v_j$ for $j=1,\dots,k$ and $i=1,\dots, n$, we obtain the linear equation $M c =0$. 

As the coefficients of $M$ are integers and $0\ne c\in \ker M$, by Gaussian elimination one can always find a solution $\Tilde{c}\in \mathbb{Q}^k\setminus\{0\}$ so that $M\Tilde{c}=0$. We can define
\[
c^*=\Tilde{c}\cdot d\in \Z^k
\]
by choosing $d\in\Z$ that divides all the denominators of the components of $\Tilde{c}$. This implies that $Mc^*=0$ for some $c^*\in \Z^k\setminus\{0\}$, and hence $v_1,\dots,v_k$ are $\Z$-linearly dependent, contradicting the assumption.

\end{proof}

\bibliographystyle{plainnat}
\bibliography{ref}

\end{document}